\documentclass{article}       
\usepackage{geometry}
\usepackage{fancyhdr}
\usepackage{graphicx}
\graphicspath{{images/}}

\usepackage{amsmath,amssymb,amsthm}
\usepackage{enumerate}

\usepackage{hyperref}

\usepackage{color}
\newcommand{\colblue}[1]{{#1}}

\newcommand{\Lcal}{\mathcal{L}}
\newcommand{\Gcal}{\mathcal{G}}

\newcommand{\diverg}{\mathrm{div}}

\newcommand{\N}{\mathbb{N}}
\newcommand{\R}{\mathbb{R}}
\newcommand{\Rb}{\bar{\mathbb{R}}}

\newcommand{\prox}{\mathrm{prox}}
\newcommand{\dom}{\mathrm{dom} \hspace{1pt}}

\newcommand{\proj}{\mathrm{proj}}
\newcommand{\diam}{\mathrm{diam}}

\newcommand{\X}{\mathcal{X}}
\newcommand{\Y}{\mathcal{Y}}
\newcommand{\Z}{\mathcal{Z}}

\newcommand{\xt}{\tilde{x}}
\newcommand{\yt}{\tilde{y}}

\newcommand{\xh}{\hat{x}}
\newcommand{\yh}{\hat{y}}
\newcommand{\zh}{\hat{z}}

\newcommand{\yb}{\bar{y}}
\newcommand{\xb}{\bar{x}}

\newcommand{\xs}{x^\star}
\newcommand{\ys}{y^\star}

\newcommand{\xc}{\check{x}}
\newcommand{\yc}{\check{y}}

\newcommand{\e}{\varepsilon}

\newcommand{\dx}{\mathrm{d}x}

\newcommand{\Bigo}[1]{O\left(#1\right)}

\newtheorem{theorem}{Theorem}[section]
\newtheorem{definition}[theorem]{Definition}
\newtheorem{proposition}[theorem]{Proposition}
\newtheorem{lemma}[theorem]{Lemma}
\newtheorem{corollary}[theorem]{Corollary}
\newtheorem{remark}[theorem]{Remark}

\numberwithin{equation}{section}

\begin{document}

\title{Inexact First-Order Primal-Dual Algorithms}
\author{Julian Rasch\thanks{Applied Mathematics M\"unster: Institute for Analysis and Computational Mathematics, 
Westf\"alische Wilhelms-Universit\"at (WWU) M\"unster, Germany (julian.rasch@wwu.de).},
Antonin Chambolle\thanks{CMAP, \'Ecole Polytechnique, CNRS, France (antonin.chambolle@cmap.polytechnique.fr).}}
\maketitle

\begin{abstract}
We investigate the convergence of a recently popular class of first-order primal-dual algorithms for saddle point problems under the presence of errors in the proximal maps and gradients.
We study several types of errors and show that, provided a sufficient decay of these errors, the same convergence rates as for the error-free algorithm can be established. 
More precisely, we prove the (optimal) $\Bigo{1/N}$ convergence to a saddle point in finite dimensions for the class of non-smooth problems considered in this paper, and prove a $\Bigo{1/N^2}$ or even linear $\Bigo{\theta^N}$ convergence rate if either the primal or dual objective respectively both are strongly convex.
Moreover we show that also under a slower decay of errors we can establish rates, however slower and directly depending on the decay of the errors.
We demonstrate the performance and practical use of the algorithms on the example of nested algorithms and show how they can be used to split the global objective more efficiently. 
\end{abstract}


\section{Introduction}\label{sec:intro}
The numerical solution of nonsmooth optimization problems and the acceleration of their convergence has been regarded a fundamental issue in the past ten to twenty years.
This is mainly due to the development of image reconstruction and processing, data science and machine learning which require to solve large and highly nonlinear minimization problems.
Two of the most popular approaches are forward-backward splittings \cite{Lions1979,Combettes2005,Combettes2011}, in particular the FISTA method \cite{Beck2009,Beck2009b}, and first-order primal-dual methods, first introduced in \cite{Pock2009,Esser2010} and further studied in \cite{Chambolle2011,Chambolle2016a}. 
The common thread of all these methods is that they split the global objective into many elementary bricks which, individually, may be ``easy'' to optimize. 

In their original version, all the above mentioned approaches require that the mathematical operations necessary in every step of the respective algorithms can be executed without errors.
However, one might stumble over situations in which these operations can only be performed up to a certain precision, e.g. due to an erroneous computation of a gradient or due to the application of a proximal operator lacking a closed-form solution.
Examples where this problem arises are TV-type regularized inverse problems \cite{Beck2009b,Barbero2011,Fadili2011,Sawatzky2013,Ehrhardt2016} or low-rank minimization and matrix completion \cite{Cai2010,Ma2011}.
To address this issue, there has been a lot of work investigating the convergence of proximal point methods \cite{Rockafellar1976,Lemaire1992,Gueler1992,Cominetti1997,Salzo2012,He2012}, where the latter two also prove rates, and proximal forward-backward splittings \cite{Combettes2005} under the presence of errors. 
The objectives of these publications reach from general convergence proofs \cite{Luo1993,Patriksson1997,Combettes2004,Zaslavski2010,Friedlander2012} and convergence up to some accuracy level \cite{Nedic2001,Aspremont2008,Devolder2014} to convergence rates in dependence of the errors \cite{Schmidt2011,Villa2013,Aujol2015} also for the accelerated versions including the FISTA method.
The recent paper \cite{Bonettini2016} follows a similar approach to \cite{Villa2013}, however extending also to nonconvex problems using variable metric strategies and linesearch.

For the recently popular class of first-order primal-dual algorithms mentioned above the list remains short: 
to the best of our knowledge the only work which considers inaccuracies in the proximal operators for this class of algorithms is the one of Condat \cite{Condat2013}, who explicitly models errors and proves convergence under mild assumptions on the decay of the errors. 
However, he does not show any convergence rates.
\colblue{We must also mention Nemirovski's approach in~\cite{Nemirovski2004} which is
an extension of the extragradient method. This saddle-point optimization algorithm converges with an optimal $O(1/N)$ convergence rate as soon as a particular inequality is satisfied at each iteration, possibly with a controlled error term (\textit{cf.}~Prop~2.2 in~\cite{Nemirovski2004}).}

The goal of this paper is twofold: 
Most importantly, we investigate the convergence of the primal-dual algorithms presented in \cite{Chambolle2011,Chambolle2016a} for problems of the form 
\begin{align*}
 \min_{x \in \X} \max_{y \in \Y} ~ \langle Kx, y \rangle + f(x) + g(x) - h^*(y),
\end{align*}
for convex and lower semicontinuous $g$ and $h$ and convex and Lipschitz differentiable $f$, 
with errors occurring in the computation of $\nabla f$ and the proximal operators for $g$ and $h^*$.
Following the line of the preceding works on forward-backward splittings, we consider the different notions of inexact proximal points used in \cite{Schmidt2011} and extended in \cite{Salzo2012,Villa2013,Aujol2015} and, assuming an appropriate decay of the errors, establish the convergence rates of \cite{Chambolle2011,Chambolle2016a} also for perturbed algorithms.
More precisely, we prove the well-known $\Bigo{1/N}$ rate for the basic version, a $\Bigo{1/N^2}$ rate if either $f,g$ or $h^*$ are strongly convex, and a linear convergence rate in case both the primal and dual objective are strongly convex.
Moreover we show that also under a slower decay of errors we can establish rates, however unsurprisingly slower depending on the errors.

In the spirit of \cite{Villa2013} for inexact forward-backward algorithms, the second goal of this paper is to provide an interesting application for such inexact primal-dual algorithms.
We put the focus on situations where one takes the path of inexactness deliberately in order to split the global objective more efficiently.
A particular instance are problems of the form 
\begin{align}\label{eq:intro_application}
 \min_x ~ h(K_1x) + g(K_2x) = \min_x \max_y ~ \langle y,K_1 x\rangle + g(K_2x) - h^*(y).
\end{align}
A popular example is the TV-$L^1$ model with some imaging operator $K_1$, where $g$ and $h$ are chosen to be $L^1$-norms and $K_2 = \nabla$ is a gradient.
It has e.g been studied analytically by \cite{Alliney1992,Alliney1996,Alliney1997} and subsequently by \cite{Nikolova2002,Nikolova2004,Karkkainen2005,Chan2005}.
However, due to the nonlinearity and nondifferentiability of the involved terms, solutions of the model are numerically hard to compute.
One can find a variety of approaches to solve the TV-$L^1$ model, reaching from (smoothed) gradient descent \cite{Chan2005} over interior point methods \cite{Fu2006}, primal-dual methods using a semi-smooth Newton method \cite{Dong2009} to augmented Lagrangian methods \cite{Esser2010a,Wu2012}.
Interestingly, the inexact framework we propose in this paper provides a very simple and intuitive algorithmic approach to the solution of the TV-$L^1$ model.
More precisely, applying an inexact primal-dual algorithm to formulation \eqref{eq:intro_application}, we obtain a nested algorithm in the spirit of \cite{Beck2009b,Chaux2009,Barbero2011,Fadili2011,Sawatzky2013,Villa2013,Ehrhardt2016},
\begin{align*}
 y^{n+1} &= \prox_{\sigma h^*}(y^n + \sigma K_1(x^{n+1} + \theta(x^{n+1} - x^n))), \\
 x^{n+1} &= \prox_{\tau (g \circ K_2)}(x^n - \tau K_1^*y^{n+1}),
\end{align*}
where $\prox_{\sigma h^*}$ denotes the proximal map with respect to $h^*$ and step size $\sigma$ (cf. Section \ref{sec:inexact_proximum}).
It requires to solve the inner subproblem of the proximal step with respect to $g \circ K_2$, i.e. a TV-denoising problem, which does not have a closed-form solution but has to be solved numerically. 
It has been observed in \cite{Beck2009b} that, using this strategy on the primal TV-$L^2$ deblurring problem can cause the FISTA algorithm to diverge if the inner step is not executed with sufficient precision.  
As a remedy, the authors of \cite{Villa2013} demonstrated that the theoretical error bounds they established for inexact FISTA can also serve as a criterion for the necessary precision of the inner proximal problem and hence make the nested approach viable.
We show that the bounds for inexact primal-dual algorithms established in this paper can be used to make the nested approach viable for entirely non-differentiable problems such as the TV-$L^1$ model, while the results of \cite{Villa2013} for partly smooth objectives can also be obtained as a special case of the accelerated versions. 

In the context of inexact and nested algorithms it is worthwhile mentioning the very recent `Catalyst' framework \cite{Lin2015,Lin2017}, which uses nested inexact proximal point methods to accelerate a large class of generic optimization problems in the context of machine learning. 
The inexactness criterion applied there is the same as in \cite{Schmidt2011,Aujol2015}. 
Our approach, however, is much closer to \cite{Schmidt2011,Villa2013,Aujol2015}, stating convergence rates for perturbed algorithms, while \cite{Lin2015,Lin2017} focus entirely on nested algorithms, which we only consider as a particular instance of perturbed algorithms in the numerical experiments.

The remainder of the paper is organized as follows: 
In the next section we introduce the notions of inexact proxima that we will use throughout the paper.
In the following Section 3 we formulate inexact versions of the primal-dual algorithms presented in \cite{Chambolle2011} and \cite{Chambolle2016a} and prove their convergence including rates depending on the decay of the errors.
In the numerical Section 4 we apply the above splitting idea for nested algorithms to some well-known imaging problems and show how inexact primal-dual algorithms can be used to improve their convergence behavior.

\section{Inexact computations of the proximal point}\label{sec:inexact_proximum}
In this section we introduce and discuss the idea of the proximal point and several ways for its approximation. 
For a proper, convex and lower semicontinuous function $g \colon \X \to \Rb$ mapping from a Hilbert space $\X$ to the extended real line $\Rb = \R \cup \{ \infty \}$ and $y \in \X$ the proximal point  \cite{Moreau1965,Martinet1970,Rockafellar1976a,Rockafellar1976} is given by
\begin{align}\label{eq:proximal_operator}
 \prox_{\tau g} (y) = \arg \min_{x \in \X} \left\{ \frac{1}{2\tau} \|x - y\|^2 + g(x) \right\}.
\end{align}
Since $g$ is proper we directly obtain $\prox_{\tau g} (y) \in \dom g$.
The $1/\tau$-strongly convex mapping $\prox_{\tau g} \colon \X \to \X$ is called proximity operator of $\tau g$.
Letting 
\begin{align}\label{eq:proximity_function}
 G_\tau \colon \X \to \Rb, \quad x \mapsto \frac{1}{2\tau} \|x - y\|^2 + g(x)
\end{align}
be the proximity function, the first-order optimality condition for the proximum gives different characterizations of the proximal point: 
\begin{align}\label{eq:optimality_proximal_point}
 z = \prox_{\tau g} (y) \iff 0 \in \partial G_\tau(z) \iff \frac{y-z}{\tau} \in \partial
 g(z).
\end{align}
Based on these equivalences we introduce four different types of inexact computations of the proximal point, which are differently restrictive. 
Most have already been considered in literature and we give reference next to the definitions. 
We like to recommend \cite{Salzo2012,Villa2013} for some illustrations of the different notions in case of a projection.
We start with the first expression in \eqref{eq:optimality_proximal_point}. 
\begin{definition}\label{def:type0}
 Let $\e \geq 0$. 
 We say that $z \in \X$ is a type-0 approximation of the proximal point $\prox_{\tau g}(y)$ with precision $\e$ if 
 \begin{align*}
  z \approx_0^{\e} \prox_{\tau g}(y)  \overset{\mathrm{def}}{\iff} \| z - \prox_{\tau g}(y) \| \leq \sqrt{2 \tau \e}.
 \end{align*}
\end{definition}
This refers to choosing a proximal point from the $\sqrt{2\tau\e}$-ball around the true proximum.
It is important to notice that a type-0 approximation is {\it not} necessarily feasible, i.e. it can occur that $z \notin \dom g$.
This can easily be verified e.g. for $g$ being the indicator function of a convex set, and requires us to treat this approximation slightly differently from the following ones in Appendix \ref{sec:type_0}.
\colblue{Observe that it is easy to check a posteriori the precision of a type-0 approximation provided $\partial g$ is easy to evaluate. Indeed,
  if $e\in \tau\partial g(z)+ z-y$, standard estimates show that $\| z - \prox_{\tau g}(y) \|\le \|e\|$ and $z\approx_0^{\e}\prox_{\tau g}(y)$
  for $\e=\|e\|^2/(2\tau)$.}

In order to relax the second or third expression in \eqref{eq:optimality_proximal_point}, we need the notion of an $\e$-subdifferential of $g \colon \X \to \Rb$ at $z \in \X$: 
\begin{align*}
 \partial_{\e} g(z) := \{ p \in \X ~|~ g(x) \geq g(z) + \langle p, x-z \rangle - \e ~ \forall x \in \X \}. 
\end{align*}
As a direct consequence of the definition we obtain a notion of $\e$-optimality 
\begin{align}\label{eq:e_subdiff_optimality}
 0 \in \partial_{\e} g(z) \iff g(z) \leq \inf g + \e.
\end{align}
Based on this and the second expression in \eqref{eq:optimality_proximal_point}, we define a second notion of an inexact proximum. 
It has e.g. been considered in \cite{Schmidt2011,Aujol2015} to prove the convergence of inexact proximal gradient methods under the presence of errors.
\begin{definition}\label{def:type1}
 Let $\e \geq 0$. 
 We say that $z \in \X$ is a type-1 approximation of the proximal point $\prox_{\tau g}(y)$ with precision $\e$ if 
 \begin{align*}
  z \approx_1^{\e} \prox_{\tau g}(y)  \overset{\mathrm{def}}{\iff} 0 \in \partial_{\e} G_\tau(z).
 \end{align*}
\end{definition}
Hence, by \eqref{eq:e_subdiff_optimality}, a type-1 approximation minimizes the energy of the proximity function \eqref{eq:proximity_function} up to an error of $\e$. 
It turns out that a type-0 approximation is weaker than a type-1 approximation:
\begin{lemma}\label{lem:prox0}
 Let $z \approx_1^{\e} \prox_{\tau g}(y)$. 
 Then $z \in \dom g $ and 
 \begin{align*}
  \| z - \prox_{\tau g}(y) \| \leq \sqrt{2 \tau \e},
 \end{align*}
 that is $z \approx_0^{\e} \prox_{\tau g}(y)$.
\end{lemma}
The result is easy to verify and can be found e.g. in \cite{Rockafellar1976,Gueler1992,Salzo2012}.
An even more restrictive version of an inexact proximum can be obtained by relaxing the third expression in \eqref{eq:optimality_proximal_point}, which has been introduced in \cite{Lemaire1992} and subsequently been used in \cite{Cominetti1997,Salzo2012} in the context of inexact proximal point methods.
\begin{definition}\label{def:type2}
 Let $\e \geq 0$. 
 We say that $z \in \X$ is a type-2 approximation of the proximal point $\prox_{\tau g}(y)$ with precision $\e$ if 
 \begin{align*}
  z \approx_2^{\e} \prox_{\tau g}(y)  \overset{\mathrm{def}}{\iff} \frac{y-z}{\tau} \in \partial_{\e} g(z).
 \end{align*}
\end{definition}
Letting $\phi_\tau(z) = \|z - y \|^2/(2 \tau)$, the following characterization from \cite{Salzo2012} of the $\e$-subdifferential of the proximity function $G_\tau = g + \phi_\tau$ sheds light on the difference between a type-1 and type-2 approximation: 
\begin{align}\label{eq:subdiff_decomp}
 \partial_\e G_\tau(z) 
 &= \bigcup_{0 \leq \e_1 + \e_2 \leq \e} \partial_{\e_1} g(z) + \partial_{\e_2} \phi_\tau(z) \nonumber \\
 &= \bigcup_{0 \leq \e_1 + \e_2 \leq \e} \partial_{\e_1} g(z) + \left\{ \frac{z-y-e}{\tau} : \|e\| \leq \sqrt{2 \tau \e_2} \right\}.
\end{align}
Equation \eqref{eq:subdiff_decomp} decomposes the error in the $\e$-subdifferential of $G_\tau$ into two parts related to $g$ respectively $\phi_\tau$.
As a consequence, for a type-1 approximation it is not clear how the error is distributed between $g$ or $\phi_\tau$, while for a type-2 approximation the error occurs solely in $g$.
Hence a type-2 approximation can be seen as an extreme case of a type-1 approximation with $\e_2 = 0$. 

In view of the decomposition \eqref{eq:subdiff_decomp}, we define a fourth notion of an inexact proximum as the extreme case $\e_1 = 0$, which has been considered in e.g. \cite{Rockafellar1976} and \cite{Gueler1992}. 
\begin{definition}\label{def:type3}
 Let $\e \geq 0$. 
 We say that $z \in \X$ is a type-3 approximation of the proximal point $\prox_{\tau g}(y)$ with precision $\e$ if 
 \begin{align*}
  z \approx_3^{\e} \prox_{\tau g}(y)  \overset{\mathrm{def}}{\iff} \exists e \in \X, \|e\| \leq \sqrt{2 \tau \e} : z = \prox_{\tau_g}(y + e).
 \end{align*}
\end{definition}
Definition \ref{def:type3} gives the notion of a ``correct'' output for an incorrect input of the proximal operator.
Being the two extreme cases, type-2 and type-3 proxima are also proxima of type 1. 
The decomposition \eqref{eq:subdiff_decomp} further leads to the following lemma from \cite{Schmidt2011}, which allows to treat the type-1, -2 and -3 approximations in the same setting. 
\begin{lemma}\label{lem:existence_r}
 If $z \in \X$ is a type-1 approximation of $\prox_{\tau g}(y)$ with precision $\e$, then there exists $r \in \X$ with $\|r\| \leq \sqrt{2 \tau \e}$ such that 
 \begin{align*}
  (y-z-r)/\tau \in \partial_\e g(z).
 \end{align*}
\end{lemma}
Now note that letting $r = 0$ in Lemma \ref{lem:existence_r} gives a type-2 approximation, replacing the $\e$-subdifferential by a normal one gives a type-3 approximation.
We mention that there exist further notions of approximations of the proximal point, e.g. used in \cite{Gueler1992}, and refer to \cite[Section 2.2]{Villa2013} for a compact discussion.

Even tough we prove our results for different types of approximations, the most interesting one in terms of practicability is the approximation of type 2. 
This is due to the following insights obtained by \cite{Villa2013}:
Without loss of generality let $g(x) = w(Bx),$
with proper, convex and l.s.c. $w\colon \Z \to \Rb$ and linear $B \colon \X \to \Z$. 
Then the calculation of the proximum requires to solve 
\begin{align}\label{eq:prox_prob}
 \min_{x\in\X}~ G_\tau(x) 
 &=  \min_{x \in \X} ~ \frac{1}{2\tau} \| x - y \|^2 + w(Bx).
\end{align}
Now if there exists $x_0 \in \X$ such that $g$ is continuous in $Bx_0$, the Fenchel-Moreau-Rockafellar duality formula \cite[Corollary 2.8.5]{Zalinescu2002} states that 
\begin{align*}
 \min_{x\in\X}~ G_\tau(x) = - \min_{z \in \Z} ~ \frac{\tau}{2} \|B^*z\|^2 - \langle B^*z,y\rangle + w^*(z), 
\end{align*}
where we refer to the right hand side as the ``dual'' problem $W_\tau(z)$.
Furthermore we can always recover the primal solution $\xh$ from the dual solution $\zh$ via the relation $\xh = y - B^* \zh$. 
Most importantly, we obtain that $\xh$ and $\zh$ solve the primal respectively dual problem if and only if the duality gap $\Gcal(x,z) := G_\tau(x) + W_\tau(z)$ vanishes, i.e. 
\begin{align*}
 0 = \min_{(x,z) \in \X \times \Z} G_\tau(x) + W_\tau(z) = \Gcal(\xh,\zh).
\end{align*}
The following result in \cite{Villa2013} states that the duality gap can also be used as a criterion to assess admissible type-2 approximations of the proximal point: 
\begin{proposition}\label{prop:duality_gap_prox}
 Let $z \in \Z$. 
 Then 
 \begin{align*}
  \Gcal(y - B^* z,z) \leq \e \Rightarrow y - B^*z \approx_2^\e \prox_{\tau g}(y).
 \end{align*}
\end{proposition}
Proposition \ref{prop:duality_gap_prox} has an interesting implication: 
if one can construct a feasible dual variable $z$ during the solution of \eqref{eq:prox_prob}, it is easy to check the admissibility of the corresponding primal variable $x$ to be a type-2 approximation by evaluating the duality gap.
We shall make use of that in the numerical experiments in Section \ref{sec:numerics}.

\colblue{
Of course, since a type-2 approximation automatically is a type-1 and type-0 approximation, the above argumentation is also valid to find feasible approximations in these cases, which however is of no use. 
Since type-1 and type-0 approximations are technically less restrictive, it stands to find criteria on how to evaluate when an approximation is of such type \emph{without} already being an approximation of type 2.
An example of a type-0 approximation may be found in problems where the desired proximum is the projection onto the intersection of convex sets. 
The (inexact) proximum may be computed in a straightforward fashion using Dykstra's algorithm \cite{Dykstra1983}, which has e.g. been done in \cite{Brancolini2018} or \cite[Ex. 7.7]{Alberti2003,Chambolle2016a} for Mumford-Shah-type segmentation problems. 
Depending on the involved sets, one may get an upper bound on the maximal distance of the current iterate of Dykstra's algorithm to these sets, obtaining a bound on how far the current iterate is from the true proximum. 
These considerations, however, require to be tested in the respective cases. 
}

\section{Inexact primal-dual algorithms}\label{sec:ipd}
We can now prove the convergence of some primal-dual algorithms under the presence of the respective error. 
We start with the type-1, -2 and -3 approximations and outline in Appendix \ref{sec:type_0} how to get a grip also on the type-0 approximation.
The convergence analysis in this chapter is based on a combination of techniques derived in previous works on the topic: 
similar results on the convergence of exact primal-dual algorithms can be found e.g. in \cite{Chambolle2011,Chambolle2016} and \cite{Chambolle2016a}, the techniques to obtain error bounds for the inexact proximum are mainly taken from \cite{Schmidt2011} and \cite{Aujol2015}. 
We consider the saddle-point problem 
\begin{align}\label{eq:saddle_point_problem}
 \min_{x \in \X} \max_{y \in \Y} ~ \Lcal(x,y) = \langle Kx, y \rangle + f(x) + g(x) - h^*(y),
\end{align}
where we make the following assumptions:
\begin{enumerate}
 \item $K \colon \X \to \Y$ is a linear and bounded operator between Hilbert spaces $\X$ and $\Y$ with norm $L = \|K\|$,
 \item $f \colon \X \to \Rb$ is proper, convex, lower semicontinuous and differentiable with $L_f$-Lipschitz gradient, 
   \begin{align*}
    \| \nabla f(x) - \nabla f(x') \| \leq L_f \| x - x' \| \quad \text{ for all } x,x' \in \dom f,
   \end{align*}
 \item $g ,h \colon \X \to \Rb$ are proper, lower semicontinuous and convex functions, 
 \item problem \eqref{eq:saddle_point_problem} admits at least one solution $(\xs,\ys) \in \X \times \Y$.
\end{enumerate}
It is well-known that taking the supremum over $y$ in $\Lcal(x,y)$ leads to the corresponding ``primal'' formulation of the saddle-point problem \eqref{eq:saddle_point_problem} 
\begin{align*}
 \min_{x \in \X} ~ f(x) + g(x) + h(Kx),  
\end{align*}
which for a lot of variational problems might be the starting point. 
Analogously, taking the infimum over $x$ leads to the dual problem.
Given an algorithm producing iterates $(x^N,y^N)$ for the solution of \eqref{eq:saddle_point_problem}, the goal of this section is to obtain estimates 
\begin{align}\label{eq:goal}
 \Lcal(x^N,y) - \Lcal(x,y^N) \leq \frac{C(x,y,x^0,y^0)}{N^\alpha}
\end{align}
for $\alpha > 0$ and $(x,y) \in \X \times \Y$. 
If $(x,y) = (\xs,\ys)$ is a saddle point, the left hand side vanishes if and only if the pair $(x^N,y^N)$ is a saddle point itself, yielding a convergence rate in terms of the primal-dual objective in $\Bigo{1/N^\alpha}$.
Under mild additional assumptions one can then derive estimates e.g. for the error in the primal objective. 
If the supremum over $y$ in $\Lcal(x^N,y)$ is attained at some $\tilde{y}$, one easily sees that 
\begin{align}\label{eq:primal_estimate}
 &f(x^N) + g(x^N) + h(Kx^N) - [ f(\xs) + g(\xs) + h(K\xs) ] \\
 &= \sup_{y \in \Y} ~ \Lcal(x^N,y) - \sup_{y \in \Y} ~ \Lcal(\xs,y) 
 \leq \Lcal(x^N,\yt) - \Lcal(\xs,y^N)
 \leq \frac{C(\xs,\yt,x^0,y^0)}{N^\alpha},\nonumber
\end{align}
giving a convergence estimate for the primal objective.

In the original versions of primal-dual algorithms (e.g. \cite{Chambolle2011,Chambolle2016}), the authors require $h^*$ and $g$ to have a simple structure such that their proximal operators \eqref{eq:proximal_operator} can be sufficiently easily evaluated. 
A particular feature of most of these operators is that they have a closed-form solution and can hence be evaluated exactly.
We study the situation where the proximal operators for $g$ or $h^*$ can only be evaluated up to a certain precision in the sense of Section \ref{sec:inexact_proximum}, and as well the gradient of $f$ may contain errors.
As opposed to the general iteration of an exact primal-dual algorithm \cite{Chambolle2016}
\begin{align}\label{eq:pd_exact}
  \begin{split}
   \yh &= \prox_{\sigma h^*} (\yb + \sigma K\xt), \\ 
   \xh &= \prox_{\tau g} (\xb - \tau ( K^*\yt + \nabla f(\xb))),
  \end{split}
\end{align}
where $(\xb,\yb)$ and $(\xt,\yt)$ are the previous points, and $(\xh,\yh)$ are the updated exact points, we introduce the general iteration of an inexact primal-dual algorithm
\begin{align} \label{eq:pd_general}
  \begin{split}
   \yc &\approx_2^\delta \prox_{\sigma h^*} (\yb + \sigma K\xt), \\ 
   \xc &\approx_i^\varepsilon \prox_{\tau g} (\xb - \tau ( K^*\yt + \nabla f(\xb) + e)).
  \end{split}
\end{align}
Here the updated points $(\xc,\yc)$ denote the inexact proximal point 
, which are only computed up to precision $\e$ respectively $\delta$,
in the sense of a type-$2$ approximation from Section \ref{sec:inexact_proximum} for $\yc$ and a type-$i$ approximation for $i\in \{1,2,3\}$ for $\xc$. 
The vector $e \in \X$ denotes a possible error in the gradient of $f$.
We use two different pairs of input points $(\xb,\yb)$ and $(\xt,\yt)$ in order to include intermediate overrelaxed input points. 
It is clear, however, that we require $\xt$ to depend on $\xc$ respectively $\yt$ on $\yc$ in order to couple the primal and dual updates.

At first glance it seems counterintuitive that we allow errors of type 1, 2 and 3 in $\xc$, while only allowing for type-2 errors in $\yc$. 
The following general descent rule for the iteration \eqref{eq:pd_general} sheds some more light on this fact and forms the basis for all the following proofs.
It can be derived using simple convexity results and resembles the classical energy descent rules for forward-backward splittings.
It can then be used to obtain estimates on the decay of the objective of the form \eqref{eq:goal}. 
We prove the descent rule for a type-1 approximation of the primal proximum since we always obtain the result for a type-2 or type-3 approximation as a special case.
\begin{lemma}\label{lem:general_inequality}
 Let $\tau, \sigma > 0$ and $(\xc,\yc)$ be obtained from $(\xb,\yb)$ and $(\xt,\yt)$ and the updates \eqref{eq:pd_general} for $i = 1$.
 Then for every $(x,y) \in \X \times \Y$ we have
 \begin{align}
    \Lcal(\xc,y) &- \Lcal(x,\yc) \leq ~ \frac{\| x - \xb \|^2}{2 \tau} + \frac{\|y - \yb \|^2}{2 \sigma} - \frac{\| x - \xc \|^2}{2 \tau} - \frac{1-\tau L_f}{2 \tau} \| \xb - \xc \|^2 \nonumber \\
 & - \frac{\| y - \yc \|^2}{2\sigma} - \frac{\| \yb - \yc \|^2}{2 \sigma} + \langle K(x-\xc), \yt - \yc \rangle - \langle K (\xt - \xc), y- \yc \rangle \nonumber \\
 & \qquad + \left( \|e \| + \sqrt{2 \e / \tau} \right) \| x - \xc \| + \e + \delta. \label{eq:general_inequality}
 \end{align}
\end{lemma} 
\begin{proof}
 For the inexact type-2 proximum $\yc$ we have by Definition \ref{def:type2} that \\ $(\bar{y} + \sigma K \tilde{x} - \yc)/\sigma \in \partial_\delta h^*(\yc)$, 
so by the definition of the subdifferential we find
\begin{align}\label{eq:dual_inequality}
 h^*(\yc) 
 & \leq h^*(y) + \langle \frac{\yb + \sigma K \xt - \yc}{\sigma}, \yc -y \rangle \nonumber + \delta \nonumber \\
 &= h^*(y) + \langle \frac{\yb - \yc}{\sigma}, \yc -y \rangle + \langle K \xt, \yc - y \rangle + \delta \nonumber \\ 
 & \leq h^*(y) - \frac{\| \yb - \yc \|^2}{2\sigma} - \frac{\| y - \yc \|^2}{2 \sigma} + \frac{\| \yb - y \|^2}{2 \sigma} + \langle K \xt, \yc - y \rangle + \delta. 
\end{align}
For the inexact type-$1$ primal proximum, from Definition \ref{def:type1} and Lemma \ref{lem:existence_r} we have that there exists a vector $r$ with $\|r\| \leq \sqrt{2 \tau \varepsilon}$ such that 
\begin{align*}
 (\bar{x} - \tau (K^* \tilde{y} + \nabla f(\bar{x}) + e) - \xc - r)/\tau \in \partial_{\varepsilon} g(\xc).
\end{align*}
Hence we find that
\begin{align}\label{eq:primal_inequality}
 g(\xc) 
 & \leq g(x) + \langle \frac{\xb - \tau (K^* \yt + \nabla f(\xb) + e) - \xc - r}{\tau}, \xc - x \rangle + \varepsilon \nonumber \\
 & \leq g(x) - \frac{\| \xb - \xc \|^2}{2 \tau} - \frac{\| x - \xc \|^2}{2 \tau} + \frac{\| \xb - x \|^2}{2 \tau} + \langle \yt, K(x - \xc) \rangle \nonumber\\
 &\qquad - \langle \nabla f(\xb), \xc - x \rangle   + \left( \|e\| + \sqrt{(2 \varepsilon / \tau} \right) \|x - \xc\| + \varepsilon,
\end{align}
where we applied the Cauchy-Schwarz inequality to the error term.
Further by the Lipschitz property \colblue{and convexity} of $f$ we have (cf. \cite{Nesterov2014}) 
\begin{align}\label{eq:Lipschitz}
 f(\xc) \leq f(x) + \langle \nabla f(\xb), \xc - x \rangle + \frac{L_f}{2} \| \xc - \xb \|^2.
\end{align}
Now we add the equations \eqref{eq:dual_inequality}, \eqref{eq:primal_inequality} and \eqref{eq:Lipschitz},  insert 
\begin{align*}
 \langle K\xc,y \rangle - \langle K\xc,y \rangle, \qquad 
 \langle Kx,\yc \rangle - \langle Kx,\yc \rangle, \qquad 
 \langle K\xc,\yc \rangle - \langle K\xc,\yc \rangle, 
\end{align*}
and rearrange to arrive at the result.
\end{proof}
We point out that, as a special case, choosing a type-2 approximation for the primal proximum in Lemma \ref{lem:general_inequality} corresponds to dropping the square root in the estimate \eqref{eq:general_inequality}, choosing a type-3 approximation corresponds to dropping the additional $\e$ at the end.  
Also note that the above descent rule is the same as the one in \cite{Chambolle2011,Chambolle2016} except for the additional error terms in the last line of \eqref{eq:general_inequality}.

Lemma \ref{lem:general_inequality} has an immediate implication: 
in order to control the errors $\|e\|$ and $\e$ in the last line of Lemma \ref{lem:general_inequality} it is obvious that we need to find a useful bound on $\| x - \xc \|$.  
We shall obtain this bound using a linear extrapolation in the primal variable $x$ \cite{Chambolle2011}. 
However, if we allow e.g. a type-1 approximation also in $\yc$, we obtain an additional error term in \eqref{eq:general_inequality} involving $\| y - \yc\|$ that we need to bound as well.
Even though we shall be able to obtain a bound in most cases, it will be arbitrarily weak due to the asymmetric nature of the used primal-dual algorithms, or go along with severe restrictions on the step sizes. 
This fact will become more obvious from the proofs in the following.


\subsection{The convex case: no acceleration} 
We start with a proof for a basic version of algorithm \eqref{eq:pd_general} using a technical lemma taken from \cite{Schmidt2011} (see Appendix \ref{sec:technical_lemmas}).
The following inexact primal-dual algorithm corresponds to the choice 
\begin{align}\label{eq:choices1}
 (\xc,\yc) = (x^{n+1},y^{n+1}), \quad (\xb,\yb) = (x^n, y^n), \quad (\xt, \yt) = (2x^n - x^{n-1},y^{n+1}) 
\end{align}
in algorithm \eqref{eq:pd_general}:
\begin{align}\label{eq:alg_no_acc}
 \begin{split}
 y^{n+1} &\approx_2^{\delta_{n+1}} \prox_{\sigma h^*} (y^n + \sigma K (2 x^n - x^{n-1} )) \\
 x^{n+1} &\approx_i^{\varepsilon_{n+1}} \prox_{\tau g} ( x^n - \tau (K^* y^{n+1} + \nabla f(x^n) + e^{n+1})).
 \end{split}
\end{align}
\begin{theorem}\label{thm:inex_pd_basic}
 Let $L = \|K\|$ and choose \colblue{small $\beta > 0$} and $\tau, \sigma > 0$ such that $\tau L_f + \sigma \tau L^2 + \tau \beta L< 1$, and let the iterates $(x^n, y^n)$ be defined by Algorithm \eqref{eq:alg_no_acc} for $i \in \{1,2,3\}$. 
 Then for any $N \geq 1$ and $X^N := \left(\sum_{n=1}^N x^n \right)/N$, $ Y^N := \left(\sum_{n=1}^N y^n \right)/N$
 we have for a saddle point $(\xs,\ys) \in \X \times \Y$ that
 \begin{align}
  \Lcal (X^N,\ys) &- \Lcal (\xs, Y^N) \nonumber \\
  &\leq  \frac{1}{2\tau N} \left( \|\xs - x^0 \| + \sqrt{\frac{\tau}{\sigma}} \|\ys - y^0 \| + 2 A_{N,i} + \sqrt{2 B_{N,i}} \right)^2, \label{eq:no_acc_Lagrangian} 
 \end{align}
 where
 \begin{alignat*}{3}
  A_{N,1} &= \sum_{n=1}^N \tau \| e^n \| + \sqrt{2 \tau \varepsilon_n}, && 
  \quad B_{N,1} = \sum_{n=1}^N \tau \varepsilon_n + \tau \delta_n, \\
  A_{N,2} &= \sum_{n=1}^N \tau \| e^n \|, && 
  \quad B_{N,2} = \sum_{n=1}^N \tau \varepsilon_n + \tau \delta_n, \\
  A_{N,3} &= \sum_{n=1}^N \tau \| e^n \| + \sqrt{2 \tau \varepsilon_n}, &&
  \quad B_{N,3} = \sum_{n=1}^N \tau \delta_n.
 \end{alignat*}
\end{theorem}
\colblue{
\begin{remark}\textup{
 The purpose of the parameter $\beta > 0$ is only of technical nature and is needed in order to show convergence of the iterates of algorithm \eqref{eq:alg_no_acc}. 
 In practice, however, we did not observe any issues setting it super small (respectively, to zero).  
 Its role will become obvious in the next Theorem. }
\end{remark}
}
\begin{proof}
Using the choices \eqref{eq:choices1} in Lemma \ref{lem:general_inequality} leads us to
\begin{align}\label{eq:basic_inequality_no_acc}
 \Lcal&(x^{n+1},y) - \Lcal(x,y^{n+1}) \leq \frac{\| x - x^n \|^2}{2\tau} - \frac{\| x - x^{n+1} \|^2}{2\tau} \nonumber \\ 
 & - \frac{1-\tau L_f}{2\tau} \|x^{n+1} - x^n \|^2 + \frac{\|y - y^n\|^2}{2\sigma} - \frac{\|y - y^{n+1} \|^2}{2\sigma} - \frac{\|y^{n+1} - y^n \|^2}{2\sigma} \nonumber \\
 & \quad +\langle K((x^{n+1} - x^n) - (x^n - x^{n-1})),y-y^{n+1} \rangle \nonumber \\ 
 & \qquad + \left( \|e^{n+1}\| + \sqrt{(2 \varepsilon_{n+1})/\tau} \right) \| x - x^{n+1} \| + \e_{n+1} + \delta_{n+1}.
\end{align}
The goal of the proof is, as usual, to manipulate this inequality such that we obtain a recursion where
most of the terms cancel when summing the inequality.
The starting point is an extension of the scalar product on the right hand side: 
\begin{align*}
 & ~ \langle K((x^{n+1} - x^n) - (x^n - x^{n-1})),y-y^{n+1} \rangle \\
 = & ~\langle K(x^{n+1} - x^n), y - y^{n+1} \rangle - \langle K(x^n - x^{n-1}),y-y^n \rangle \\
 &\qquad \quad + \langle K(x^n - x^{n-1}), y^{n+1} - y^n \rangle \\
 \leq & ~ \langle K(x^{n+1} - x^n), y - y^{n+1} \rangle - \langle K(x^n - x^{n-1}),y-y^n \rangle \nonumber \\ 
 &\qquad \quad + (\tau \sigma L^2 + \tau \beta L) \frac{\|x^n - x^{n-1} \|^2}{2\tau} + \frac{\sigma L}{\sigma L + \beta} \frac{\|y^{n+1} - y^n \|^2}{2\sigma} ,
\end{align*}
where we used (for $\alpha > 0$) that by Young's inequality for every $x,x' \in \X$ and $y,y' \in \Y$,
\begin{align}\label{eq:Young_inequality}
 \langle K(x-x'),y-y' \rangle \leq L \|x-x'\| \|y-y'\| \leq \frac{L\alpha \tau}{2\tau} \|x-x'\|^2 + \frac{L\sigma}{2\alpha \sigma} \|y-y'\|^2,
\end{align}
and $\alpha = \sigma L + \beta$.
This gives 
\begin{align}\label{eq:no_acc_ineq}
 \Lcal&(x^{n+1},y) - \Lcal(x,y^{n+1}) \leq \frac{\| x - x^n \|^2}{2\tau} - \frac{\|x - x^{n+1} \|^2}{2\tau}  \nonumber \\
    &- \frac{1 - \tau L_f}{2\tau} \|x^{n+1} - x^n \|^2 + \frac{\tau \sigma L^2 + \tau \beta L}{2\tau} \|x^n - x^{n-1} \|^2 + \frac{\| y - y^n \|^2}{2\sigma} \nonumber \\ 
    &\quad - \frac{\| y - y^{n+1} \|^2}{2\sigma} - \frac{\beta}{\sigma L + \beta}  \frac{\|y^{n+1} - y^n \|^2}{2\sigma} \nonumber \\
    &\qquad + \langle K(x^{n+1} - x^n), y - y^{n+1} \rangle - \langle K(x^n - x^{n-1}),y-y^n \rangle \nonumber \\ 
   & \qquad \quad + \left( \|e^{n+1}\| + \sqrt{(2 \varepsilon_{n+1})/\tau} \right) \| x - x^{n+1} \| + \varepsilon_{n+1} + \delta_{n+1}.
\end{align}
Now we let $x^0 = x^{-1}$ and sum \eqref{eq:no_acc_ineq} from $n = 0, \dots, N-1$ to obtain 
\begin{align*}
 \sum_{n=1}^N &\Lcal(x^n,y) - \Lcal(x,y^n) \leq  \frac{\| x - x^0 \|^2}{2\tau} + \frac{\| y - y^0 \|^2}{2\sigma} - \frac{\| x - x^N \|^2}{2 \tau} - \frac{\| y - y^N \|^2}{2 \sigma} \\
 & - \frac{1 - \tau L_f}{2\tau} \|x^N - x^{N-1} \|^2   - \kappa \sum_{n=1}^{N-1} \frac{1}{2\tau} \|x^n - x^{n-1} \|^2 \\ 
 & \quad - \frac{\beta}{\sigma L + \beta} \sum_{n=1}^N \frac{\|y^n - y^{n-1} \|^2}{2\sigma} + \langle K(x^N - x^{N-1}, y - y^N \rangle \\
 &\qquad \quad + \sum_{n=1}^N \left( \|e^n\| + \sqrt{(2 \varepsilon_n)\tau} \right) \| x - x^n \| + \sum_{n=1}^N (\varepsilon_n + \delta_n),
\end{align*}
where $\kappa = 1 - \tau L_f - \tau \sigma L^2 - \tau \beta L$.
With Young's inequality on the inner product with $\alpha = \frac{1 - \tau L_f}{L \tau}$ such that $L \alpha \tau = 1 - \tau L_f$ and $(L \sigma)/\alpha = (\tau \sigma L^2)/(1 - \tau L_f) < 1$ we obtain
\begin{align}\label{eq:no_acc_proof_main_ineq}
 &\sum_{n=1}^N \Lcal(x^n,y) - \Lcal(x,y^n) + \frac{1}{2 \tau} \| x - x^N \|^2 + (1 - \frac{\tau \sigma L^2}{1 - \tau L_f}) \frac{\| y - y^N \|^2}{2 \sigma} \nonumber \\ 
 &\leq \frac{1}{2\tau} \| x - x^0 \|^2 + \frac{1}{2\sigma} \| y - y^0 \|^2 + \sum_{n=1}^N \left(\|e^n\| + \sqrt{(2 \varepsilon_n)/\tau} \right) \| x - x^n \|  \\ 
 & \quad + \sum_{n=1}^N ( \varepsilon_n + \delta_n )  - \kappa \sum_{n=1}^{N-1} \frac{1}{2\tau} \|x^n - x^{n-1} \|^2 - \frac{\beta}{\sigma L + \beta} \sum_{n=1}^N \frac{1}{2\sigma}\|y^n - y^{n-1} \|^2. \nonumber
\end{align}
\colblue{Note that the introduction of the parameter $\beta > 0$ allowed us to ``keep'' an additional term involving the difference of the dual iterates on the right hand side of the inequality. 
This will be essential in order to prove the convergence of the iterates later in Theorem \ref{thm:convergence_no_acc}.
}
The above inequality can now be used to bound the sum on the left hand side as well as $\| x - x^N \|$ by only the initialization $(x^0,y^0)$ and the errors $e^n$, $\varepsilon_n$ and $\delta_n$.
We start with the latter and let $(x,y) = (\xs,\ys)$ such that the sum on the left hand side is nonnegative, hence with $\Delta_0(x,y) := \| x - x^0 \|^2/(2\tau) + \| y - y^0 \|^2/(2\sigma)$ we have
\begin{align*}
 \frac{1}{2 \tau} \| &\xs -  x^N \|^2  \\
 &\leq \Delta_0(\xs,\ys) + \sum_{n=1}^N \left( \|e^n\| + \sqrt{(2 \varepsilon_n)/\tau} \right) \| \xs - x^n \| + \sum_{n=1}^N (\varepsilon_n + \delta_n),
\end{align*}
(note that the third and fourth sum on the right hand side are negative).
We multiply the equation by $2 \tau$ and continue with a technical result by \cite[p.12]{Schmidt2011}. 
Using Lemma \ref{lem:technical_result} with $u_N = \| \xs - x^N \|$, $S_N = 2\tau \Delta_0(\xs,\ys) + 2 \tau \sum_{n=1}^N (\varepsilon_n + \delta_n)$ and $\lambda_n = 2 (\tau \|e^n \| + \sqrt{2 \tau \varepsilon_n})$
we obtain a bound on $\| \xs - x^N \|$:
\begin{align*}
 \|x^N - \xs\| \leq A_N + \sqrt{ 2\tau \Delta_0(\xs,\ys) + 2B_N + A_N^2 },
\end{align*}
where we set $A_N := \sum_{n=1}^N (\tau \|e^n \| + \sqrt{2 \tau \varepsilon_n})$ and $B_N := \sum_{n=1}^N \tau (\varepsilon_n + \delta_n)$.
Since $A_N$ and $B_N$ are increasing we find for all $n \leq N$:
\begin{align}\label{eq:no_acc_bounded_x}
 \|x^n - \xs \| 
 &\leq A_n + \sqrt{2\tau \Delta_0(\xs,\ys) + 2B_n + A_n^2 } \nonumber \\
 &\leq A_N + \sqrt{ 2\tau \Delta_0(\xs,\ys) + 2B_N + A_N^2 } \nonumber \\ 
 &\leq 2 A_N + \| x^0 - \xs \| + \sqrt{\frac{\tau}{\sigma}} \|y^0 - \ys \| + \sqrt{2B_N}.
\end{align}
This finally gives 
\begin{align}\label{eq:no_acc_bound2}
  &\Delta_0(\xs,\ys) + \sum_{n=1}^N \left(\|e^n\| + \sqrt{(2 \varepsilon_n)/\tau} \right) \| \xs - x^n \| + \sum_{n=1}^N ( \varepsilon_n + \delta_n ) \nonumber \\
 \leq & \Delta_0(\xs,\ys) + \frac{1}{\tau} B_N + \frac{1}{\tau} A_N \left( 2 A_N + \|x^0 - \xs \| + \sqrt{\frac{\tau}{\sigma}} \|y^0 - \ys \| + \sqrt{2B_N} \right) \nonumber \\
  = &\frac{1}{2\tau} \left( \|x^0 - \xs \|^2 + \frac{\tau}{\sigma} \|y^0 - \ys \|^2 + 2B_N + 4 A_N^2 \right. \nonumber \\ 
  & \hspace{3cm} \left. + 2 A_N \|x^0 - \xs \| + 2 A_N \sqrt{\frac{\tau}{\sigma}} \|y^0 - \ys \| + 2 A_N \sqrt{2 B_N} \right) \nonumber \\
 \leq &\frac{1}{2\tau} \left( \|x^0 - \xs \| + \sqrt{\frac{\tau}{\sigma}} \|y^0 - \ys \| + 2 A_N + \sqrt{2B_N} \right)^2, 
\end{align}
and bounds the error terms.
We now obtain from \eqref{eq:no_acc_proof_main_ineq} that
\begin{align*}
\sum_{n=1}^N \Lcal(x^n,\ys) &- \Lcal(\xs,y^n) \\
&\leq \frac{1}{2\tau} \left( \|x^0 - \xs\| + \sqrt{\frac{\tau}{\sigma}} \|y^0 - \ys \| + 2 A_N + \sqrt{2B_N} \right)^2,
\end{align*}
which gives the assertion using the convexity of $g,f$ and $h^*$, the definition of the ergodic sequences $X^N$ and $Y^N$ and Jensen's inequality. 
It remains to note that for a type-2 approximation the square root in $A_N$ can be dropped and for a type-3 approximation $B_N =0$, which gives the different $A_{N,i}, B_{N,i}$.
\end{proof}
%
\colblue{ 
We can immediately deduce the following corollary.
\begin{corollary}\label{cor:ifo_basic_rate}
 If $i \in \{1,2,3\}$, $\alpha > 0$ and 
 \begin{align*}
  \|e^n\| = \Bigo{\frac{1}{n^{\alpha + \frac{1}{2}}}}, 
  \quad \delta_n = \Bigo{\frac{1}{n^{2\alpha}}}, 
  \quad \e_n = 
  \begin{cases}
   \Bigo{\frac{1}{n^{2\alpha + 1}}}, & \text{if} \quad i \in \{1,3\} \\ 
   \Bigo{\frac{1}{n^{2\alpha}}}, & \text{if} \quad i = 2.
  \end{cases}
 \end{align*}
 then 
 \begin{align*}
  \Lcal(X^N,\ys) - \Lcal(\xs,Y^N) = 
  \begin{cases}
   \Bigo{N^{-1}} & \text{if} \quad \alpha > 1/2, \\
   \Bigo{\ln^2(N)/N} & \text{if} \quad \alpha = 1/2, \\
   \Bigo{N^{-2\alpha}} & \text{if} \quad \alpha \in (0,\frac{1}{2}). \\
  \end{cases}
 \end{align*}
 \end{corollary}
\begin{proof}
 Under the assumptions of the corollary, if $\alpha > 1/2$, the sequences $\{\|e^n\|\}$, $\{\e_n\}$ and $\{\delta_n\}$ are summable and the error term on the right hand side of equation \eqref{eq:no_acc_Lagrangian} is bounded, hence we obtain a convergence rate of $O(1/N)$. 
 If $\alpha = 1/2$, all errors behave like $O(1/n)$ (note the square root on $\e_n$ for $i \in \{1,3\}$), hence $A_{N,i} = B_{N,i} = O(\ln(N))$, which gives the second assertion. 
 If $0 < \alpha < 1/2$, then by Lemma \ref{lem:behavior} we obtain $A_{N,i}^2 = B_{N,i} = O(N^{1-2 \alpha})$, which gives the third assertion.
 \end{proof}
}

\colblue{ 
Before we establish a convergence result from Theorem \ref{thm:inex_pd_basic}, respectively Corollary \ref{cor:ifo_basic_rate}, we want to comment on some of its particularities. 
Indeed, in most situations this result can be quite weak. 
Recall from \cite{Chambolle2011} that the key ingredient of the exact version of primal-dual algorithms is to have the inequality \eqref{eq:no_acc_Lagrangian} for all $(x,y) \in \X \times \Y$, instead of only for a saddle point $(\xs,\ys)$.
This allows, under some additional assumptions, to both state a rate for the primal and/or dual energy as well as the primal-dual gap and, for infinite dimensional $\X$ and $\Y$, that the cluster points of the ergodic averages $(X^N,Y^N)$ are saddle points and hence a solution to our initial problem.\\
Theorem \ref{thm:inex_pd_basic}, however, due to the necessary bound on the error terms, establishes the desired inequality only for a saddle point $(\xs,\ys)$, which does not allow taking the supremum over $x$ or $y$ on both sides in order to obtain rates for, e.g., the primal energy. 
Instead, Corollary \ref{cor:ifo_basic_rate} states a rate in a more degenerate distance, namely a Bregman distance \cite{Bregman1967,Osher2005}.
This can be seen by rewriting the left hand side of \eqref{eq:no_acc_Lagrangian}, adding $\langle K\xs,\ys \rangle - \langle \xs, K^*\ys \rangle$, 
\begin{align}\label{eq:ifo_bregman_dist}
 \Lcal (X^N,\ys) - \Lcal (\xs, Y^N)
 &= \langle KX^N, \ys \rangle + f(X^N) + g(X^N) - h^*(\ys) \nonumber \\
 &\qquad \quad - (\langle K\xs, Y^N \rangle + f(\xs) + g(\xs) - h^*(Y^N)) \nonumber \\
 &= f(X^N) + g(X^N) \nonumber \\ 
 &\qquad \quad - (f(\xs) + g(\xs) -\langle K^*\ys,X^N-\xs \rangle) \nonumber \\
 &\qquad \quad + h^*(y^N) - (h^*(\ys) + \langle K \xs, y^N - \ys \rangle ). 
\end{align}
Now, realizing that for the saddle point $(\xs,\ys)$ we actually have (by optimality) that 
\begin{align*}
 p = -K^* \ys \in \partial g(\xs) + \nabla f(\xs), \quad q = K\xs \in \partial h^*(\ys),
\end{align*}
we obtain that \eqref{eq:ifo_bregman_dist} is the sum of two Bregman distances, 
\begin{align*}
 \Lcal (X^N,\ys) - \Lcal (\xs, Y^N) = D_{f+g}^p (X^N,\xs) + D_{h^*}^q (Y^N,\ys),
\end{align*}
between the (ergodic) iterates $(X^N,Y^N)$ and the saddle point $(\xs,\ys)$.
Hence, Corollary \ref{cor:ifo_basic_rate} states the rate with respect to this distance. \\
Unfortunately, there exist situations in which a Bregman distance is a very weak measure of convergence. 
As shown in, e.g., \cite{Brinkmann2017}, a vanishing Bregman distance, e.g., 
\begin{align*}
 D_{f+g}^p (x,\xs) + D_{h^*}^q (y,\ys) = 0,
\end{align*}
for some $(x,y) \in \X \times \Y$, in general does not imply that $x = \xs$ or $y = \ys$, neither does it imply that the pair $(x,y)$ is even a saddle point. 
Instead, without any further assumptions on $f,g$ and $h^*$, the set of zeros of a Bregman distance can be arbitrarily large.
In other words, the left-hand side of the inequality in Corollary \ref{cor:ifo_basic_rate} can vanish even though we have {\it not} found a solution to our problem.\\
}
\colblue{
A particularly bad situation is a simple matrix game (cf. \cite{Chambolle2016a}), 
\begin{align*}
 \min_{x \in \Delta_l} \max_{y \in \Delta_k} ~ \langle Ax,y \rangle,  
\end{align*}
where $A \in \R^{k \times l}$ and $\Delta_l,\Delta_k$ denote the unit simplices in $\R^l$ respectively $\R^k$. 
Quite obviously, here we have $f = 0$, $g = \delta_{\Delta_l}$ and $h^* = \delta_{\Delta_k}$, such that we have to compute the Bregman distances with respect to a characteristic function, which can only be zero or infinity. 
Hence, every feasible point causes the Bregman distance to vanish such that a rate in this distance is of no use. 
However, there is a simple workaround in such cases, whenever the primal (or even the dual) variable are restricted to some bounded set $D$, such that $f$ and/or $g$ have bounded domain. 
\colblue{Note that this is a standard assumption also arising in similar works on the topic (e.g. \cite{Nemirovski2004}).}
As can be seen from the proof, one needs a bound on $\|x^n - \xs\|$ in order to control the errors. 
In this case one can estimate $\|x^n - \xs\| \leq \diam(D)$, and following the line of the proof (cf. inequality \eqref{eq:no_acc_proof_main_ineq}) we obtain for all $(x,y) \in \X \times \Y$ that
\begin{align*}
  \Lcal (X^N,y) - \Lcal (x, Y^N) \leq  \frac{1}{N} \bigg(\frac{\|x - x^0 \|^2}{2\tau} &+ \frac{\|y - y^0 \|^2}{2\sigma} \\
  &+  \frac{\diam(D)}{\tau}  A_{N,i} + \frac{1}{\tau} B_{N,i} \bigg).
\end{align*}
Eventually, this again allows deducing a rate for the primal-dual gap (e.g., along the lines in \cite{Chambolle2016a}).
\begin{remark}\textup{
 Even when such a workaround is not possible, there might exist situations where a rate in a Bregman distance is useful.
 For instance, the basis pursuit problem often aims at simply finding the support of the solution, instead of its quantitative values. 
 As shown in \cite{Brinkmann2017} a Bregman distance with respect to the $1$-norm can only vanish if the support of both given arguments agrees. 
 Hence, given a vanishing left-hand side in Corollary \ref{cor:ifo_basic_rate}, we might not have found a saddle point, however, an element with the same support such that our problem is solved.}
\end{remark}
As we have lined out, a rate in a Bregman distance can be difficult to interpret, and it depends on the particular situation whether it is useful or not. 
However, we can at least show the convergence of the iterates in case $\X$ and $\Y$ have finite dimension.
\begin{theorem}\label{thm:convergence_no_acc}
 Let $\X$ and $\Y$ be finite dimensional and let the sequences $(x^n,y^n)$ and $(X^N,Y^N)$ be defined by Theorem \ref{thm:inex_pd_basic}.
 If the partial sums $A_{N,i}$ and $B_{N,i}$ in Theorem \ref{thm:inex_pd_basic} converge, there exists a saddle point $(x^*,y^*) \in \X \times \Y$ such that $x^n \to x^*$ and $y^n \to y^*$.
\end{theorem}
\begin{proof}
 Since by assumption $A_{N,i}$ and $B_{N,i}$ are summable, 
plugging $(\xs,\ys)$ into inequality \eqref{eq:no_acc_proof_main_ineq} and using \eqref{eq:no_acc_bound2} establishes the boundedness of the sequence $(x^n,y^n)$ for all $n \in \N$.
 Hence there exists a subsequence $(x^{n_k},y^{n_k})$ (strongly) converging to a cluster point $(x^*,y^*)$.
 Using $(x,y) = (\xs,\xs)$ in \eqref{eq:no_acc_proof_main_ineq} and the boundedness of the error terms established in \eqref{eq:no_acc_bound2} we also find that $\| x^{n-1} - x^n \| \to 0$ and $\|y^{n-1} - y^n \| \to 0$ (note that this is precisely the reason for the introduction of $\beta$ and the strict inequality in $\tau L_f + \tau \sigma L^2 + \tau \beta L < 1$).
 As a consequence we also have $\|x^{n_k-1} - x^{n_k} \|  \to 0$ and 
 \begin{align*}
  \| x^{n_k-1} - x^* \| \leq \|x^{n_k-1} - x^{n_k} \| + \|x^{n_k} - x^* \| \to 0, \quad k \to \infty,
 \end{align*}
 i.e. also $x^{n_k-1} \to x^*$.
 Let now $T$ denote the primal update of the exact algorithm \eqref{eq:pd_exact}, i.e. $\xh^{n+1} = T(\xh^n)$, and $T_{\e_n}$ denote the primal update of the inexact Algorithm \ref{eq:pd_general}, i.e. $x^{n+1} = T_{\e_n}(x^n)$.
 Then, due to the continuity of $T$, we obtain 
 \begin{align*}
  \| x^* - T(x^*) \| 
  &= \lim_{k \to \infty} \| x^{n_k-1} - T(x^{n_k-1}) \|  \\
  &\leq \lim_{k \to \infty} \left( \| x^{n_k-1} - T_{\e_{n_k}}(x^{n_k-1}) \| + \| T_{\e_{n_k}}(x^{n_k-1}) - T(x^{n_k-1}) \| \right)\\ 
  &\leq \lim_{k \to \infty} \left( \| x^{n_k-1} - x^{n_k} \| + \sqrt{2 \tau \e_{n_k}} \right) = 0.
 \end{align*}
 We apply the same argumentation to $y^n$, which together implies that $(x^*,y^*)$ is a fixed point of the (exact) iteration \ref{eq:pd_exact} and hence a saddle point of our original problem \eqref{eq:saddle_point_problem}.
 We now use $(x,y) = (x^*,y^*)$ in inequality \eqref{eq:no_acc_ineq} and sum from $n = n_k, \dots, N-1$ (leaving out negative terms on the right hand side) to obtain for $N > n_k$
 \begin{align*}
  \frac{1}{2\tau} \| x^* - x^N \|^2 &+ \frac{1}{2\sigma} \|y^* - y^N \|^2 \\ 
  &\leq \langle K(x^N - x^{N-1}),y^*-y^N \rangle - \langle K(x^{n_k} - x^{n_k-1}),y^* - y^{n_k} \rangle \\ 
  &+ \frac{\tau \sigma L^2 + \tau \beta L}{2\tau} \| x^{n_k} - x^{n_k-1} \|^2 + \frac{1}{2\tau} \| x^* - x^{n_k} \|^2 + \frac{1}{2\sigma} \| y^* - y^{n_k} \|^2 \\
  &+ \sum_{n = n_k+1}^N \left( \|e^n \| + \sqrt{(2 \e_n)/\tau} \right) \| x^* - x^n \| + \sum_{n = n_k+1}^N (\e_n + \delta_n).
 \end{align*}
 It remains to notice that since $\|e_n\| \to 0, \e_n \to 0, \delta_n \to 0$ and the above observations, the right hand side tends to zero for $k \to \infty$, which implies that also $x^N \to x^*$ and $y^N \to y^*$ for $N \to \infty$.
\end{proof}
}

\subsection{The convex case: a stronger version}\label{sec:stronger_version}

If we restrict ourselves to type-2 approximations, we can state a stronger version for the reduced problem with $f=0$:
\begin{align}\label{eq:saddle_point_problem_reduced}
 \min_{x \in \X} \max_{y \in \Y} ~ \Lcal(x,y) = \langle y,Kx \rangle + g(x) - h^*(y),
\end{align}
again assuming it has at least one saddle point $(\xs,\ys)$. 
We consider the algorithm
\begin{align} \label{eq:pd_general_prox_2}
  \begin{split}
   y^{n+1} &\approx_2^{\delta_{n+1}} \prox_{\sigma h^*} (y^n + \sigma K(2x^n - x^{n-1})), \\
   x^{n+1} &\approx_2^{\e_{n+1}} \prox_{\tau g} (x^n - \tau K^* y^{n+1})),   
  \end{split}
\end{align}
which is the inexact analog of the basic exact primal-dual algorithm presented in \cite{Chambolle2011}.
\colblue{
Simply speaking, the main difference to the previous section is that, choosing a type-2 approximation and $f = 0$, there are no errors occurring in the \emph{input} of the proximal operators, such that we do not need a bound on $\|x-x^n\|$, which allows us to obtain a rate for the objective for all $(x,y) \in \X \times \Y$ instead of only for a saddle point $(\xs,\ys)$ (cf. Theorem \ref{thm:inex_pd_basic}).
}
Following their line of proof, we can state the following result: 
\begin{theorem}\label{thm:inex_pd_basic_reduced}
 Let $L = \|K\|$ and $\tau, \sigma > 0$ such that $\tau \sigma L^2 < 1$, and let the sequence $(x^n,y^n)$ be defined by algorithm \eqref{eq:pd_general_prox_2}. 
 Then for $X^N := \left( \sum_{n=1}^N x^n \right)/N$, $Y^N := \left( \sum_{n=1}^N y^n \right)/N$
 and every $(x,y) \in \X \times \Y$ we have
 \begin{align}\label{eq:main_res_reduced}
  \Lcal(X^N,y) - \Lcal(x,Y^N) \leq \frac{1}{N} \left( \frac{1}{2\tau} \|x-x^0\|^2 + \frac{1}{2\sigma} \|y-y^0\|^2 + \sum_{n=1}^N (\e_n + \delta_n) \right).
 \end{align}
 Furthermore, if $\e_n = \Bigo{n^{-\alpha}}$ and $\delta_n = \Bigo{n^{-\alpha}}$, then
 \begin{align*}
  \Lcal(X^N,y) - \Lcal(x,Y^N) 
  = \begin{cases}  
   \Bigo{N^{-1}}, &\text{if} \quad \alpha > 1, \\
   \Bigo{\ln(N)/N}, &\text{if} \quad \alpha = 1, \\
   \Bigo{N^{-\alpha}}, &\text{if} \quad \alpha \in (0,1).
  \end{cases}
 \end{align*}
\end{theorem}
\begin{proof}
 The proof can be done exactly along the lines of \cite[Theorem 1]{Chambolle2011} (or along the proof of Theorem \ref{thm:inex_pd_basic}), so we just give the main steps. 
 Letting $f = 0$ and choosing a type-2 approximation gives $L_f = 0$ and lets us drop the term $( \|e^{n+1}\| + \sqrt{(2 \varepsilon_{n+1})/\tau}) \| x - x^{n+1} \| $ in inequality \eqref{eq:basic_inequality_no_acc}.
 This is the essential difference, since we do not have to establish a bound on $\| x - x^{n+1} \|$.
 Choosing $\alpha = \sqrt{\sigma/\tau}$ in Young's inequality and proceeding as before the gives 
 \begin{align}\label{eq:reduced_version_main}
  \sum_{n=1}^N ( &\Lcal(x^n,y) - \Lcal(x,y^n) ) + (1-\tau \sigma L^2) \frac{\|y-y^N\|^2}{2\sigma} + \frac{\|x-x^N\|^2}{2\tau} \nonumber \\
  &+ (1- \sqrt{\tau \sigma} L) \sum_{n=1}^N \frac{\|y^n - y^{n-1}\|^2}{2\sigma} + (1- \sqrt{\tau \sigma} L) \sum_{n=1}^{N-1} \frac{\|x^n - x^{n-1}\|^2}{2\sigma} \nonumber \\
  &\qquad \leq \frac{1}{2\sigma} \|y-y^0\|^2 + \frac{1}{2\tau} \|x-x^0\|^2 + \sum_{n=1}^N (\e_n + \delta_n).
 \end{align}
 The definition of the ergodic sequences and Jensen's inequality yield the assertion.
\end{proof}
As before we can state convergence of the iterates if the errors $\{\e_n\}$ and $\{\delta_n\}$ decay fast enough.
The proof is the same as for Theorem \ref{thm:convergence_no_acc}.
\begin{theorem}
 Let the sequences $(x^n,y^n)$ and $(X^N,Y^N)$ be defined by \eqref{eq:pd_general_prox_2} respectively.
 If the sequences $\{\e_n\}$ and $\{\delta_n\}$ in Theorem \ref{thm:inex_pd_basic_reduced} are summable, then every weak cluster point $(x^*,y^*)$ of $(X^N,Y^N)$ is a saddle point of problem \eqref{eq:saddle_point_problem_reduced}.
 Moreover, if the dimension of $\X$ and $\Y$ is finite, there exists a saddle point $(x^*,y^*) \in \X \times \Y$ such that $x^n \to x^*$ and $y^n \to y^*$.
\end{theorem}
\colblue{
\begin{proof}
  Since by assumption $A_{N,i}$ and $B_{N,i}$ are summable, plugging $(\xs,\ys)$ into equation \eqref{eq:reduced_version_main} establishes the boundedness of the sequence $x^N$ for all $N \in \N$, which also implies the boundedness of the ergodic average $X^N$.
 Note that by the same argumentation as for $x^N$, this also establish a global bound on $y^N$ and $Y^N$.
 Hence there exists a subsequence $(X^{N_k},Y^{N_k})$ weakly converging to a cluster point $(x^*, y^*)$.
 Then, since $f,g$ and $h^*$ are l.s.c. (thus also weakly l.s.c.), we deduce from equation \eqref{eq:main_res_reduced} that, for every fixed $(x,y) \in \X \times \Y$, 
 \begin{align*}
  \Lcal(x^*,y) - \Lcal(x,y^*) 
  &\leq \liminf_{k \to \infty} \Lcal (X^{N_k},y) - \Lcal(x,Y^{N_k}) \\
  &\leq \liminf_{k \to \infty} \frac{1}{N_k} \left( \frac{1}{2\tau} \|x-x^0\|^2 + \frac{1}{2\sigma} \|y-y^0\|^2 + \sum_{n=1}^N (\e_n + \delta_n) \right) \\
  &= 0, 
 \end{align*}
 Taking the supremum over $(x,y)$ then implies that $(x^*,y^*)$ is a saddle point itself and establishes the first assertion.
 The rest of the proof follows analogously to the proof of Theorem \ref{thm:convergence_no_acc}.
\end{proof}
}

\begin{remark}\label{rem:1}\textup{
The main difference between Theorem \ref{thm:inex_pd_basic_reduced} and Theorem \ref{thm:inex_pd_basic} is that inequality \eqref{eq:main_res_reduced} bounds the left hand side for all $x,y \in \X \times \Y$ and not only for a saddle point $(\xs,\ys)$. 
Following \cite[Remark 2]{Chambolle2011} and if $\{\e_n\}, \{\delta_n\}$ are summable we can state the same $\Bigo{1/N}$ convergence of the primal energy, dual energy or the global primal-dual gap under the additional assumption that $h$ has full domain, $g^*$ has full domain or both have full domain.
More precisely, if e.g. $h$ has full domain, then it is classical that $h^*$ is superlinear and that the supremum appearing in the conjugate is attained at some $\tilde{y}^N \in \partial h(KX^N)$, which is uniformly bounded in $N$ due to \eqref{eq:reduced_version_main} (if $(x,y) = (\xs,\ys)$ then $(X^N,Y^N)$ is globally bounded), such that 
\begin{align*}
 \max_{y \in \Y} ~ \Lcal(x^N,y) = ~ \langle \tilde{y}^N, KX^N \rangle -h^*(\yt^N) + g(X^N)  = h(KX^N) + g(X^N).
\end{align*}
Now evoking inequality \eqref{eq:main_res_reduced} and proceeding exactly along \eqref{eq:primal_estimate} we can state that
\begin{alignat*}{2}
 h(KX^N) + g(X^N) &- [h(K\xs) + g(\xs)] \\ 
 &\leq \frac{1}{N} \left( \frac{1}{2\tau} \|\xs-x^0\|^2 + C + \sum_{n=1}^N (\e_n + \delta_n) \right),
\end{alignat*}
with a constant $C$ depending on $x^0$, $y^0$, $h$ and $\|K\|$.
Analogously we can establish the convergence rates for the dual problem and also the global gap.
}
\end{remark}
\begin{remark}\label{rem:mixed_rates}\textup{
If $h^*$ has bounded domain, e.g. if $h$ is a norm, we can even state ``mixed'' rates for the primal energy if the errors are not summable. 
Since in this case $\|y - y^0 \| \leq \diam(\dom h^*)$ we may take the supremum over all
$y \in \dom h^*$ and obtain
\begin{align*}
 h(KX^N) &+ g(X^N) - [h(K\xs) + g(\xs)] \\ 
 &\leq \frac{1}{N} \left( \frac{\|\xs - x^0 \|^2}{2\tau}  + \frac{\diam(\dom h^*)^2}{2 \sigma} + \sum_{n=1}^N (\e_n + \delta_n) \right) = \Bigo{N^{-\alpha}},
\end{align*}
for $\e_n,\delta_n \in \Bigo{n^{-\alpha}}$.
The above result in particular holds for the aforementioned TV-$L^1$ model, which we shall consider in the numerical section.}
\end{remark}


\subsection{The strongly convex case: primal acceleration}
We now turn our focus on possible accelerations of the scheme and consider again the full problem \eqref{eq:saddle_point_problem} with the additional assumption that $g$ is $\gamma$-strongly convex, i.e. for any $x \in \dom \partial g$ 
\begin{align*}
 g(x') \geq g(x) + \langle p, x' - x \rangle + \frac{\gamma}{2} \|x'-x\|^2, \quad \forall p \in \partial g(x), \quad \forall x' \in \X.
\end{align*}
It is a known fact that if $g$ is $\gamma$-strongly convex, its conjugate $g^*$ has  $1/\gamma$-Lipschitz gradient, which guarantees the possibility to accelerate the algorithm.
We mention that we obtain the same result if $f$ (or both $g$ and $f$) are strongly convex, since it is possible to transfer the strong convexity from $f$ to $g$ and vice versa \cite[Section 5]{Chambolle2016}.
Hence for simplicity we focus on the case where $g$ is strongly convex.
Choosing 
\begin{align}
 (\xc,\yc) = (x^{n+1},y^{n+1}),\hspace{1pt} (\xb,\yb) = (x^n,y^n),\hspace{1pt} (\xt,\yt) = (x^n + \theta_n(x^n - x^{n-1}),y^{n+1}),
\end{align}
in algorithm \eqref{eq:pd_general} we define an accelerated inexact primal-dual algorithm: 
\begin{align}
 \begin{split}\label{eq:alg_primal_acc}
 y^{n+1} &\approx_2^{\delta_{n+1}} \prox_{\sigma_n h^*} (y^n + \sigma_n K (x^n + \theta_n (x^n - x^{n-1})) \\
 x^{n+1} &\approx_i^{\varepsilon_{n+1}} \prox_{\tau_n g} ( x^n - \tau_n (K^* y^{n+1} + \nabla f(x^n) + e^{n+1})) \\ 
 \theta_{n+1} &= 1 / \sqrt{1 + \gamma \tau_n}, \quad \tau_{n+1} = \theta_{n+1} \tau_n, \quad \sigma_{n+1} = \sigma_n / \theta_{n+1}.
 \end{split}
\end{align}
We prove the following theorem in Appendix \ref{subsec:inex_pd_acc2}.
\begin{theorem}\label{thm:inex_pd_acc2}
 Let $L = \|K\|$ and $\tau_n, \sigma_n, \theta_n$ such that 
 \begin{align*}
  \tau_n L_f + \tau_n \sigma_n L^2 \leq 1, 
  \qquad \theta_{n+1} \sigma_{n+1} = \sigma_n,
  \qquad (1 + \gamma \tau_n) \tau_{n+1} \theta_{n+1} \geq \tau_n. 
 \end{align*} 
 Let $(x^n,y^n) \in \X \times \Y$ be defined by the above algorithm for $i \in \{1,2,3\}$.
 Then for any saddle point $(\xs, \ys) \in \X \times \Y$ of \eqref{eq:saddle_point_problem} and
 \begin{align*}
  T_N := \sum_{n=1}^N \frac{\sigma_{n-1}}{\sigma_0}, \qquad X^N := \frac{1}{T_N} \sum_{n=1}^N \frac{\sigma_{n-1}}{\sigma_0} x^n, \qquad Y^N := \frac{1}{T_N} \sum_{n=1}^N \frac{\sigma_{n-1}}{\sigma_0} y^n, 
 \end{align*}
 we have that
 \begin{align*}
  T_N ( &\Lcal(X^N,\ys) - \Lcal(\xs, Y^N)) \\ 
  &\leq \frac{1}{2 \sigma_0} \left( \sqrt{\frac{\sigma_0}{\tau_0}} \| \xs - x^0 \| +  \| \ys - y^0 \| + \sqrt{2 B_{N,i}} + 2 \sqrt{\frac{\tau_N}{\sigma_N}} A_{N,i} \right)^2,  
  \intertext {and}
   \frac{\sigma_N}{2\tau_N} &\| \xs - x^N \|^2 \\
   &\leq \frac{1}{2} \left( \sqrt{\frac{\sigma_0}{\tau_0}} \| \xs - x^0 \| +  \| \ys - y^0 \| + \sqrt{2 B_{N,i}} + 2 \sqrt{\frac{\tau_N}{\sigma_N}} A_{N,i} \right)^2,
 \end{align*}
 where 
 \begin{alignat*}{3}
  A_ {N,1} &= \sum_{n=1}^N \sigma_{n-1} \|e^n\| + \sqrt{\frac{2 \sigma_{n-1}^2 \varepsilon_n}{\tau_{n-1}}}, 
  && \qquad B_{N,1} = 2 \sum_{n=1}^N \sigma_{n-1} (\varepsilon_n + \delta_n), \\
  A_{N,2} &= \sum_{n=1}^N \sigma_{n-1} \|e^n\|,
  && \qquad B_{N,2} = 2 \sum_{n=1}^N \sigma_{n-1} (\varepsilon_n + \delta_n), \\ 
  A_{N,3} &= \sum_{n=1}^N \sigma_{n-1} \|e^n\| + \sqrt{\frac{2 \sigma_{n-1}^2 \varepsilon_n}{\tau_{n-1}}},
  && \qquad B_{N,3} = 2 \sum_{n=1}^N \sigma_{n-1} \delta_n.
 \end{alignat*}
\end{theorem}
As a direct consequence of Theorem \ref{thm:inex_pd_acc2} we can state convergence rates of the accelerated algorithm \eqref{eq:alg_primal_acc} in dependence on the errors $\{\|e^n\|\}, \{\delta_n\}$ and $\{\e_n\}$.
\begin{corollary}
 Let $\tau_0 = 1/(2L_f)$, $\sigma_0 = L_f/L^2$ and $\tau_n,\sigma_n$ and $\theta_n$ be given by \eqref{eq:alg_primal_acc}. 
 Let $\alpha > 0$, $i \in \{1,2,3\}$ and 
 \begin{align*}
  \|e^n\| = \Bigo{\frac{1}{n^\alpha}},
  \qquad \delta_n = \Bigo{\frac{1}{n^{2\alpha}}}, 
  \qquad \e_n = \begin{cases}
		  \Bigo{\frac{1}{n^{1+2\alpha}}}, &\text{if } i \in \{1,3\} \\ 
		  \Bigo{\frac{1}{n^{2\alpha}}}, &\text{if } i =2.
	       \end{cases}
 \end{align*}
 Then 
 \begin{align*}
  \Lcal(X^N,\ys) - \Lcal(\xs,Y^N) = 
  \begin{cases}
   \Bigo{N^{-2}} & \text{if} \quad \alpha > 1, \\
   \Bigo{\ln^2(N)/N^2} & \text{if} \quad \alpha = 1, \\
   \Bigo{N^{-2\alpha}} & \text{if} \quad \alpha \in (0,1). \\
  \end{cases}
 \end{align*}
\end{corollary}
\begin{proof}
 In \cite{Chambolle2011} it has been shown that with this choice we have $\tau_n \sim 2/(n\gamma)$. 
 Since the product $\tau_n \sigma_n = \tau_0 \sigma_0 = 1/(2L^2)$ stays constant over the course of the iterations, this implies that $\sigma_n \sim (n\gamma)/(4L^2)$, from which we directly deduce that $T_N \sim (\gamma N^2)/(8L_f)$, hence $T_N = \Bigo{N^2}$.
 Moreover we find that $\sqrt{\tau_N/\sigma_N} \sim (\sqrt{8}L)/(\gamma N)$.
 Now let $i = 1$ and $\alpha \in (0,1)$, then we have
 \begin{align*}
  A_{N,1} = \sum_{n=1}^N \sigma_{n-1} \|e^n\| + \sqrt{\frac{2 \sigma_{n-1}^2 \varepsilon_n}{\tau_{n-1}}} 
  \sim \sum_{n=1}^N \frac{(n-1)\gamma}{4L^2} \|e^n\| + \sqrt{\frac{2\gamma^3((n-1)^3 \e_n)}{32L^4}}
 \end{align*}
 Now by assumption $\|e^n\| = \Bigo{n^{-\alpha}}$ and $\e_n = \Bigo{n^{-(1+2\alpha)}}$ which implies that $A_{N,1} = \Bigo{N^{2-\alpha}}$.
 By analogous reasoning we find $B_{N,1} = \Bigo{N^{2-2\alpha}}$. 
 Summing up we obtain that 
 \begin{align*}
  \frac{\tau_N}{\sigma_N} \frac{A_{N,1}^2}{T_N} = \Bigo{N^{-2\alpha}}, \qquad \text{and} \qquad \frac{B_{N,1}}{T_N} = \Bigo{N^{-2\alpha}}, 
 \end{align*}
 yielding the last row of the assertion.
 For $\alpha = 1$ we see that $\sqrt{\tau_N/\sigma_N} A_{N,1}$ is finite and $B_{N,1} = \Bigo{\log(N)}$, for $\alpha > 1$ also $B_{N,1}$ is summable, implying the other two rates. 
 It remains to notice that the cases $i \in \{2,3\}$ can be obtained as special cases.
\end{proof}

\subsection{The strongly convex case: dual acceleration}\label{sec:dual_acceleration}
This section is devoted to the comparison of inexact primal-dual algorithms and inexact forward-backward splittings established in \cite{Schmidt2011,Villa2013,Aujol2015}, considering the problem 
\begin{align}\label{eq:fista_problem}
 \min_{x \in \X} ~ h(Kx) + g(x),
\end{align}
with $h$ having a $1/\gamma$-Lipschitz gradient and proximable $g$. 
The above mentioned works establish convergence rates for an inexact forward-backward splitting  on this problem, where both the computation of the proximal operator with respect to $g$ and the gradient of $h$ might contain errors (\cite{Villa2013} only considers errors in the proximum).

The corresponding primal-dual formulation of problem \eqref{eq:fista_problem} reads 
\begin{align}\label{eq:saddle_point_problem2}
 \min_{x \in \X} \max_{y \in \Y} ~ \Lcal(x,y) = \langle Kx, y \rangle + g(x) - h^*(y),
\end{align}
where now $h^*$ is $\gamma$-strongly convex. 
%
Hence we know that the algorithm can be accelerated ``\`a la'' \cite{Chambolle2011,Chambolle2016} or as in the previous section, and we shall be able to essentially recover the results on inexact forward-backward splittings/inexact FISTA obtained by \cite{Schmidt2011,Villa2013,Aujol2015}.
Choosing (note $f = 0$ and $e = 0$)
\begin{align}
 (\xc,\yc) = (x^{n+1},y^{n+1}), \hspace{2pt} (\xb,\yb) = (x^n,y^n), \hspace{2pt} (\xt,\yt) = (x^n + \theta_n(x^n - x^{n-1}),y^{n+1}),
\end{align}
in algorithm \eqref{eq:pd_general} we define an accelerated inexact primal-dual algorithm: 
\begin{align*}
 y^{n+1} &\approx_2^{\delta_{n+1}} \prox_{\sigma_n h^*} (y^n + \sigma_n K (x^n + \theta_n (x^n - x^{n-1})) \\
 x^{n+1} &\approx_i^{\varepsilon_{n+1}} \prox_{\tau_n g} ( x^n - \tau_n K^* y^{n+1} ) \\ 
 \theta_{n+1} &= 1 / \sqrt{1 + 2\gamma \sigma_n}, \quad \sigma_{n+1} = \theta_{n+1} \sigma_n, \quad \tau_{n+1} = \tau_n / \theta_{n+1}.
\end{align*}
We prove the following theorem in Appendix \ref{subsec:inex_pd_acc}.
\begin{theorem}\label{thm:inex_pd_acc}
 Let $L = \|K\|$ and $\tau_n, \sigma_n, \theta_n$ such that 
 \begin{align*}
  \tau_n \sigma_n \theta_n^2 L^2 \leq 1, 
  \qquad \theta_{n+1} \tau_{n+1} = \tau_n,
  \qquad (1 + \gamma \sigma_n) \sigma_{n+1} \theta_{n+1} \geq \sigma_n. 
 \end{align*} 
 Let $(x^n,y^n) \in \X \times \Y$ be defined by the above algorithm for $i \in \{1,2,3\}$.
 Then for a saddle point $(\xs, \ys) \in \X \times \Y$ and
 \begin{align*}
  T_N := \sum_{n=1}^N \frac{\tau_{n-1}}{\tau_0}, \qquad X^N := \frac{1}{T_N} \sum_{n=1}^N \frac{\tau_{n-1}}{\tau_0} x^n, \qquad Y^N := \frac{1}{T_N} \sum_{n=1}^N \frac{\tau_{n-1}}{\tau_0} y^n,
 \end{align*}
 we have that
 \begin{align*}
  T_N ( \Lcal(X^N,\ys) &- \Lcal(\xs, Y^N)) \\
  \leq \frac{1}{2 \tau_0} &\left( \| \xs - x^0 \| + \sqrt{\frac{\tau_0}{\sigma_0}} \| \ys - y^0 \| + \sqrt{2 B_{N,i}} + 2A_{N,i} \right)^2,  
  \intertext{and}
   C_N \frac{\tau_N}{\sigma_N} \|\ys - y^N \|^2
   &\leq \left( \| \xs - x^0 \| + \sqrt{\frac{\tau_0}{\sigma_0}} \| \ys - y^0 \| + \sqrt{2 B_{N,i}} + 2A_{N,i} \right)^2,  
 \end{align*}
 where $C_N = 1 - \sigma_N \tau_N \theta_N^2 L^2$ and
 \begin{alignat*}{3}
  A_{N,1} &= \sum_{n=1}^N \sqrt{2 \tau_{n-1} \varepsilon_n},
  && \quad B_{N,1} = 2 \sum_{n=1}^N \tau_{n-1} (\varepsilon_n + \delta_n), \\
  A_{N,2} &= 0, 
  && \quad B_{N,2} = 2 \sum_{n=1}^N \tau_{n-1} (\varepsilon_n + \delta_n), \\
  \quad A_{N,3} &= \sum_{n=1}^N \sqrt{2 \tau_{n-1} \varepsilon_n},
  &&\quad B_{N,3} = 2 \sum_{n=1}^N \tau_{n-1} \delta_n.
 \end{alignat*}
\end{theorem}
We can once more establish convergence rates depending on the decay of the errors.
\begin{corollary}\label{cor:fista_type}
 Let $\tau_0,\sigma_0$ such that $\tau_0 \sigma_0 L^2 = 1$. 
 Let $\alpha > 0$, $i \in \{1,2,3\}$ and 
 \begin{align*}
  \delta_n = \Bigo{\frac{1}{n^{2\alpha}}}, 
  \qquad \e_n = \begin{cases}
		  \Bigo{\frac{1}{n^{1+2\alpha}}}, &\text{if } i \in \{1,3\} \\ 
		  \Bigo{\frac{1}{n^{2\alpha}}}, &\text{if } i =2.
	       \end{cases}
 \end{align*}
 Then 
 \begin{align*}
  \Lcal(X^N,\ys) - \Lcal(\xs,Y^N) = 
  \begin{cases}
   \Bigo{N^{-2}} & \text{if} \quad \alpha > 1, \\
   \Bigo{\ln^2(N)/N^2} & \text{if} \quad \alpha = 1, \\
   \Bigo{N^{-2\alpha}} & \text{if} \quad \alpha \in (0,1). \\
  \end{cases}
 \end{align*}
\end{corollary}
\begin{proof}
 We refer to \cite{Chambolle2011} for a proof that using the step sizes in Theorem \ref{thm:inex_pd_acc}, it can be shown that $\sigma_n \sim 1/(n\gamma)$ and accordingly $\tau_n \sim (n\gamma)/L^2$.
 This directly implies that $T_N \sim (\gamma N^2)/(2L)$.
 Now for $i=1$ and $\alpha \in (0,1)$ we have that 
 \begin{align*}
  A_{N,1} 
  = \sum_{n=1}^N \sqrt{2 \tau_{n-1} \varepsilon_n} 
  \sim \sum_{n=1}^N \sqrt{\frac{2 \gamma (n-1) \varepsilon_n}{L^2}}
  = \frac{\sqrt{2 \gamma}}{L}\sum_{n=1}^N \sqrt{(n-1) \e_n}.
 \end{align*}
 Now by assumption $\e_n = \Bigo{n^{-1-2\alpha}}$, which implies that $\sqrt{(n-1)\e_n} = \Bigo{n^{-\alpha}}$ and we deduce $A_{N,1} = \Bigo{N^{1-\alpha}}$ using Lemma \ref{lem:behavior}.
 By an analogous argumentation 
 \begin{align*}
  B_{N,1} = 2 \sum_{n=1}^N \tau_{n-1} (\varepsilon_n + \delta_n) \sim \frac{2 \gamma}{L^2} \sum_{n=1}^N (n-1) (\varepsilon_n + \delta_n).
 \end{align*}
 Now since $\delta_n = \Bigo{n^{-2 \alpha}}$ we deduce that $n \delta_n = \Bigo{n^{1-2\alpha}}$ and hence $B_{N,1} = \Bigo{N^{2-2 \alpha}}$.
 Using $T_N = \Bigo{N^{2}}$, we find 
 \begin{align*}
  \frac{B_{N,1}}{T_N} = \Bigo{N^{-2\alpha}}, \quad \text{and} \quad \frac{A_{N,1}^2}{T_N} = \Bigo{N^{-2\alpha}}, 
 \end{align*}
 which gives the result for $i=1$ and $\alpha \in (0,1)$.
 Choosing $\alpha > 1$ will yield convergence for $A_{N,1}$ and $B_{N,1}$, which implies the fastest overall convergence rate $\Bigo{1/N^2}$, the case $\alpha = 1$ gives $A_{N,1} = \Bigo{\log(N)}$ and $B_{N,1} = \Bigo{\log(N)}$.
 It remains to note that the results for $i=2,3$ can be obtained as special cases.
\end{proof}
Corollary \ref{cor:fista_type} essentially recovers the results given in \cite{Schmidt2011,Villa2013,Aujol2015}, though the comparison is not exactly straightforward. 
For an optimal $\Bigo{N^{-2}}$ convergence in objective with a type-1 approximation the authors of \cite{Schmidt2011} require $\e_n = \Bigo{1/n^{4+\kappa}}$ for any $\kappa > 0$, for the error $d_n$ in the gradient of $h \circ K$ they need $\|d_n\| = \Bigo{1/n^{4+\kappa}}$. 
Since a type-2 approximation of the proximum is more demanding, the authors of \cite{Villa2013} obtain a weaker dependence of the convergence on the error and only require $\e_n = \Bigo{n^{3+\kappa}}$. 
Note that they only consider the case $d_n = 0$.
The work in \cite{Aujol2015} essentially refines both results under the same assumptions on the errors.
Corollary \ref{cor:fista_type} now states that for an optimal $\Bigo{N^{-2}}$ convergence we require $\e_n = \Bigo{n^{3 + \kappa}}$ in case of a type-1 approximation and $\e_n = \Bigo{n^{2 + \kappa}}$ in case of an error of type-2, which seems to be one order less than the other results.
We do not have a precise mathematical explanation at this point. 
The main difference appears to be the changing step sizes $\tau_n, \sigma_n$ in the proximal operators for the inexact primal-dual algorithm in Theorem \ref{thm:inex_pd_acc}, which behave like $n$ respectively $1/n$, while the step sizes remain fixed for inexact forward-backward.
The numerical section, however, indeed confirms the weaker dependence of the inexact primal-dual algorithm on the errors.

\begin{remark}\textup{
We want to highlight that, in the spirit of Section \ref{sec:stronger_version} it is as well possible to state a stronger version in case the approximations are of type-2 in both the primal and dual proximal point, which then bounds the ``gap'' for all $(x,y) \in \X \times \Y$ instead of for a saddle point $(\xs,\ys)$ in Theorem \ref{thm:inex_pd_acc} (cf. inequality \eqref{eq:intermediate_res}): 
\begin{align*}
 \Lcal (X^N,y) &- \Lcal (x,Y^N) \\ 
 &\leq \frac{1}{T_N} \left( \frac{1}{2\tau_0} \|x - x^0 \|^2 + \frac{1}{2\sigma_0} \| y - y^0 \|^2 + \sum_{n=1}^N \frac{\tau_{n-1}}{\tau_0} (\e_n + \delta_n) \right).
\end{align*}
Under some additional assumptions we can then again derive estimates on the primal energy for every fixed $N \in \N$. 
If again $h$ has full domain, the supremum appearing in the conjugate is attained at some $\tilde{y}^N$ and exactly along \eqref{eq:primal_estimate} we derive
\begin{align*}
 h(KX^N) &+ g(X^N) + f(X^N) - \left[ h(K\xs) + g(\xs) + f(\xs) \right] \\ 
 &\leq \frac{1}{T_N} \left( \frac{1}{2\tau_0} \|\xs - x^0 \|^2 + \frac{1}{2\sigma_0} \| \yt^N - y^0 \|^2 + \sum_{n=1}^N \frac{\tau_{n-1}}{\tau_0} (\e_n + \delta_n) \right).
\end{align*}
In case the errors are summable we again obtain that also $\yt^N$ is globally bounded (cf. Remark \ref{rem:1}) and we obtain convergence in $\Bigo{1/N^2}$. 
If the errors are not summable there is no similar argument to obtain the global boundedness of the $\yt^N$, however at least on a heuristic level one can expect a convergence to $y^*$ at a similar rate as $X^N$.
This is indeed confirmed in the numerical section where we observe the $\Bigo{N^{-2\alpha}}$ decay from Corollary \ref{cor:fista_type} also for the primal objective for nonsummable errors.
}
\end{remark}

\subsection{The smooth case}\label{sec:smooth}
We finally discuss an accelerated primal-dual algorithm if both $g$ and $h^*$ are $\gamma$- respectively $\mu$-strongly convex.
In this setting the primal objective is both smooth and strongly convex, and first-order algorithms can be accelerated to linear convergence.
We consider the algorithm 
\begin{align}
 \begin{split}\label{eq:alg_smooth}
 y^{n+1} &\approx_2^{\delta_{n+1}} \prox_{\sigma h^*} (y^n + \sigma K (x^n + \theta (x^n - x^{n-1})) \\
 x^{n+1} &\approx_i^{\varepsilon_{n+1}} \prox_{\tau g} ( x^n - \tau (K^* y^{n+1} + \nabla f(x^n) + e^{n+1})), \\ 
 \frac{1}{\theta} &= 1+\gamma\tau = 1 + \mu \sigma, \quad \tau L_f + \tau \sigma \theta^2 L^2 \leq 1,
 \end{split}
\end{align}
and prove the following result in Appendix \ref{subsec:smooth}
\begin{theorem}\label{thm:linear_case}
 Let $L = \|K\|$ and $\tau,\sigma, \theta$ such that 
 \begin{align}\label{eq:stepsizes_smooth}
  1+\gamma\tau = 1 + \mu \sigma = \frac{1}{\theta} \quad \text{and} \quad \tau L_f + \tau \sigma \theta^2 L^2 \leq 1.
 \end{align}
 Let $(x^n,y^n) \in \X \times \Y$ be defined by algorithm \eqref{eq:alg_smooth} for $i \in \{1,2,3\}$. 
 Then for the unique saddle-point $(\xs,\ys)$ and
 \begin{align*}
 T_N := \sum_{n=1}^N \frac{1}{\theta^{n-1}} 
 , \quad X^N := \frac{1}{T_N} \sum_{n=1}^N \frac{1}{\theta^{n-1}} x^n, \quad Y^N := \frac{1}{T_N} \sum_{n=1}^N \frac{1}{\theta^{n-1}} y^n
\end{align*}
 we have 
 \begin{align*}
  T_N (\Lcal(X^N,\ys) &- \Lcal(\xs,Y^N)) \\
  &\leq \frac{1}{2\tau} \left( \|\xs -x^0\| + \sqrt{\frac{\tau}{\sigma}} \|\ys-y^0\| + 2 \theta^{\frac{N}{2}} A_{N,i} + \sqrt{2 B_{N,i}}  \right)^2
  \intertext{and}
  \|\xs-x^N\|^2 
  &\leq \theta^N \left( \|\xs -x^0\| + \sqrt{\frac{\tau}{\sigma}} \|\ys-y^0\| + 2 \theta^{\frac{N}{2}} A_N + \sqrt{2 B_N}  \right)^2
 \end{align*}
 where
 \begin{alignat*}{3}
 A_{N,1} &= \sum_{n=1}^N \frac{1}{\theta^{n-1}} (\tau \|e^n\| + \sqrt{2\tau \e_n}),
 &&\quad B_{N,1} = \sum_{n=1}^N \frac{\tau}{\theta^{n-1}} (\e_n + \delta_n), \\
 A_{N,2} &= \sum_{n=1}^N \frac{\tau \|e^n\|}{\theta^{n-1}}, 
 &&\quad B_{N,2} = \sum_{n=1}^N \frac{\tau}{\theta^{n-1}} (\e_n + \delta_n), \\
 A_{N,1} &= \sum_{n=1}^N \frac{1}{\theta^{n-1}} (\tau \|e^n\| + \sqrt{2\tau \e_n}),
 &&\quad B_{N,3} = \sum_{n=1}^N \frac{\tau}{\theta^{n-1}} \delta_n.
\end{alignat*}
 
\end{theorem}

We can now state convergence rates, if the decay of the errors is also linear. 
\begin{corollary}\label{cor:smooth_convergence}
 Let $\alpha > 0$, $i \in \{1,2,3\}$ and for $0<q<1$
  \begin{align*}
  \|e^n\| = \Bigo{\sqrt{q}^n}, 
  \qquad \delta_n = \Bigo{q^n}, 
  \qquad \e_n = \Bigo{q^n}.
 \end{align*}
 Then 
 \begin{align*}
  \Lcal(X^N,\ys) - \Lcal(\xs,Y^N) + \frac{\|\xs-x^N\|^2}{2\tau} 
  = \begin{cases}
     \Bigo{\theta^N}, &\text{ if } \theta > q, \\
     \Bigo{N\theta^N}, &\text{ if } \theta = q, \\
     \Bigo{q^N}, &\text{ if } \theta < q.
    \end{cases}
 \end{align*}
\end{corollary}

\begin{proof}
 It is clear that we need to investigate the decay of the term
 \begin{align*}
  C_{N,i} := \theta^{2N} A_{N,i}^2 + \theta^N B_{N,i} 
 \end{align*}
 to obtain a convergence rate.
 In view of the specific form of $A_{N,i}$ and $B_{N,i}$ and the rate of $\e_n$, $\delta_n$ and $\|e^n\|$ we consider 
 \begin{align}\label{eq:aux_proof_linear}
  \theta^N \sum_{n=1}^N \frac{q^{n-1}}{\theta^{n-1}} = \theta^N \sum_{n=0}^{N-1} \left( \frac{q}{\theta} \right)^n = (\theta^N - q^N) \frac{\theta}{\theta-q}. 
 \end{align}
 Now if $q<\theta<1$, Equation \eqref{eq:aux_proof_linear} implies that $\theta^N B_{N,i} = \Bigo{\theta^N}$. 
 For $A_{N,i}$ we note the factor $\theta^{N}$ is squared, as opposed to the factor of $B_{N,i}$, which implies that the decay of $\|e^n\|$ and $\sqrt{\e_n}$ can be less restrictive for $A_{N,i}$ and implies the square root on the constant $q$ for $\|e^n\|$.
 We have to distinguish whether $\sqrt{q} < \theta$ or $\sqrt{q} > \theta$. 
 In the former case we have by Equation \eqref{eq:aux_proof_linear}, now with $\sqrt{q}$ instead of $q$, that 
 \begin{align*}
  \theta^{2N} A_{N,i}^2 = (\theta^N A_{N,i})^2 = \Bigo{(\theta^N - \sqrt{q}^N)^2} = \Bigo{\theta^{2N}} = \Bigo{\theta^N}, 
 \end{align*}
 while in the latter we obtain $\theta^{2N} A_{N,i}^2 = \Bigo{q^N} = \Bigo{\theta^N}$, which in sum gives $C_{N,i} = \Bigo{\theta^N}$. 
 If $\theta < q<1$, we have by analogous argumentation and \eqref{eq:aux_proof_linear} that $\theta^N B_{N,i} = \Bigo{q^N}$ and since $\theta < q < \sqrt{q} < 1$ also $\theta^{2N} A_{N,i}^2 = \Bigo{q^N}$, which implies $C_{N,i} = \Bigo{q^N}$. 
 For the case $\theta = q$ it is sufficient to notice that \eqref{eq:aux_proof_linear} is in $\Bigo{N\theta^N}$. 
\end{proof}
It remains to give some explicit formulation of the step sizes that fulfill the conditions \eqref{eq:stepsizes_smooth}. 
Solving \eqref{eq:stepsizes_smooth} for $\tau,\sigma$ and $\theta$ gives \cite{Chambolle2016}
\begin{align*}
 \tau &= \frac{1 + \sqrt{1 + 4L^2/(\gamma \mu) + L_f^2 / \gamma^2 + 2L_f / \gamma} - L_f/\gamma}{2L_f + 2L^2/\mu}, \\
 \sigma &= \frac{1 + \sqrt{1 + 4L^2/(\gamma \mu) + L_f^2 / \gamma^2 + 2L_f / \gamma} - L_f/\gamma}{2 L_f \mu / \gamma + 2 L^2 / \gamma}, \\
 \theta &= 1 - \frac{\sqrt{1 + 4L^2/(\gamma \mu) + L_f^2/\gamma^2 + 2 L_f / \gamma} - L_f/\gamma - 1}{2L^2/(\gamma \mu)}.
\end{align*}


\section{Numerical experiments}\label{sec:numerics}
There exists a large variety of interesting optimization problems, e.g. in imaging, that could be investigated in the context of inexact primal-dual algorithms, and even creating numerical examples for all the discussed notions of inexact proxima and different versions of algorithms clearly goes beyond the scope of this paper.
Instead, we want to discuss two different questions on two classical imaging problems and leave further studies to the interested reader. 
The main goal of this section is to confirm numerically, that the convergence rates we proved above are ``sharp'' in some sense, meaning that if the errors are close to the upper bounds we obtain the convergence rates predicted by the theory. 
The second point we want to adress is whether one can actually benefit from the theory and employ different splitting strategies in order to obtain nested algorithms, which can then only be solved in an inexact fashion (cf. \cite{Villa2013}).


We investigate both questions using problems of the form 
\begin{align}\label{eq:sp_two_operators}
 \min_{x \in X} ~ h(K_1x) + g(K_2x) = \min_{x \in X} \max_{y \in Y} ~ \langle y,K_1 x\rangle + g(K_2x) - h^*(y),
\end{align}
$K_1 \colon X \to Y$, $K_2 \colon X \to Z$, where we assume that the proximal operators of both $g$ and $h^*$ (or $g^*$ and $h$ by Moreau's decomposition) have an exact closed form solution.
The right hand side of \eqref{eq:sp_two_operators} leads to a nested inexact primal-dual algorithm 
\begin{align}\label{eq:pd_two_operators}
 y^{n+1} &= \prox_{\sigma h^*}(y^n + \sigma K_1(x^{n+1} + \theta(x^{n+1} - x^n))), \nonumber \\
 x^{n+1} &\approx_2^{\e_{n+1}} \prox_{\tau (g \circ K_2)}(x^n - \tau K_1^*y^{n+1}).
\end{align}
Hence the dual proximal operator can be evaluated exactly (i.e. $\delta_n = 0$), while the inner subproblem has to be computed in an inner loop up to the necessary precision $\e_n$.
We choose the type-2 approximation since in this case, according to Proposition \ref{prop:duality_gap_prox}, the precision of the proximum can be assessed by means of the duality gap.
In order to be able to evaluate the gap, we solve the $1/\tau$-strongly convex dual problem 
\begin{align*}
 \min_{z \in Z} ~ \frac{\tau}{2} \|K_2^*z\|^2 - \langle K_2^* z, y^{n+1} \rangle + g^*(z),
\end{align*}
using FISTA \cite{Beck2009}. 
To distinguish between outer and inner problems for the splittings we denote the iteration number for the outer problem by $n$, while the iteration number of the inner problem is $k$.
In order to achieve the necessary precision, we iterate the proximal problem until the primal-dual gap (cf. also Section \ref{sec:inexact_proximum}) satisfies
\begin{align}\label{eq:gap_constant}
 \Gcal(y^{n+1}-\tau B^*z^k,z^k) \leq C\e_n,
\end{align} 
where $\e_n = \Bigo{1/n^\alpha}$, respectively $\e_n = \Bigo{\theta^n}$ for the last experiment.
\colblue{We vary the parameter $\alpha$ in order to show the effect of the error decay rate on the algorithm (cf. Remark \ref{rem:mixed_rates}).}
While for the asymptotic results we proved in the precious section the constant $C$ of the rate does not matter, it indeed does in practice. 
In order to use Proposition \ref{prop:duality_gap_prox} as a criterion, $C$ should correspond to the ``natural'' size of the duality gap of \eqref{eq:pd_two_operators}.
In order not to choose the constraint too restrictive but still active we follow \cite{Villa2013} and choose $ C = \Gcal(y^0-\tau B^*y^0,0)$, which is the duality gap of the first proximal subproblem for $n = 1$ evaluated at $z = 0$.

For the sake of brevity we discuss only three problems: 
we start with the non-differentiable TV-$L^1$ model for deblurring, a problem which cannot be accelerated,
and continue with ``standard'' TV-$L^2$ deblurring, which also serves as a prototype for a manifold of applications with a general operator instead of a blurring kernel (cf. e.g. \cite{Sawatzky2013,Ehrhardt2016}).
Since in this case the objective is Lipschitz-differentiable, the convex conjugate is strongly convex, which allows to accelerate the algorithm. 
The third problem we investigate is a ``smoothed'' version of the TV-$L^2$ model, which can be accelerated to linear convergence.

We investigate two different setups: 
as already announced above, we want to confirm the convergence rates predicted by the theory numerically. 
We hence require the inexact proximal problem \eqref{eq:pd_two_operators} to be solved with an error close to the accuracy level $\e_n$. 
To achieve this we, where it is necessary, deliberately solve the inner problem suboptimally, meaning that we use a {\it cold} start \colblue{(random initialization of the algorithm)} and reduced step sizes for the inner problem, ensuring that the inner problem is not solved ``accidentially'' at a higher precision. 
We shall see that this is indeed necessary for the slow TV-$L^1$ problem.
In a second setup we investigate whether the obtained error bounds can also be used as a criterion to ensure (optimal) convergence of the nested algorithm \eqref{eq:pd_two_operators}. 
As observed in e.g. \cite{Beck2009b} for the TV-$L^2$ model and the FISTA algorithm, insufficient precision of the inner proximum can cause the algorithm to diverge.
Instead of performing a fixed high number of inner iterations as a remedy, we solve the inner problem only up to precision $\e_n$ in every step, which by the theory ensures that the algorithm converges with the same rate as the decay of the errors. 
We now use the best possible step sizes and a {\it warm} start strategy \colblue{(initialization by the previous solution)} in order to minimize the computational costs of the inner loop and it has already been stated in \cite{Villa2013}, that this strategy significantly speeds up the process.
We use a standard primal-dual reconstruction (PDHG) after $10^5$ iterations as a numerical ``ground truth'' $u^*$ to compute the optimal energy $F^* = F(u^*)$.


\subsection{Nondifferentiable deblurring with the TV-$L^1$ model}\label{sec:tv_l1_deblurring}
\begin{figure}[t]
\begin{center}
 \begin{tabular}{cc}
  \includegraphics[width=0.4\textwidth]{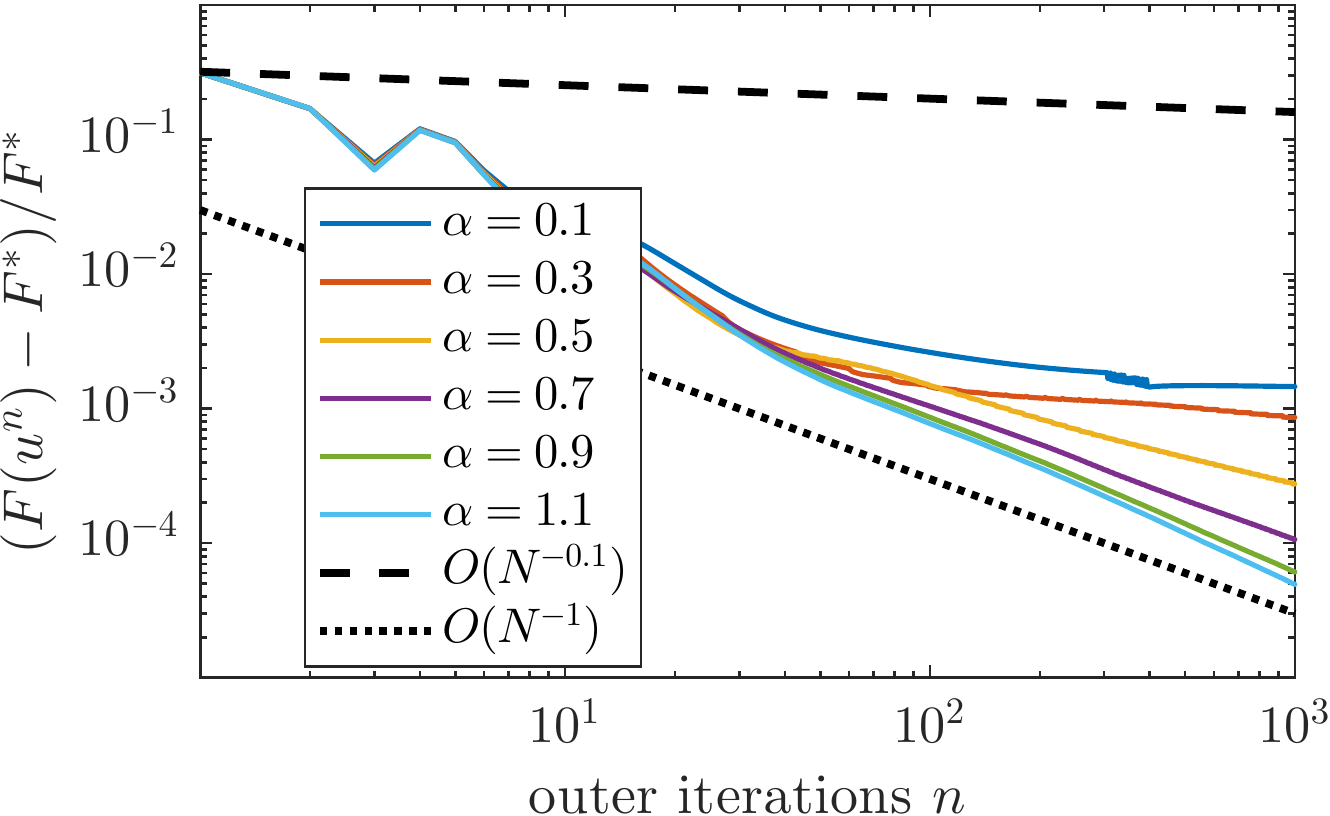} & 
  \includegraphics[width=0.4\textwidth]{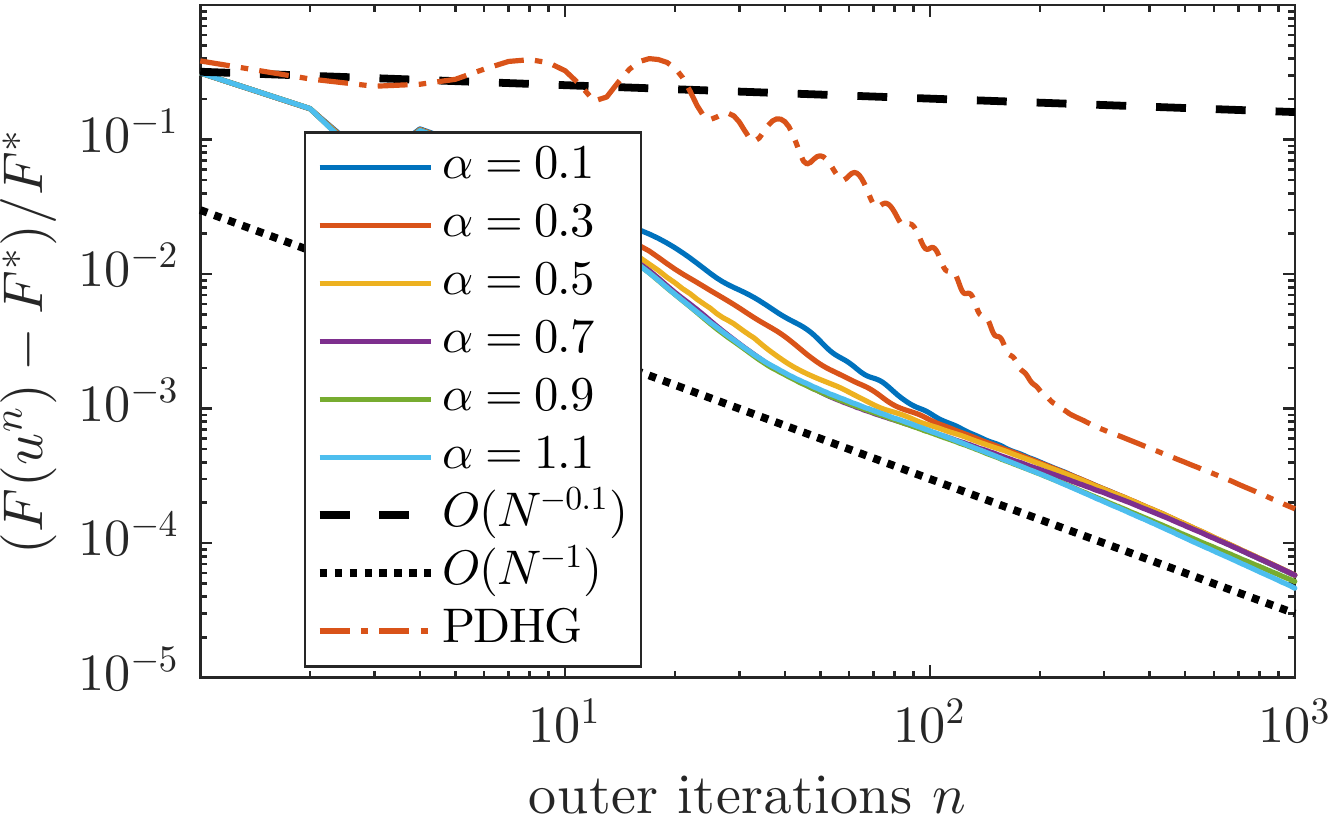} \\
  (a) & (b) \\
  \includegraphics[width=0.4\textwidth]{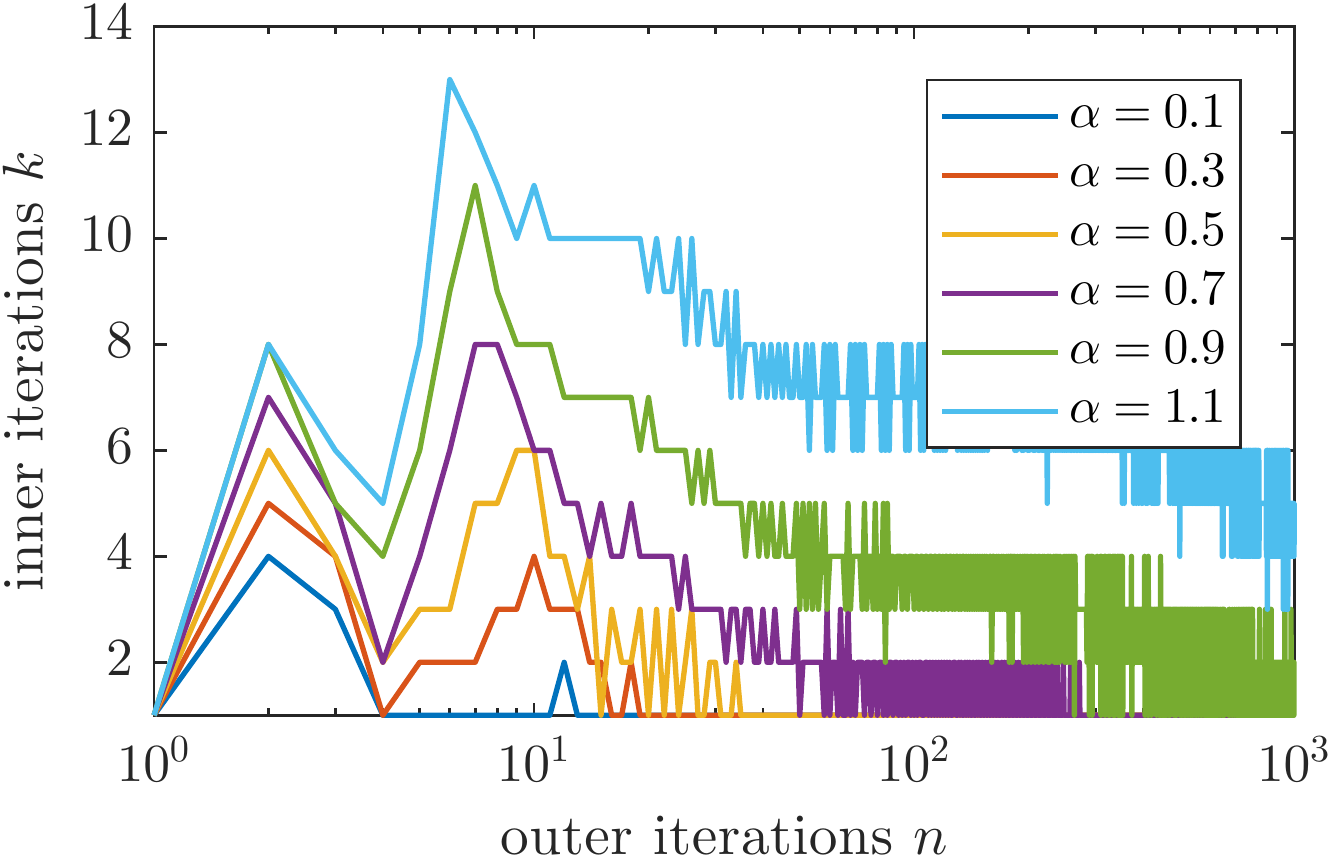} & 
  \includegraphics[width=0.4\textwidth]{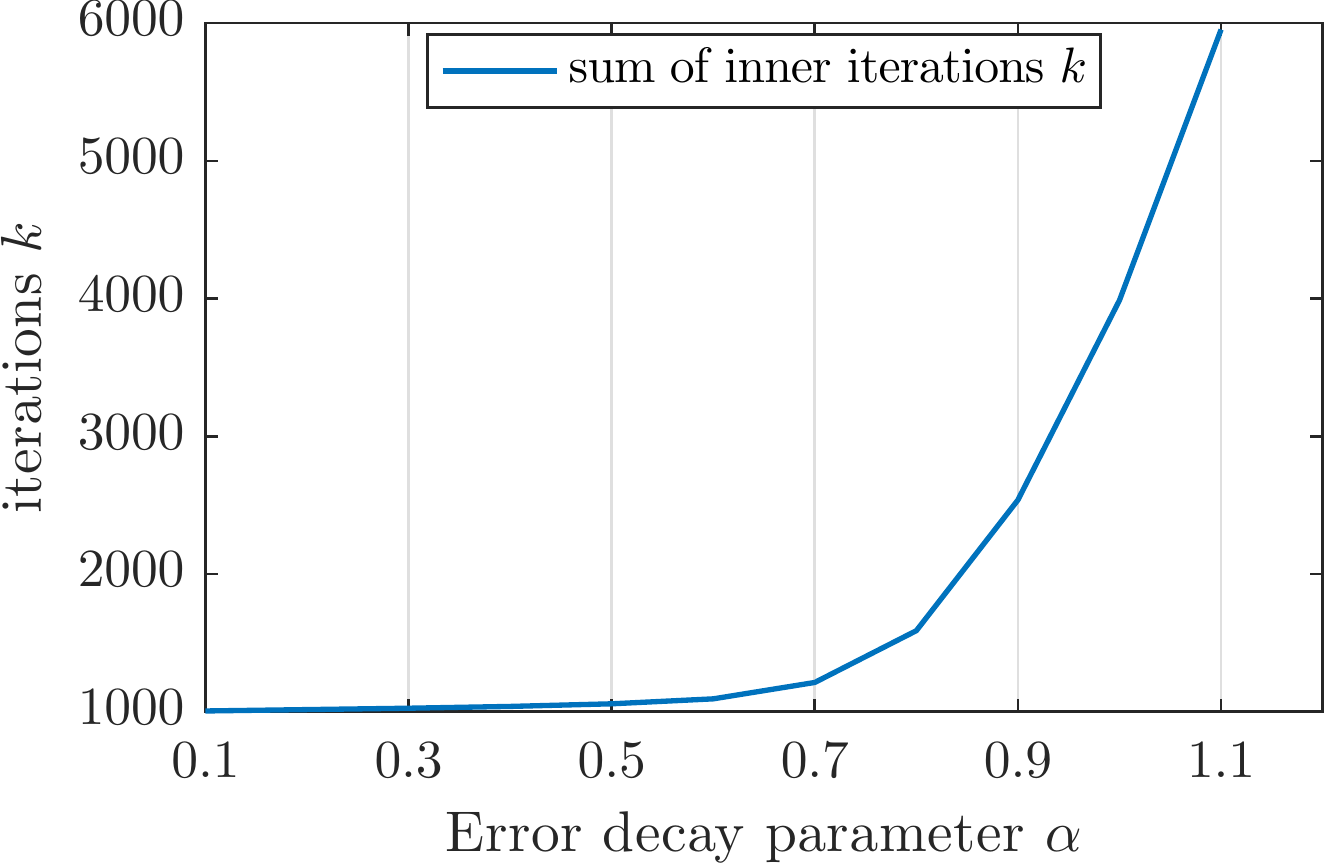} \\
  (c) & (d)
 \end{tabular}
 \caption{{\bf Inexact primal-dual on the TV-$L^1$ problem.} 
 (a) and (b) loglog plots of the relative objective error vs. the outer iteration number for different decay rates $\alpha$ of the errors. (a) cold start, error close to the bound $\Bigo{1/n^\alpha}$, (b) warm start.
 (c) and (d) number of inner iterations respectively sum of inner iterations vs. number of outer iterations for different decay rates $\alpha$.
 One can observe in (a) that the predicted rates in the worst case are attained, while in practice the problem also converges for very few inner iterations (b), (c) and (d).
 }
 \label{fig:l1_tv_stats}
 \end{center}
\end{figure}
In this section we study the numerical solution of the TV-$L^1$ model 
\begin{align}
 u^* \in \arg \min_{u \in X} ~ F(u) = \|Au - f \|_1 + \lambda \| \nabla u \|_1,
\end{align}
with a discrete blurring operator $A \colon X \to X$.
As already lined out in the introduction, there exist a variety of methods to solve the problem (e.g. \cite{Chan2005,Fu2006,Dong2009,Wu2012}), where most of them make use of the fact that the operator $A$ can be written as a convolution. 
We use an easy strategy which does not rely on the structure of the operator and is hence also applicable to operators different from convolutions.
Due to the nondifferentiability of both the data term and regularizer, a very simple approach is to dualize both terms (similar to ADMM \cite{Boyd2011} or 'PD-Split' in \cite{Chambolle2016a}): 
\begin{align*}
 \min_{u \in X} \max_{y_1 \in X,y_2 \in Y} ~ \langle y_1, Au-f \rangle + \langle y_2, \nabla u \rangle - \delta_{P_1} (y_1) - \delta_{P_\lambda} (y_2),
\end{align*}
where $P_\lambda$ denotes the convex set $P_\lambda = \{ x \in X ~|~ \|x \|_\infty \leq \lambda \}$.
One can then employ a standard primal-dual method (PDHG \cite{Chambolle2011}) which reads  
\begin{align*}
 y_1^{n+1} &= \proj_{P_1} (y_1^n + \sigma A (2 u^{n+1} - u^n)), \\
 y_2^{n+1} &= \proj_{P_\lambda} (y_2^n + \sigma \nabla (2 u^{n+1} - u^n)), \\
 u^{n+1} &= u^n - \tau (A^*y_1^{n+1} - \diverg(y_2^{n+1})).
\end{align*}
Unfortunately one can observe that whenever there is no primal term in the formulation of the problem, the energy tends to oscillate and convergence can be quite slow (even though of course in $\Bigo{1/N}$, cf. Figure \ref{fig:l1_tv_stats}(b)). 
As an alternative we propose to split the problem differently and operate on the following primal-dual formulation:
\begin{align*}
 \min_{u\in X} \max_{y \in Y} ~ \langle y,Au-f \rangle - \delta_{P_1}(y) + \lambda \| \nabla u \|_1.
\end{align*}
We employ algorithm \eqref{eq:pd_general_prox_2}, i.e. the non-accelerated basic inexact primal-dual algorithm (iPD) with type-2 errors and obtain
\begin{align}\label{eq:ipd_l1tv} 
 \begin{split}
  y^{n+1} &= \proj_{P_1}(y^n + \sigma A (2u^{n+1} - u^n)), \\
  u^{n+1} &\approx_2^{\e_{n+1}} \arg \min_{u\in X} ~ \frac{1}{2\tau} \|u - (u^n - \tau A^*y^{n+1})\|^2 + \lambda \| \nabla u \|_1.
 \end{split}
\end{align}
Note that the dual proximum in this case can be evaluated error-free.

As a numerical study we perform deblurring on MATLAB's {\it Lily} image in $[0,1]$ with resolution $256 \times 192$, which has been corrupted by a Gaussian blur of approximately 12 pixels full width at half maximum (where we assume a pixel size of 1) and 50 percent salt-and-pepper noise, i.e. 50 percent of the pixels have been randomly set to either $0$ or $1$. 
Furthermore, we performed power iterations to determine the operator norm of $(A, \nabla)$ as $L \approx \sqrt{8}$ and set $\sigma = \tau = 0.99/\sqrt{8}$ for (PDHG). 
For (iPD) $L$ can be determined analytically as $L = \|A \| = 1$, hence $\tau = \sigma = 0.99$ for (iPD).

At first, we want to confirm the convergence rates predicted by the theory numerically. 
One can easily observe that the decay of the relative objective is almost exactly as predicted: 
with higher $\alpha$ it approaches $\Bigo{N^{-1}}$, in fact for summable errors it even seems a little better.
In the second setup we investigate whether the obtained error bounds can also be used as a criterion to ensure  (optimal) convergence of the nested algorithm \eqref{eq:ipd_l1tv}, 
and results for varying parameter $\alpha$ can be found in Figure \ref{fig:l1_tv_stats}(b).
Interestingly for this problem, the error bounds from the theory are indeed too pessimistic or, vice versa, the TV-$L^1$ problem is ``easier'' than expected. 
As can be observed in Figure \ref{fig:l1_tv_stats}(b), the convergence rate for all choices of $\alpha$ tends towards $\Bigo{1/N}$, with slight advantages for higher $\alpha$, while the number of required inner iterations $k$ (Figure \ref{fig:l1_tv_stats}(c) and (d)) to reach the necessary precision is remarkably low. 
In fact, performing just a single inner iteration in every step of the algorithm resulted in a $\Bigo{1/N}$ convergence rate (cf. also Figure \ref{fig:l1_tv_stats}(d)).
The required number of inner iterations even {\it decreases} over the course of the outer iterations which suggests that the dual variable of the inner problem ``converges'' as well.
Note that this does not contradict the theoretical findings of this paper, but the contrary: 
while the first study clearly confirms that in the worst case the proved worst-case estimates are reached, the second implies that in practice one might as well perform by far better.

\subsection{Differentiable deblurring with the TV-$L^2$ model}\label{sec:tv_l2_deblurring}
The second problem we investigate is the TV-$L^2$ model for image deblurring 
\begin{align}\label{eq:tv_deblurring}
 u^* \in \arg \min_{u \in X} ~ \frac{1}{2} \| Au - f \|_2^2 + \lambda \| \nabla u \|_1.
\end{align}
Again, the easiest approach to solve \eqref{eq:tv_deblurring} is to write down a primal-dual formulation 
\begin{align*}
 \min_{u \in X} \max_{y_1 \in X,y_2 \in Y} ~ \langle y_1, Au-f \rangle + \langle y_2, \nabla u \rangle - \frac{1}{2} \|y_1\|^2 - \delta_{P_\lambda} (y_2).
\end{align*}
Since the above problem is {\it not} strongly convex in $y_2$ it cannot be accelerated, so a basic primal-dual algorithm \cite{Chambolle2011} (PDHG) for the solution reads
\begin{align*}
 y_1^{n+1} &= \frac{y_1^n + \sigma (A (2 u^{n+1} - u^n) - f)}{1 + \sigma}, \\
 y_2^{n+1} &= \proj_{P_\lambda} (y_2^n + \sigma \nabla (2 u^{n+1} - u^n) ), \\
 u^{n+1} &= u^n - \tau (A^* y_1^{n+1} - \diverg(y_2^{n+1}) ).
\end{align*}
We remark that, due to the special relation between the Fourier transform and a convolution, the same problem can be solved without dualizing the data term, since the primal proximal operator admits a closed form solution \cite{Chambolle2011}. 
The problem however stays non-strongly convex, and in order to keep this a general prototype for $L^2$-type problems, we do not use this formulation.

The inexact approach instead operates on a different primal-dual formulation given by
\begin{align*}
 \min_{u \in X} \max_{y \in X} ~ \langle y, Au-f \rangle - \frac{1}{2} \|y\|^2 + \lambda \| \nabla u \|_1,
\end{align*}
which is now $1$-strongly convex in $y$ and can be accelerated. 
Using the inexact primal-dual algorithm from Section \ref{sec:dual_acceleration} leads to 
\begin{align*}
 y^{n+1} &= (y^n + \sigma_n (A (u^{n+1} + \theta_n (u^{n+1} - u^n)) -f))/(1 + \sigma_n), \\
 u^{n+1} &\approx_2^{\e_{n+1}} \arg \min_{u \in \X} ~ \frac{1}{2\tau_n} \|u - (u^n - \tau_n A^*y^{n+1}) \|^2 + \| \nabla u \|_1,
\end{align*}
with $\tau_n,\sigma_n,\theta_n$ as given in Theorem \ref{thm:inex_pd_acc}.
We again perform deblurring on MATLAB's {\it Lily} image in $[0,1]$ with resolution $256 \times 192$, which has been corrupted by a Gaussian blur of approximately 12 pixels full width at half maximum (where we assume a pixel size of 1), and in this case Gaussian noise with standard deviation $s = 0.01$ and zero mean. 
We allow errors of the size $\e_n = C/n^{-2\alpha}$ for $\alpha \in (0,1)$, which by Corollary \ref{cor:fista_type} should result in a $\Bigo{N^{-2\alpha}}$ rate respectively $\Bigo{N^{-2}}$ for $\alpha > 1$.
The results can be found in Figure \ref{fig:l2_tv_stats}.
In contrast to the TV-$L^1$ problem, in this experiment it was not necessary to employ a {\it cold} start strategy and reduced step sizes for the inner problem in order to obtain the worst case rates.
Instead also for a warm start and best possible step sizes for the inner problem the bounds for the gap \eqref{eq:gap_constant} were active for all choices of $\alpha$.
Figure \ref{fig:l2_tv_stats} shows the error in relative objective for the ergodic sequence $U^N$ (a) and the iterates $u^n$ (b) for increasing $\alpha$. 
It can be observed that the rate is almost exactly the one predicted, while the iterates themselves even decay a little faster than the ergodic sequence. 
The amount of inner iterations necessary to obtain the required precision of the proximum is unsurprisingly higher than in the non-accelerated case, though they stay reasonable for rather low outer iteration numbers.

\begin{figure}[t]
\begin{center}
 \begin{tabular}{cc}
  \includegraphics[width=0.4\textwidth]{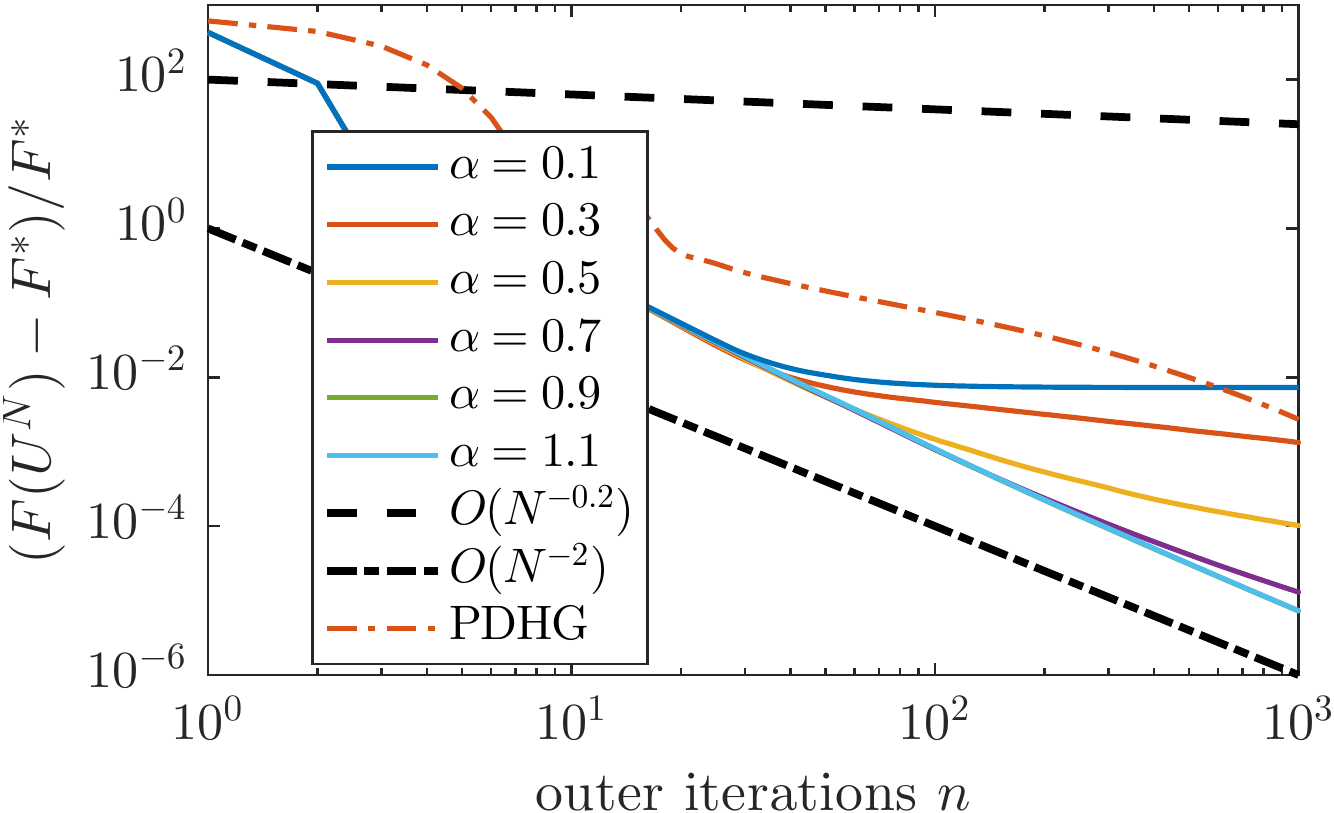} & 
  \includegraphics[width=0.4\textwidth]{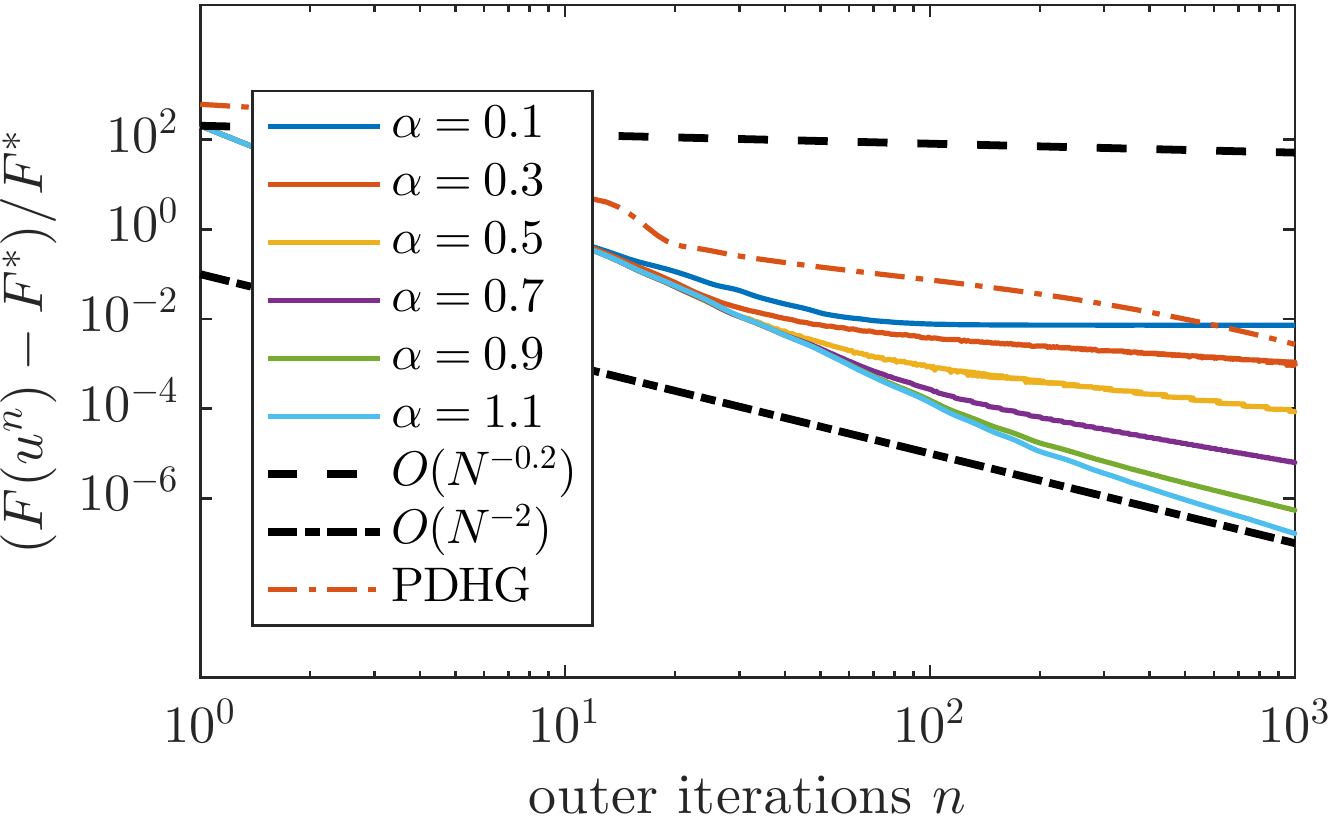} \\
  (a) & (b) \\
  \includegraphics[width=0.4\textwidth]{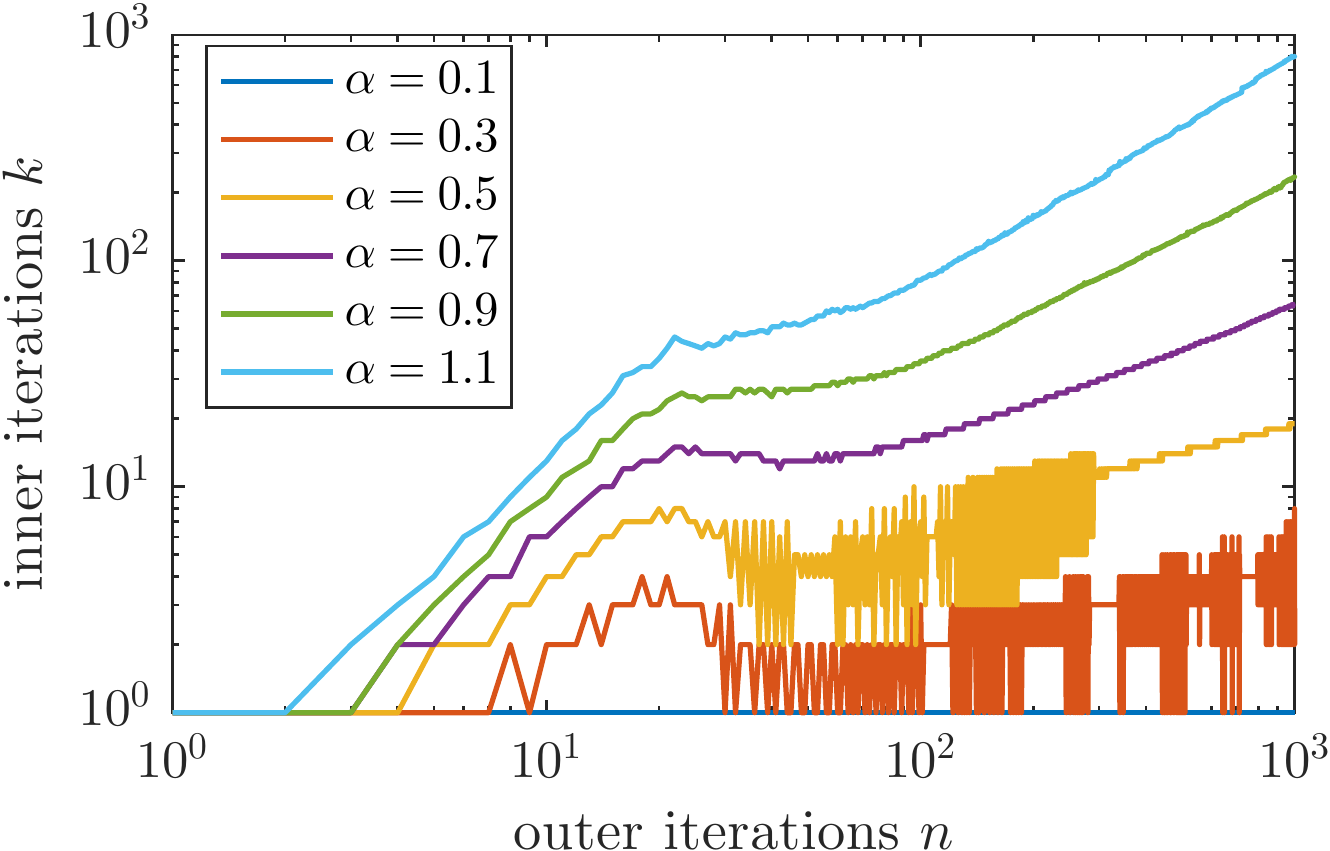} & 
  \includegraphics[width=0.4\textwidth]{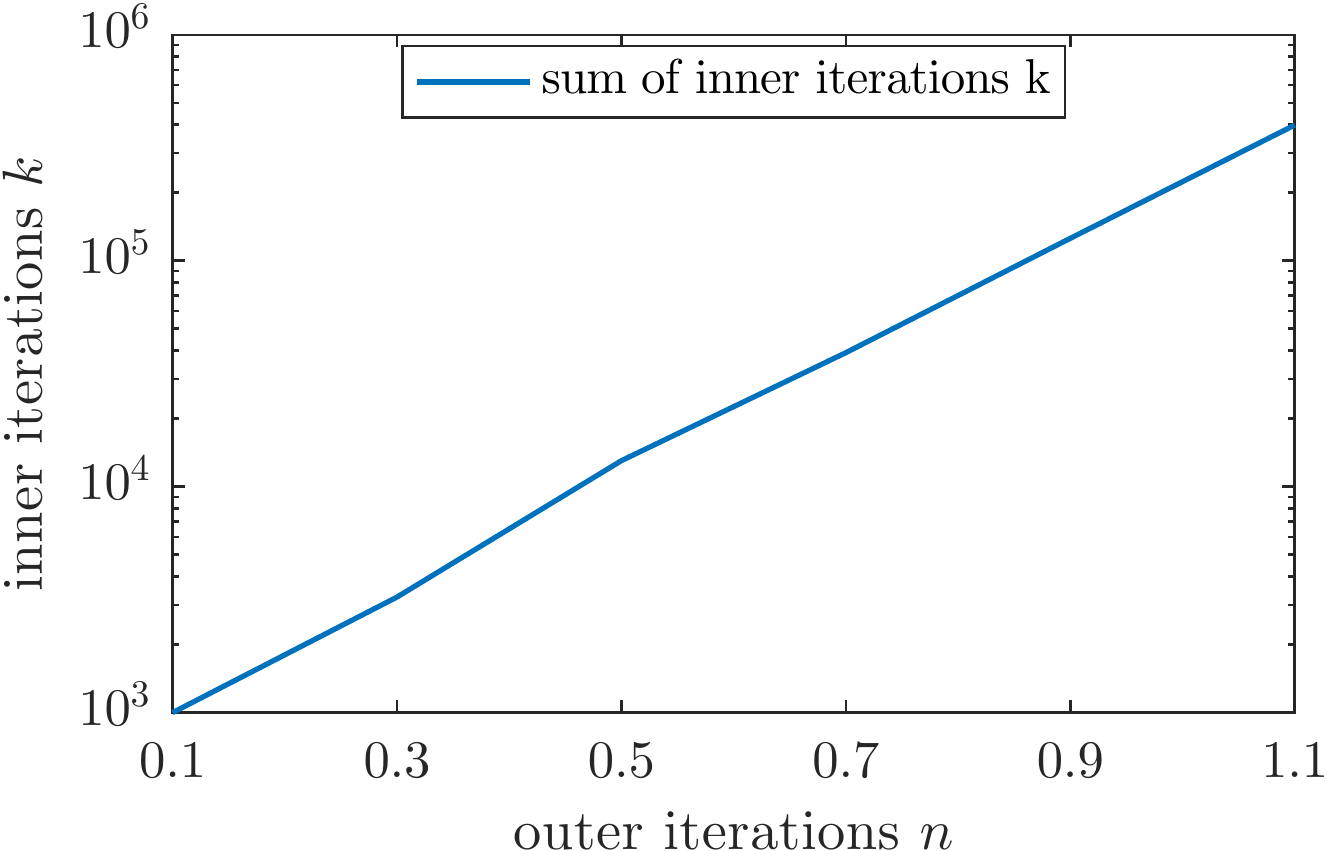} \\
 \end{tabular}
 \caption{{\bf Inexact primal-dual on the TV-$L^2$ problem.} 
 (a) and (b) loglog plots of the relative objective error vs. the outer iteration number for different precisions $C/n^{-2\alpha}$ of the errors.
 (a) ergodic sequence, (b) iterates.
 (c) and (d) number of inner iterations respectively sum of inner iterations vs. number of outer iterations for different decay rates $\alpha$.
 One can observe that the predicted rate of $\Bigo{N^{-2\alpha}}$ is attained both for the ergodic sequence and the single iterates, exactly reflecting the influence of the errors/imprecision.
 }
 \label{fig:l2_tv_stats}
 \end{center}
\end{figure}

\subsection{Smooth deblurring with the TV-$L^2$ model}\label{sec:tv_l2_deblurring_smooth}
The last problem we consider is a smoothed version of the TV-$L^2$ model from the previous experiments: 
\begin{align}\label{eq:tv_deblurring_smooth}
 u^* \in \arg \min_{u \in X} ~ \frac{1}{2} \| Au - f \|_2^2 + \lambda \| \nabla u \|_1 + \frac{\gamma}{2} \|u\|^2,
\end{align}
for small $\gamma$, with primal-dual formulation 
\begin{align}\label{eq:pd_not_acceleratable}
 \min_{u \in X} \max_{y_1 \in X,y_2 \in Y} ~ \langle y_1, Au-f \rangle + \langle y_2, \nabla u \rangle - \frac{1}{2} \|y_1\|^2 - \delta_{P_\lambda} (y_2) + \frac{\gamma}{2} \|u\|^2.
\end{align}
Since the above problem is $\gamma$-strongly convex in $u$ (note that it is also $L_f = \gamma$-Lipschitz differentiable in the primal variable), a possible accelerated primal-dual algorithm \cite{Chambolle2016} (PDHGacc) for the solution reads
\begin{align*}
 y_1^{n+1} &= \frac{y_1^n + \sigma_n (A (u^{n+1} + \theta_n (u^{n+1} - u^n)) - f)}{1 + \sigma_n}, \\
 y_2^{n+1} &= \proj_{P_\lambda} (y_2^n + \sigma_n \nabla (u^{n+1} + \theta_n (u^{n+1} - u^n)) ), \\
 u^{n+1} &= (1-\tau_n \gamma) u^n - \tau_n (A^* y_1^{n+1} - \diverg(y_2^{n+1}) ),
\end{align*}
with $\tau_n, \sigma_n,\theta_n$ given by Theorem \ref{thm:inex_pd_acc2} (see also \cite{Chambolle2016}). 
We choose $\tau_0 = 0.99 / L$, $\sigma_0 = (1 - \tau_0 L_f) / \tau_0 L^2$ such that $\tau_0 L_f + \tau_0 \sigma_0 L^2 = 1$ as required, with $L = \| (A,\nabla) \| \approx \sqrt{8}$ (see also the previous section).
We remark that the primal term involving $\gamma$ could also be handled implicitly, leading to a linear proximal step instead of the explicit evaluation of the gradient which, however, did not substantially affect the results.
In the spirit of the previous experiments we employ a different splitting on this problem:
\begin{align}\label{eq:pd_accaleratable}
 \min_{u \in X} \max_{y \in Y} ~ \langle y,Au -f \rangle -\frac{1}{2} \|y\|^2 + \lambda \| \nabla u \|_1 + \frac{\gamma}{2} \|u\|^2.
\end{align}
The benefit is that even for small $\gamma$ this problem is $\gamma$-strongly convex in the primal {\it and} $1$-strongly convex in the dual variable and hence can be accelerated to linear convergence, which provides a huge boost in performance.
Note that the same is not possible in formulation \eqref{eq:pd_not_acceleratable}, since the problem is not strongly convex in $y_2$.
We can handle the smooth primal term in \eqref{eq:pd_accaleratable} explicitly such that the associated inexact primal-dual algorithm (iPD) from Section \ref{sec:smooth} reads 
\begin{align*}
 y^{n+1} &= (y^n + \sigma (A (u^{n+1} + \theta (u^{n+1} - u^n)) -f))/(1 + \sigma), \\
 u^{n+1} &\approx_2^{\e_{n+1}} \arg \min_{u \in \X} ~ \frac{1}{2\tau} \|u - [(1 - \tau \gamma) u^n - \tau A^*y^{n+1}] \|^2 + \| \nabla u \|_1.
\end{align*}
with $\tau,\sigma,\theta$ defined at the end of Section \ref{sec:smooth}.
In this case we have $\gamma = L_f$, such that the formulas simplify to 
\begin{align*}
 \tau = \frac{\sqrt{4 + 4L^2/(\gamma \mu)}}{2 \gamma + 2 L^2/\mu}, \quad 
 \sigma = \frac{\sqrt{4 + 4L^2/(\gamma \mu)}}{2 \mu + 2L^2 /\gamma}, \quad 
 \theta = 1 - \frac{\sqrt{4 + 4L^2/(\gamma \mu)} -2}{2L^2 / (\gamma \mu)}.
\end{align*}
We revisit the experimental setting from Section \ref{sec:tv_l2_deblurring_smooth}, such that $L = \|A\|= 1$, $\lambda = 0.01$ and choose $\gamma = 1e-3$. 
With this size of $\gamma$ the results were barely distinguishable from the results of the non-smoothed model from Section \ref{sec:tv_l2_deblurring}.
This leads to $\theta \approx 0.96$ for the constant of the linear convergence.
Figure \ref{fig:l2_tv_smooth_stats} shows the results for (PDHGacc) and (iPD) using an error decay rate of $q = 0.9$, i.e. according to Corollary \ref{cor:smooth_convergence} we expect a linear convergence with constant $\theta > q$, which is indeed the case.
One can observe that already after 250 iterations (iPD) reaches a relative objective error of $1\mathrm{e-}10$, while the accelerated PD version has barely reached $1\mathrm{e-}2$. 
It should however be mentioned that also (PDHGacc) reaches the $\Bigo{N^{-2}}$ rate soon after these 250 iterations.
Figure \ref{fig:l2_tv_smooth_stats}(c) shows the price we pay for the inner loop, i.e. the number of inner iterations which is necessary over the course of the 250 outer iterations. 
As one expects for linear convergence, the number of inner iterations explodes for high outer iteration numbers, which substantially slows down the algorithm. 
However, the algorithm reaches an error of $1\mathrm{e-}6$ in relative objective already after approximately 100 iterations, in which case the number of inner iterations is still remarkably low (around 10-20), which makes the approach viable in practice. 
This is in particular interesting for problems with a very costly operator $A$, where the tradeoff between outer and inner iterations is high.

 \begin{figure}[t]
 \begin{center}
 \begin{tabular}{ccc}
  \includegraphics[width=0.3\textwidth]{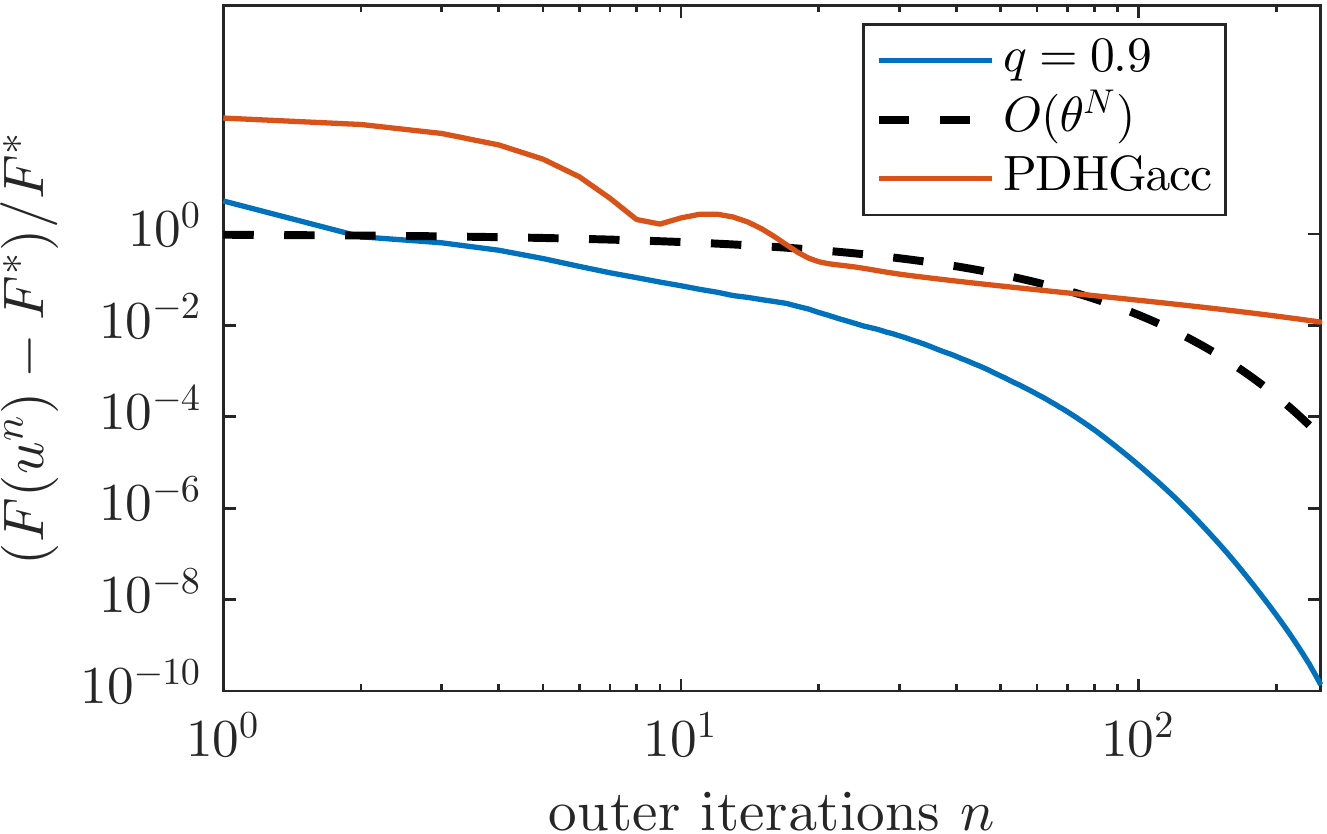} & 
  \includegraphics[width=0.3\textwidth]{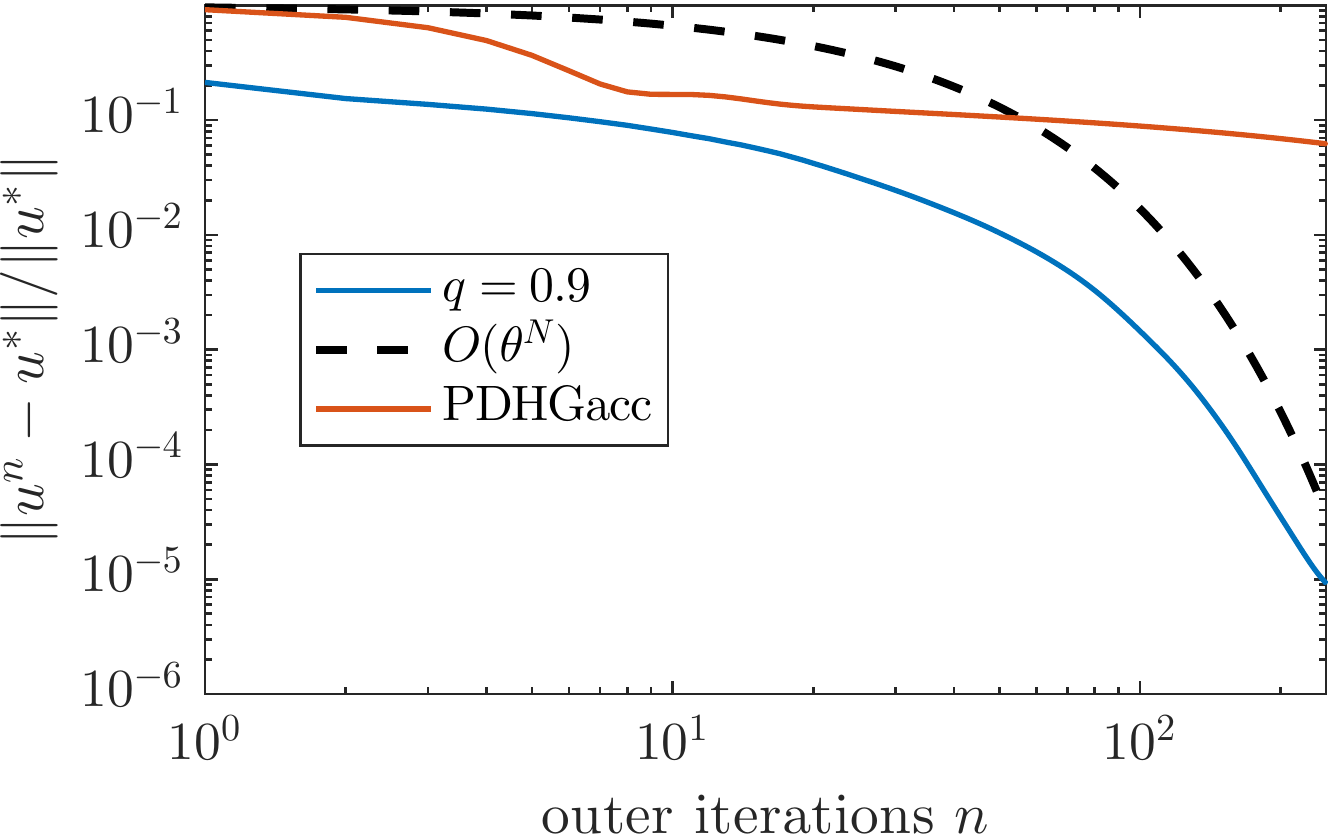} & 
  \includegraphics[width=0.3\textwidth]{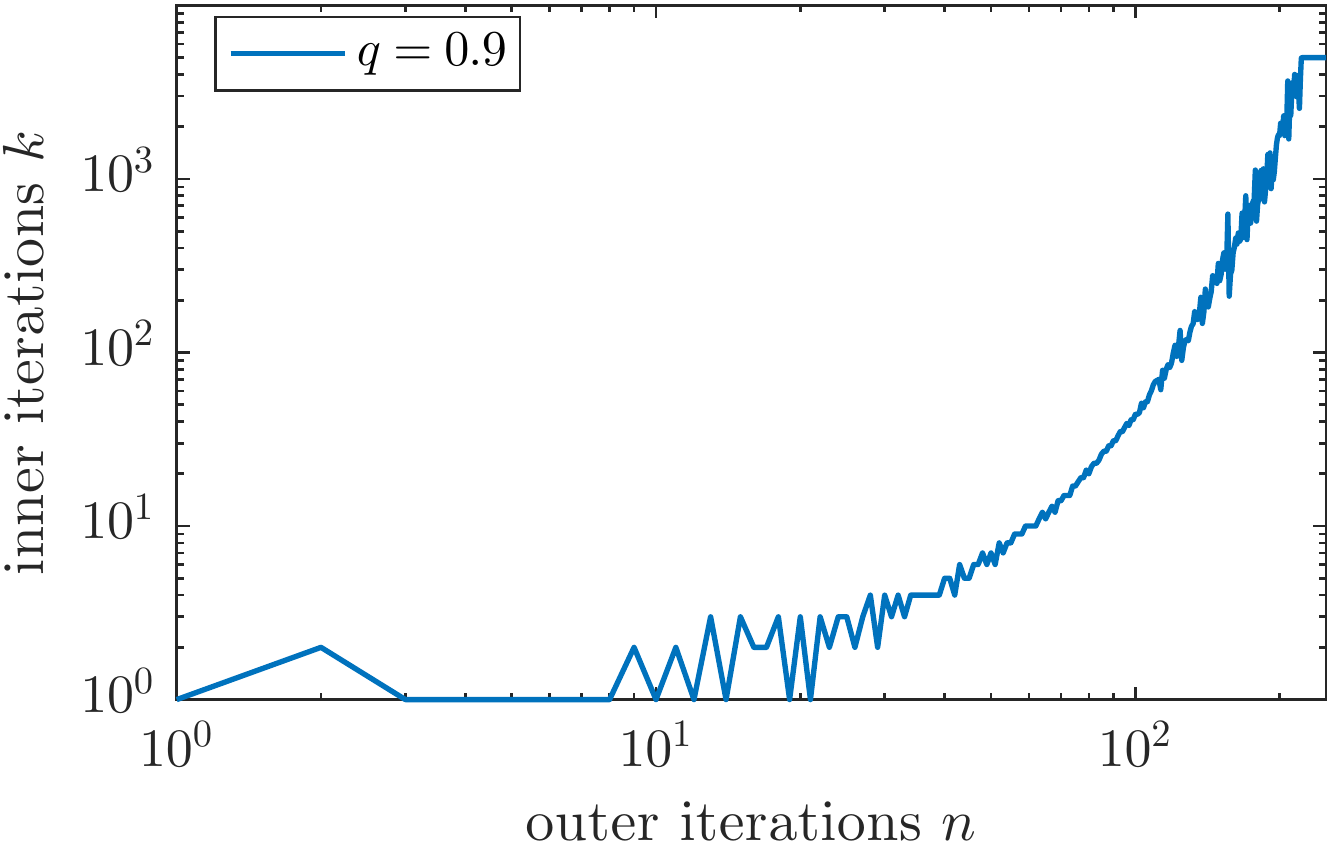} \\
  (a) & (b) & (c) \\
 \end{tabular}
 \caption{{\bf Inexact primal-dual on the smoothed TV-$L^2$ problem.} (a) and (b) loglog plots of the relative objective error respectively relative error in norm vs. the outer iteration numbers for accelerated primal-dual (PDHGacc) and inexact primal-dual (iPD) for $q = 0.9$, (c) loglog plot of the inner iteration number vs. outer iteration number for $q = 0.9$.
 One can observe that the predicted convergence rate of $\Bigo{\theta^N}$ is exactly attained, while for lower outer iteration numbers the necessary amount of inner iterations stays reasonably low. 
 }
 \label{fig:l2_tv_smooth_stats}
 \end{center}
\end{figure}

\section{Conclusion and Outlook}\label{sec:conclusion}
In this paper we investigated the convergence of the class of primal-dual algorithms developed in \cite{Pock2009,Chambolle2011,Chambolle2016} under the presence of errors occurring in the computation of the proximal points and/or gradients. 
Following \cite{Schmidt2011,Villa2013,Aujol2015} we studied several types of errors and showed that under a sufficiently fast decay of these errors we can establish the same convergence rates as for the error-free algorithms. 
More precisely we proved the (optimal) $\Bigo{1/N}$ convergence to a saddle-point in finite dimensions for the class of non-smooth problems considered in this paper, and proved a $\Bigo{1/N^2}$ or even linear $\Bigo{\theta^N}$ convergence rate for partly smooth respectively entirely smooth problems.
We demonstrated both the performance and the practical use of the approach on the example of nested algorithms, which can be used to split the global objective more efficiently in many situations. 
A particular example is the nondifferentiable TV-$L^1$ model which can be very easily solved by our approach. 
A few questions remain open for the future: 
A very practical one is whether one can use the idea of nested algorithms to (heuristically) speed up the convergence of real life problems which are not possible to accelerate, such as TV-type methods in medical imaging. 
As demonstrated in the numerical section, using an inexact primal-dual algorithm one can often ``introduce'' strong convexity by splitting the problem differently and hence obtain the possibility to accelerate.  
This can in particular be interesting for problems with operators of very different costs, where the trade-off between inner and outer iterations is high and hence a lot of inner iterations are still feasible.
Following the same line, it would furthermore be interesting to combine the convergence results for inexact algorithms with stochastic approaches as done in \cite{Chambolle2017}, which are also designed to speed up the convergence for this particular situation, which could provide an additional boost.
Another point to investigate is whether one can combine the inexact approach with linesearch and variable metric strategies similar to \cite{Bonettini2016}.

\appendix

\section{Appendix}
In the Appendix we provide two technical results and the proofs for all the accelerated versions of the algorithm, since they basically follow the same line as the basic proof.
\subsection{Two technical lemmas}\label{sec:technical_lemmas}
The following lemma is taken from \cite{Schmidt2011}.
\begin{lemma}[\cite{Schmidt2011}]\label{lem:technical_result}
 Assume that the sequence $\{ u_N \}$ is nonnegative and satisfies the recursion
 \begin{align*}
  u_N^2 \leq S_N + \sum_{n=1}^N \lambda_n u_n 
 \end{align*}
 for all $N \geq 1$, where $\{S_N \}$ is an increasing sequence, $S_0 \geq u_0^2$, and $\lambda_n \geq 0$ for all $n \geq 0$.
 Then for all $N \geq 1$
 \begin{align*}
  u_N \leq \frac{1}{2} \sum_{n=1}^N \lambda_n + \left( S_N + \left( \frac{1}{2} \sum_{n=1}^N \lambda_n \right)^2 \right)^{\frac{1}{2}}.
 \end{align*}
\end{lemma}

\begin{lemma}\label{lem:behavior}
 For $\alpha > -1$ let $s_N := \sum_{n=1}^N n^\alpha$.
 Then 
 \begin{align*}
  s_N = \Bigo{N^{1+\alpha}}.
 \end{align*}
\end{lemma}
\begin{proof}
 Let $\alpha \in (-1,0)$ and $n \geq 1$.
 Then by the monotonicity of $x \mapsto x^\alpha$ we have for all $n-1 \leq x \leq n$ that $x^\alpha \geq n^\alpha$.
 Integrating both sides of the inequality from $n-1$ to $n$ and summing from $n=1,\dots,N$ we obtain
 \begin{align*}
  s_N \leq \int_0^N x^\alpha \dx.
 \end{align*}
 We proceed analogously for $n \leq x \leq n+1$ to obtain 
 \begin{align*}
  \int_1^{N+1} x^\alpha \dx \leq s_N.
 \end{align*}
 By computing both integrals we hence find
 \begin{align*}
  \frac{1}{1 + \alpha} \left[ (N+1)^{1+\alpha} - 1 \right] 
  = \int_1^{N+1} x^\alpha \dx 
  \leq s_N 
  \leq \int_0^N x^\alpha \dx 
  = \frac{1}{1+\alpha} N^{1+\alpha},
 \end{align*}
 which implies $s_N = O(N^{1+ \alpha})$.
 The proof for $\alpha > 0$ follows the same idea. 
 Now for every $n-1 \leq x \leq n$ we have that $(n-1)^\alpha \leq x^\alpha \leq n^\alpha$.
 Integrating the inequality from $n-1$ to $n$ and summing from $n=1, \dots, N$ we obtain 
 \begin{align*}
  s_{N-1} 
  = \sum_{n=1}^{N-1} n^\alpha 
  \leq \int_0^N x^\alpha \dx 
  = \frac{1}{1 + \alpha} N^{1+\alpha} 
  \leq \sum_{n=1}^N n^\alpha = s_N.
 \end{align*}
 Furthermore $s_{N-1} = s_N - N^\alpha$, so for every $N\geq1$ 
 \begin{align*}
  \frac{1}{1+\alpha} N^{1+\alpha} \leq s_N \leq \frac{1}{1+\alpha} N^{1 + \alpha} + N^\alpha,
 \end{align*}
 from which we deduce that $s_N = O(N^{1+ \alpha})$.
\end{proof}

\subsection{Type-0 approximations}\label{sec:type_0}
It is interesting to consider the notion of a type-0 approximation (cf. Definition \ref{def:type1}) as well, since it seems to be the most intuitive one (the authors of \cite{Aujol2015} mention it but do not explicitly handle the situation).
The problem however is that neither the inexact proximal point needs to be feasible, nor do we have an equivalent definition of a type-0 approximation in terms of an ($\e$-) subdifferential.
For simplicity we briefly outline a possible strategy on the reduced problem 
\begin{align*}
 \min_{x \in \X} \max_{y \in \Y} ~ \langle y,Kx \rangle + g(x) - h^*(y)
\end{align*}
and condsider the algorithm
 \begin{align} \label{eq:pd_general_prox_0}
  \begin{split}
   \xc &\approx_0^\varepsilon \prox_{\tau g} (\xb - \tau K^*\yt)), \\
   \yh &\approx_2^\delta \prox_{\sigma h^*} (\yb + \sigma K\xt), 
  \end{split}
\end{align}
where again $(\xc,\yc)$ are the erroneus proximal points and $(\xt,\yt)$ and $(\xb,\yb)$ are the previous points.
A possible way to deal with the type-0 approximation is to ``transfer'' the error in the primal proximum to the dual proximum.
Note that, following the same line as before, by interchanging the order of iterates (starting with the primal variable $x$) we can now perform the overrelaxation in the dual variable instead of the primal in order to get a bound on $y$.

So let $\xh$ be the true primal proximum and choose $\xt = \xc$.
Then by the definition of the type-0 approximation there exists $s \in \X$ with $\|s\| \leq \sqrt{2 \tau \e}$ such that $\xc = \xh + s$, which implies that 
\begin{align*}
 \yc 
 \approx_2 \prox_{\sigma h^*} (\yb + \sigma K\xt) 
 = \prox_{\sigma h^*} (\yb + \sigma K\xc)
 = \prox_{\sigma h^*} (\yb + \sigma (K \xh + Ks)).
\end{align*}
Hence with $d = Ks$ we can rewrite \eqref{eq:pd_general_prox_0} as 
\begin{align}\label{eq:pd_general_prox_0_rewritten}
 \begin{split}
  \xh &= \prox_{\tau g} (\xb - \tau K^*\yt ), \\ 
  \xc &= \xh + s \\ 
  \yc &\approx_2^\delta \prox_{\sigma h^*} (\yb + \sigma (K\xh + d)).
 \end{split}
\end{align}
Now 
\begin{align*}
 \|d\| 
 = \|K s\| 
 = \|K(\xh - \xc) \| \leq \|K\| \sqrt{2 \tau \varepsilon} 
 = \sqrt{2 \sigma \kappa}
\end{align*}
with $\kappa := (\tau \e \|K\|^2)/\sigma$.
This reveals that a type-0 approximation of $\xh$ with precision $\e$ can essentially be interpreted as a type-3 approximation $\yc$ of $\yh$, which in sum with the type-2 approximation gives an overall approximation of type 1 for $\yc$, now however with ``mixed'' precision $\kappa$ and $\delta$.
Using the choices 
\begin{align*}
 (\xh,\yc) = (\xh^{n+1},y^{n+1}), \qquad (\xb,\yb) = (x^n,y^n), \qquad \yt = 2y^n - y^{n-1},
\end{align*}
we formally obtain the following algorithm:
\begin{align}\label{eq:alg_prox_3}
 \begin{split}
  \xh^{n+1} &= \prox_{\tau g} (x^n - \tau K^*(2y^n - y^{n-1})), \\ 
  \yc^{n+1} &\approx_1^{\delta_{n+1},\kappa_{n+1}} \prox_{\sigma h^*} (y^n + \sigma K \xh^{n+1}). 
 \end{split}
\end{align}
This situation can then be treated similarly to the above analysis (cf. Theorem \ref{thm:convergence_no_acc}) and is summarized in Corollary \ref{cor:inex_pd_basic_y}.
The main difference here is now that we get an estimate on the true proximum $\xh^{n+1}$ while computing $x^{n+1}$ in practice. 
\begin{corollary}\label{cor:inex_pd_basic_y}
 Let $L = \|K\|$ and choose $\beta > 0$ and $\tau,\sigma > 0$ such that $\sigma \tau L^2 + \sigma \beta L< 1$ and let $(\xh^n, \yc^n)$ be defined by Algorithm \eqref{eq:alg_prox_3}. 
 Then for $\hat{X}^N := \left( \sum_{n=1}^N \xh^n \right)/N$ and $Y^N := \left( \sum_{n=1}^N y^n \right)/N$
 we have for any saddle point $(\xs, \ys) \in \X \times \Y$ that
 \begin{align*}
  \Lcal (\hat{X}^N,\ys) - \Lcal (\xs, Y^N) \leq  \frac{1}{2\sigma N} \left( \sqrt{\frac{\sigma}{\tau}} \|\xs - x^0 \| +  \|\ys - y^0 \| + 2 A_N + \sqrt{2B_N}\right)^2, 
 \end{align*}
 with $A_N = \sum_{n=1}^N \sqrt{2 \sigma \kappa_n}$ and $B_N = \sum_{n=1}^N \sigma \delta_n$.
\end{corollary}
\begin{proof}
We can easily verify the assertion by dropping $f$ and simply interchanging the roles of $x$ and $y$ (and thus $\tau$ and $\sigma$) in Theorem \ref{thm:inex_pd_basic}.
\end{proof}
As for Theorem \ref{thm:inex_pd_basic} we can now state a rate for $(\hat{X}^N,Y^N)$ if the partial sums $A_N$ and$\sqrt{B_N}$ are in $o(\sqrt{N})$.
Since the result still relies on the unknown true proxima $\xh^n$, it then remains to note that for $\check{X}^N := (\sum_{i=1}^N \xc^n) / N$ we have 
\begin{align*}
 \| \hat{X}^N - \check{X}^N \| \leq \frac{1}{N} \sum_{i=1}^N \|\xh^n - \xc^n \| \leq \frac{1}{N} \sum_{i=1}^N  \sqrt{2 \sigma \kappa_n} = \frac{1}{N} A_N, 
\end{align*}
which implies strong convergence of $\check{X}^N$ to $\hat{X}^N$ with the same rate. 

Hence we can essentially handle the situation of a type-0 approximation by the same means as before. 
The major difference is still that none of the $\xc^n$ need to be feasible, which could impose problems in practice. 
Since type-0 approximations are the weakest among the introduced notions, they should technically impose the least restrictive error criteria. 
It however is an open question how to check $\| \xh - \xc \| \leq \sqrt{2\tau\e}$ effectively. 
It is easy to see that the duality gap bounds this quantity, in which situation Proposition \ref{prop:duality_gap_prox} ``unfortunately'' states that $\xc$ is already a stronger type-2 approximation.
Hence it remains to find a different criterion for the precision of a type-0 approximation to make this approach feasible in practice.

\subsection{Proof of Theorem \ref{thm:inex_pd_acc2}}\label{subsec:inex_pd_acc2}
\begin{proof}
 Using Lemma \ref{lem:general_inequality}, we proceed exactly as in the proof of Theorem \ref{thm:inex_pd_basic} (now only including the $\gamma$-strong convexity of $g$ as well as $\tau = \tau_n, \sigma = \sigma_n$ and introducing $\theta_n$), to arrive at the basic inequality   
 \begin{align*}
  &\Lcal (x^{n+1},y) - \Lcal(x,y^{n+1}) \leq  \Delta_n(x,y) - \frac{1 + \gamma \tau_n}{2\tau_n} \| x - x^{n+1} \|^2 - \frac{\|y-y^{n+1}\|^2}{2\sigma_n} \\
  &\hspace{-4pt}+ \langle K(x^{n+1} - x^n),y^{n+1} - y \rangle - \theta_n \langle K(x^n - x^{n-1}), y^{n+1} - y \rangle - \frac{1-\tau_n L_f}{2\tau_n} \|x^n - x^{n+1} \|^2 \\
  & \qquad \quad - \frac{\|y^n - y^{n+1}\|^2}{2\sigma_n} + \left( \|e^{n+1}\| + \sqrt{(2 \varepsilon_{n+1})/\tau_n} \right) \|x - x^{n+1} \| + \varepsilon_{n+1} + \delta_{n+1},
 \end{align*}
 where we let $\Delta_n(x,y) := \| x - x^n \|^2/(2\tau_n) + \| y - y^n \|^2/(2\sigma_n)$ for the sake of clarity.
 The goal of the proof is, again, to manipulate this inequality such that we obtain a recursion where most of the terms cancel when summing the inequality.
 In order to get a useful recursion in the first line it is clear that we require
 \begin{align}
  \sigma_n = \theta_{n+1} \sigma_{n+1}, \qquad (1 + \gamma \tau_n) \tau_{n+1} \theta_{n+1} \geq \tau_n \label{eq:cond_sigma_n2},
 \end{align}
 such that we obtain the estimate
 \begin{align*}
  - &\frac{1+\gamma \tau_n}{2\tau_n} \| x - x^{n+1} \|^2 - \frac{\|y-y^{n+1}\|^2}{2\sigma_n} \\ 
 &\qquad \quad = - \frac{(1+ \gamma \tau_n)\tau_{n+1}}{\tau_n} \frac{\|x-x^{n+1} \|^2}{2\tau_{n+1}} - \frac{\sigma_{n+1}}{\sigma_n} \frac{\|y - y^{n+1} \|^2 }{2 \sigma_{n+1}}
  \leq - \frac{1}{\theta_{n+1}} \Delta_{n+1} (x,y).
 \end{align*}
 For a useful recursion for the second line we expand 
 \begin{align*}
  &- \theta_n \langle K (x^n - x^{n-1}),y^{n+1} - y \rangle \\ 
  = &-\theta_n \langle K(x^n - x^{n-1}), y^{n+1} - y^n \rangle - \theta_n \langle K(x^n - x^{n-1}), y^n - y \rangle, 
 \end{align*}
 and compute (cf. Equation \eqref{eq:Young_inequality} with now $\alpha = \sigma_n \theta_n L$)
 \begin{align*}
  - \theta_n \langle K(x^n - x^{n-1}), y^{n+1} - y^n \rangle 
  &\leq  \frac{\sigma_{n-1} \theta_n L^2}{2} \| x^{n-1} - x^n \|^2 + \frac{\|y^{n+1} - y^n \|^2}{2 \sigma_n},
 \end{align*}
 where we used \eqref{eq:cond_sigma_n2} such that $\sigma_n \theta_n = \sigma_{n-1}$.
 We note that since $\tau_n L_f + \tau_n \sigma_n L^2 \leq 1 $ we furthermore have  
 \begin{align*}
  - \frac{1- \tau_n L_f}{2\tau_n} \| x^n - x^{n+1} \|^2 
  \leq - \frac{\sigma_n L^2}{2} \| x^n - x^{n+1} \|^2.
 \end{align*}
 
 Putting everything together and rearranging we arrive at (note that the terms $\|y^{n+1} - y^n \|^2 /(2 \sigma_n)$ cancel and $\sigma_{n+1}/\sigma_n = 1/\theta_{n+1}$)
 \begin{align*}
  \Lcal &(x^{n+1},y) - \Lcal(x,y^{n+1}) \\ 
  &\leq \Delta_n(x,y) - \theta_n \langle K(x^n - x^{n-1}),y^n -y \rangle + \frac{\sigma_{n-1} \theta_n L^2}{2} \| x^{n-1} - x^n \|^2 \\
  &- \frac{\sigma_{n+1}}{\sigma_n} \left( \Delta_{n+1}(x,y) - \theta_{n+1} \langle K(x^{n+1} - x^n, y^{n+1} - y \rangle + \frac{\sigma_n \theta_{n+1} L^2}{2} \| x^{n+1} - x^n \|^2 \right) \\
  &+ \left( \|e^{n+1}\| + \sqrt{(2 \varepsilon_{n+1})/\tau_n} \right) \|x - x^{n+1} \| + \varepsilon_{n+1} + \delta_{n+1}. 
 \end{align*}
 We multiply the inequality by $\sigma_n / \sigma_0$ to reveal the recursion 
 and sum from $n=0, \dots, N-1$: 
 \begin{align*}
  \sum_{n=1}^N &\frac{\sigma_{n-1}}{\sigma_0} (\Lcal (x^n,y) - \Lcal(x,y^n) ) \leq \Delta_0(x,y) \\ 
  &- \frac{\sigma_N}{\sigma_0} \left( \Delta_N(x,y) - \theta_N \langle K(x^N - x^{N-1}),y^N - y \rangle + \frac{\sigma_{N-1} \theta_N L^2}{2} \|x^N - x^{N-1} \|^2 \right) \\
  & + \frac{1}{\sigma_0} \sum_{n=1}^N \left( \sigma_{n-1} \|e^n\| + \sqrt{(2 \sigma_{n-1}^2 \varepsilon_n)\tau_{n-1}} \right) \|x - x^n\| + \frac{1}{\sigma_0} \sum_{n=1}^N \sigma_{n-1} (\varepsilon_n + \delta_n).
 \end{align*}
 Now, as above, we use that 
 \begin{align*}
  \theta_N \langle K(x^N - x^{N-1}, y^N - y \rangle 
  &\leq \frac{\sigma_{N-1} \theta_N L^2}{2} \| x^N - x^{N-1} \|^2 +  \frac{\|y - y^N \|^2}{2 \sigma_N},  
 \end{align*}
 which gives 
 \begin{align}\label{eq:intermediate_res2}
  \sum_{n=1}^N &\frac{\sigma_{n-1}}{\sigma_0} (\Lcal (x^n,y) - \Lcal(x,y^n) ) + \frac{\sigma_N}{\sigma_0} \frac{1}{2\tau_N} \|x - x^N \|^2 \leq \Delta_0(x,y)  \nonumber \\
  &+ \frac{1}{\sigma_0} \sum_{n=1}^N \left( \sigma_{n-1} \|e^n\| + \sqrt{(2 \sigma_{n-1}^2 \varepsilon_n)/\tau_{n-1}} \right) \|x - x^n\| + \frac{1}{\sigma_0} \sum_{n=1}^N \sigma_{n-1} (\varepsilon_n + \delta_n). 
 \end{align}
 This equation can now be use as before to bound all terms on the left hand side.
 Again for a saddle point $(\xs,\ys) \in \X \times \Y$ the sum is nonnegative, hence we obtain the inequality:
 \begin{align*}
  \| &\xs - x^N \|^2 
  \leq \frac{\tau_N}{\sigma_N} \frac{\sigma_0}{\tau_0} \| \xs - x^0 \|^2 + \frac{\tau_N}{\sigma_N} \| \ys - y^0 \|^2 \\
  &+ 2 \frac{\tau_N}{\sigma_N}  \sum_{n=1}^N \left( \sigma_{n-1} \|e^n\| + \sqrt{(2 \sigma_{n-1}^2 \varepsilon_n)/\tau_{n-1}} \right) \|x - x^n\| + 2 \frac{\tau_N}{\sigma_N}  \sum_{n=1}^N \sigma_{n-1} (\varepsilon_n + \delta_n).
 \end{align*}
 For the sake of readability let us denote 
 \begin{align*}
  \eta_N = \frac{\tau_N}{\sigma_N}, ~ A_N = \sum_{n=1}^N \left( \sigma_{n-1} \|e^n\| + \sqrt{(2 \sigma_{n-1}^2 \varepsilon_n)/\tau_{n-1}} \right), ~ B_N = \sum_{n=1}^N \sigma_{n-1} (\varepsilon_n + \delta_n).
 \end{align*}
 Then as before with Lemma \ref{lem:technical_result} we find, 
 \begin{align*}
  \|\xs - x^N \| 
  \leq \eta_N A_N + \left( \eta_N \frac{\sigma_0}{\tau_0} \| \xs - x^0 \|^2 + \eta_N \| \ys - y^0 \|^2 + 2 \eta_N B_N + \eta_N^2 A_N^2 \right)^{\frac{1}{2}}.
 \end{align*}
 Since $A_N$ and $B_N$ are increasing we have for all $n \leq N$
 \begin{align*}
  \|\xs - x^n \| 
  &\leq \eta_n A_n + \left( \eta_n \frac{\sigma_0}{\tau_0} \| \xs - x^0 \|^2 + \eta_n \| \ys - y^0 \|^2 + 2 \eta_n B_n + \eta_n^2 A_n^2 \right)^{\frac{1}{2}} \\
  &\leq \eta_N A_N + \left( \eta_N \frac{\sigma_0}{\tau_0} \| \xs - x^0 \|^2 + \eta_N \| \ys - y^0 \|^2 + 2 \eta_N B_N + \eta_N^2 A_N^2 \right)^{\frac{1}{2}} \\
  &\leq 2 \eta_N A_N + \sqrt{\eta_N} \sqrt{\frac{\sigma_0}{\tau_0}} \| \xs - x^0 \| + \sqrt{\eta_N} \| \ys - y^0 \| + \sqrt{\eta_N} \sqrt{2 B_N}.
 \end{align*}
 Now evoking equation \eqref{eq:intermediate_res2} we obtain
 \begin{align*}
  \sum_{n=1}^N &\frac{\sigma_{n-1}}{\sigma_0} (\Lcal (x^n,y) - \Lcal(x,y^n) ) \\
  &\leq \frac{1}{2\tau_0} \| \xs - x^0 \|^2 + \frac{1}{2\sigma_0} \| \ys - y^0 \|^2 + \frac{1}{\sigma_0} B_N  \\ 
  &\hspace{0.6cm}  + \frac{1}{\sigma_0}A_N \left( 2 \eta_N A_N + \sqrt{\eta_N} \sqrt{\frac{\sigma_0}{\tau_0}} \| \xs - x^0 \| + \sqrt{\eta_N} \| \ys - y^0 \| + \sqrt{\eta_N} \sqrt{2 B_N} \right) \\
  &\leq \frac{1}{2\sigma_0} \left( \frac{\sigma_0}{\tau_0} \| \xs - x^0 \|^2 + \| \ys - y^0 \|^2 + B_N + 4 \eta_N A_N^2 \right. \\
  & \qquad \qquad + \left.2 \sqrt{\eta_N} A_N \sqrt{\frac{\sigma_0}{\tau_0}} \| \xs - x^0 \| + 2 A_N \sqrt{\eta_N} \| \ys - y^0 \| + 2 A_N \sqrt{\eta_N} \sqrt{2 B_N} \right) \\ 
  &\leq \frac{1}{2\sigma_0} \left( \sqrt{\frac{\sigma_0}{\tau_0}} \| \xs - x^0 \| + \| \ys - y^0 \| + 2 \sqrt{\eta_N} A_N + \sqrt{2 B_N} \right)^2
 \end{align*}
 The convexity of $(\xi, \zeta) \mapsto \Lcal(\xi, \ys) - \Lcal(\xs, \zeta)$ and the definition of the ergodic averages yields the assertion (cf. the proof of Theorem \ref{thm:inex_pd_acc}).
 The estimate on $\|\xs - x^N \|^2$ follows analogously.
 It remains to note that for a type-2 approximation the square root in $A_N$ can be dropped and for a type-3 approximation $B_N =0$, which gives the different $A_{N,i}, B_{N,i}$.
\end{proof}

\subsection{Proof of Theorem \ref{thm:inex_pd_acc}}
\label{subsec:inex_pd_acc}
\begin{proof}
 We proceed exactly as in the proof of Theorem \ref{thm:inex_pd_acc2} with interchanged roles of $x,y,\tau_n$ and $\sigma_n$ to arrive at the basic inequality   
 \begin{align*}
  \Lcal &(x^{n+1},y) - \Lcal(x,y^{n+1}) \leq \Delta_n(x,y) - \frac{1}{2\tau_n} \| x - x^{n+1} \|^2 - \frac{1 + \gamma \sigma_n}{2\sigma_n} \|y-y^{n+1}\|^2 \\
  &+ \langle K(x^{n+1} - x^n),y^{n+1} - y \rangle - \theta_n \langle K(x^n - x^{n-1}), y^{n+1} - y \rangle - \frac{1}{2\tau_n} \|x^n - x^{n+1} \|^2  \\
  &- \frac{1}{2\sigma_n} \|y^n - y^{n+1}\|^2 + \left( \|e^{n+1}\| + \sqrt{\frac{2 \varepsilon_{n+1}}{\tau_n}} \right) \|x - x^{n+1} \| + \e_{n+1} + \delta_{n+1},
 \end{align*}
 where we again let $\Delta_n(x,y) := \| x - x^n \|^2/(2\tau_n) + \| y - y^n \|^2/(2\sigma_n)$.
 In order to get a useful recursion for the first two lines it is clear that we need to require 
 \begin{align*}
  \tau_n = \theta_{n+1} \tau_{n+1}, \\ 
  (1 + \gamma \sigma_n) \sigma_{n+1} \theta_{n+1} \geq \sigma_n,
 \end{align*}
 such that the second line becomes
 \begin{align*}
  & - \frac{1}{2\tau_n} \| x - x^{n+1} \|^2 - \frac{1 + \gamma \sigma_n}{2\sigma_n} \|y-y^{n+1}\|^2 \\
  = &- \frac{\tau_{n+1}}{\tau_n} \frac{\|x-x^{n+1} \|^2}{2\tau_{n+1}} - \frac{(1+\gamma \sigma_n) \sigma_{n+1}}{\sigma_n} \frac{\|y - y^{n+1} \|^2 }{2 \sigma_{n+1}} \\
  \leq &- \frac{1}{\theta_{n+1}} \Delta_{n+1} (x,y).
 \end{align*}
 For a useful recursion for the third line we expand 
 \begin{align*}
  &- \theta_n \langle K (x^n - x^{n-1}),y^{n+1} - y \rangle \\
  = &-\theta_n \langle K(x^n - x^{n-1}), y^{n+1} - y^n \rangle - \theta_n \langle K(x^n - x^{n-1}), y^n - y \rangle, 
 \end{align*}
 and compute with Young's inequality (cf. Equation \eqref{eq:Young_inequality} with $\alpha = 1 / (\tau_n \theta_n L)$)
 \begin{align*}
  - \theta_n \langle K(x^n - x^{n-1}), y^{n+1} - y^n \rangle 
  \leq \frac{1}{2\tau_n} \| x^{n-1} - x^n \|^2 + ( \tau_n \sigma_n \theta_n^2 L^2 ) \frac{\|y^{n+1} - y^n \|^2}{2 \sigma_n}.
 \end{align*}
 In order to obtain a recursion for the first term on the right hand side 
 we note that 
 \begin{align*}
  - \frac{1}{2\tau_n} \| x^n - x^{n+1} \|^2 
  = - \frac{\tau_{n+1}}{\tau_n} \frac{1}{2\tau_{n+1}} \| x^n - x^{n+1} \|^2
 \end{align*}
 in the fourth line.
 Putting everything together and rearranging we arrive at 
 \begin{align*}
  \Lcal &(x^{n+1},y) - \Lcal(x,y^{n+1}) 
  \leq \Delta_n(x,y) - \theta_n \langle K(x^n - x^{n-1}),y^n -y \rangle + \frac{1}{2\tau_n} \| x^n - x^{n-1} \|^2 \\
  &- \frac{\tau_{n+1}}{\tau_n} \left( \Delta_{n+1}(x,y) - \theta_{n+1} \langle K(x^{n+1} - x^n), y^{n+1} - y \rangle + \frac{1}{2 \tau_{n+1}} \| x^{n+1} - x^n \|^2 \right) \\
  &- (1 - \tau_n \sigma_n \theta_n^2 L^2) \frac{\|y^{n+1} - y^n \|^2}{2\sigma_n} + \sqrt{\frac{2 \varepsilon_{n+1}}{\tau_n}} \|x - x^{n+1} \| + \e_{n+1} + \delta_{n+1}. 
 \end{align*}
 Requiring that $\tau_n \sigma_n \theta_n^2 L^2 \leq 1$ we can discard the related term and multiply the inequality by $\tau_n / \tau_0$ to reveal the recursion:
  \begin{align*}
  \frac{\tau_n}{\tau_0} \Lcal (&x^{n+1},y) - \Lcal(x,y^{n+1}) \\
  \leq &\frac{\tau_n}{\tau_0} \left(  \Delta_n(x,y) - \theta_n \langle K(x^n - x^{n-1}),y^n -y \rangle + \frac{1}{2\tau_n} \| x^n - x^{n-1} \|^2 \right) \\
  - &\frac{\tau_{n+1}}{\tau_0} \left( \Delta_{n+1}(x,y) - \theta_{n+1} \langle K(x^{n+1} - x^n), y^{n+1} - y \rangle + \frac{1}{2 \tau_{n+1}} \| x^{n+1} - x^n \|^2 \right) \\
  + &\frac{1}{\tau_0} \sqrt{2 \tau_n \varepsilon_{n+1}} \|x - x^{n+1} \| + \frac{\tau_n}{\tau_0} (\e_{n+1} + \delta_{n+1}). 
 \end{align*}
 We now sum the above inequality from $n=0, \dots, N-1$: 
 \begin{align*}
  &\sum_{n=1}^N \frac{\tau_{n-1}}{\tau_0} (\Lcal (x^n,y) - \Lcal(x,y^n) ) \\
  &\leq \Delta_0(x,y) - \frac{\tau_N}{\tau_0} \left( \Delta_N(x,y) - \theta_N \langle K(x^N - x^{N-1}),y^N - y \rangle + \frac{1}{2\tau_N} \|x^N - x^{N-1} \|^2 \right) \\
  &+ \frac{1}{\tau_0} \sum_{n=1}^N \sqrt{2 \tau_{n-1} \varepsilon_n} \|x - x^n\| + \frac{1}{\tau_0} \sum_{n=1}^N \tau_{n-1} (\e_n + \delta_n).
 \end{align*}
 Now, as above, we use that 
 \begin{align*}
  \theta_N \langle K(x^N - x^{N-1}, y^N - y \rangle 
  \leq ~\frac{1}{2\tau_N} \| x^N - x^{N-1} \|^2 + (\sigma_N \tau_N \theta_N^2 L^2) \frac{\|y - y^N \|^2}{2 \sigma_N},  
 \end{align*}
 which gives the first intermediate result: 
 \begin{align}\label{eq:intermediate_res}
  \sum_{n=1}^N \frac{\tau_{n-1}}{\tau_0} (\Lcal (x^n,y) &- \Lcal(x,y^n) ) + \frac{1}{2\tau_0} \|x - x^N \|^2 + \frac{\tau_N}{\tau_0} (1- \sigma_N \tau_N \theta_N^2 L^2) \frac{\|y - y^N \|^2}{2\sigma_N} \nonumber \\
  &\leq \Delta_0(x,y) + \frac{1}{\tau_0} \sum_{n=1}^N \sqrt{2 \tau_{n-1} \varepsilon_n} \|x - x^n\| + \frac{1}{\tau_0} \sum_{n=1}^N \tau_{n-1} (\e_n + \delta_n). 
 \end{align}
 This equation can now be use as before to bound all terms on the left hand side and hence gives the necessary bound on $\|x - x^N\|$ appearing in the error term.
 For a saddle point $(\xs,\ys) \in \X \times \Y$ the sum on the left hand side is nonnegative and:
 \begin{align*}
  \| \xs - x^N \|^2 \leq 2\tau_0 \Delta_0(x,y) + 2\sum_{n=1}^N  \sqrt{2 \tau_{n-1} \varepsilon_n} \|\xs - x^n\| + 2 \sum_{n=1}^N \tau_{n-1} (\e_n + \delta_n).
 \end{align*}
 Hence, again with Lemma \ref{lem:technical_result}, 
 \begin{align*}
  \|\xs - x^N \| &\leq \sum_{n=1}^N \sqrt{2 \tau_{n-1} \varepsilon_n} + \left( \| \xs - x^0 \|^2 + \frac{\tau_0}{\sigma_0} \|\ys-y^0 \|^2 \right. \\ 
  &+ \left.2 \sum_{n=1}^N \tau_{n-1} ( \varepsilon_n + \delta_n )+ \left( \sum_{n=1}^N \sqrt{2 \tau_{n-1} \varepsilon_n}\right)^2 \right)^{\frac{1}{2}} \\
  &= A_N + \left( \| \xs - x^0 \|^2 + \frac{\tau_0}{\sigma_0} \| \ys - y^0 \|^2 + 2 B_N + A_N^2 \right)^{\frac{1}{2}}, 
 \end{align*}
 where we denote $A_N = \sum_{n=1}^N \sqrt{2 \tau_{n-1} \varepsilon_n}$ and $B_N = \sum_{n=1}^N \tau_{n-1} (\e_n + \delta_n)$.
 %
 Since $A_N$ and $B_N$ are increasing we have for all $n \leq N$
 \begin{align*}
  \|\xs - x^n \| &\leq A_n + \left( \| \xs - x^0 \|^2 + \frac{\tau_0}{\sigma_0} \| \ys - y^0 \|^2 + 2 B_n + A_n^2 \right)^{\frac{1}{2}} \\ 
  &\leq A_N + \left( \| \xs - x^0 \|^2 + \frac{\tau_0}{\sigma_0} \| \ys - y^0 \|^2 + 2 B_N + A_N^2 \right)^{\frac{1}{2}} \\
  &\leq 2 A_N + \| \xs - x^0 \| + \sqrt{\frac{\tau_0}{\sigma_0}} \| \ys - y^0 \| + \sqrt{2 B_N}.
 \end{align*}
 Then we find (again by equation \eqref{eq:intermediate_res})
 \begin{align*}
  &\sum_{n=1}^N \frac{\tau_{n-1}}{\tau_0} (\Lcal (x^n,\ys) - \Lcal(\xs,y^n) ) \\
  &\leq \Delta_0(\xs,\ys) + \frac{1}{\tau_0}B_N + \frac{1}{\tau_0} A_N \left( 2 A_N + \| \xs - x^0 \| + \frac{\tau_0}{\sigma_0} \| \ys - y^0 \| + \sqrt{2 B_N} \right) \\
  &= \frac{1}{2\tau_0} \left( \|\xs - x^0 \|^2 + \frac{\tau_0}{\sigma_0} \| \ys - y^0 \|^2 + 2B_N + 4 A_N^2 \right. \\ 
  &\left. \hspace{3.3cm} + 2A_N \|\xs - x^0 \| + 2 A_N \sqrt{\frac{\tau_0}{\sigma_0}} \| \ys - y^0 \| + 2 A_N \sqrt{2 B_N)}  \right)\\ 
  &\leq \frac{1}{2\tau_0} \left( \|\xs - x^0 \| + \sqrt{\frac{\tau_0}{\sigma_0}} \| \ys - y^0 \| + 2A_N + \sqrt{2 B_N} \right)^2.
 \end{align*}
 Using the convexity of $(\xi, \zeta) \mapsto \Lcal(\xi, \ys) - \Lcal(\xs, \zeta)$ and Jensen's inequality as well as the definition of the ergodic averages $(X^N,Y^N)$ yields 
 the first assertion.
 The estimate on $\|\xs - x^N\|^2$ and $\| \ys - y^N \|^2$ then follows analogously from inequality \eqref{eq:intermediate_res}.
 It remains to note that for a type-2 approximation the square root in $A_N$ can be dropped and for a type-3 approximation $B_N =0$, which gives the different $A_{N,i}, B_{N,i}$.
\end{proof}

\subsection{Proof of Theorem \ref{thm:linear_case}}
\label{subsec:smooth}
\begin{proof}
We again start with the general descent rule in Lemma \ref{lem:general_inequality}:
\begin{align*}
  \Lcal &(x^{n+1},y) - \Lcal(x,y^{n+1}) \leq  \frac{1}{2\tau} \| x - x^n \|^2 + \frac{1}{2\sigma} \|y-y^n\|^2 - \frac{1+\gamma \tau}{2\tau} \| x - x^{n+1} \|^2 \\
  & - \frac{1 + \mu \sigma}{2\sigma} \|y-y^{n+1}\|^2 + \langle K(x^{n+1} - x^n),y^{n+1} - y \rangle - \theta \langle K(x^n - x^{n-1}), y^{n+1} - y \rangle \\
  &\qquad - \frac{1 - \tau L_f}{2\tau} \|x^n - x^{n+1} \|^2 - \frac{1}{2\sigma} \|y^n - y^{n+1}\|^2 \\
  &\qquad \qquad  + \left( \|e^{n+1}\| + \sqrt{(2 \varepsilon_{n+1})/\tau} \right) \|x - x^{n+1} \| + \varepsilon_{n+1} + \delta_{n+1}.
 \end{align*}
Now we expand and apply Young's inequality 
\begin{align*}
 &- \theta \langle K(x^n - x^{n-1}),y^{n+1} -y \rangle \\
 = &- \theta \langle K(x^n - x^{n-1}),y^{n+1} - y^n \rangle - \theta \langle K(x^n - x^{n-1}),y^n -y \rangle \\
 \leq &\frac{\sigma \theta^2 L^2}{2} \| x^{n-1} - x^n \|^2 + \frac{\|y^{n+1} - y^n\|^2}{2\sigma} - \theta \langle K(x^n - x^{n-1}),y^n -y \rangle,
\end{align*}
which gives 
\begin{align*}
  \Lcal &(x^{n+1},y) - \Lcal(x,y^{n+1}) \leq  \frac{1}{2\tau} \| x - x^n \|^2 + \frac{1}{2\sigma} \|y-y^n\|^2 - \frac{1+\gamma \tau}{2\tau} \| x - x^{n+1} \|^2  \\
  & - \frac{1 + \mu \sigma}{2\sigma} \|y-y^{n+1}\|^2 + \langle K(x^{n+1} - x^n),y^{n+1} - y \rangle - \theta \langle K(x^n - x^{n-1}), y^n - y \rangle \\
  &\qquad - \frac{1 - \tau L_f}{2\tau} \|x^n - x^{n+1} \|^2 + \frac{\sigma \theta^2 L^2}{2} \|x^{n-1} - x^n \|^2 \\
  &\qquad \qquad + \left( \|e^{n+1}\| + \sqrt{\frac{2 \varepsilon_{n+1}}{\tau}} \right) \|x - x^{n+1} \| + \varepsilon_{n+1} + \delta_{n+1}.
 \end{align*}
Ensuring that $1+\gamma \tau = 1+\mu \sigma = 1/\theta$ and $(1-\tau L_f)/\tau \geq \sigma \theta^2 L^2$
we derive
\begin{align*}
 \Lcal(x^{n+1},y) - \Lcal(x,y^{n+1}) 
 &\leq \Delta_n(x,y) - \theta \langle K(x^n - x^{n-1}),y^n-y\rangle \\
 &-\frac{1}{\theta} \Big( \Delta_{n+1}(x,y) - \theta \langle K(x^{n+1} - x^n),y^{n+1} -y \rangle \Big) \\
 &+ \frac{\sigma \theta^2 L^2}{2} \| x^{n-1} -x^n \|^2 - \frac{1-\tau L_f}{2\tau} \| x^n - x^{n+1} \|^2 \\
 &+ \left( \|e^{n+1}\| + \sqrt{\frac{2 \varepsilon_{n+1}}{\tau}} \right) \|x - x^{n+1} \| + \varepsilon_{n+1} + \delta_{n+1}.
\end{align*}
We now multiply by $\theta^{-n}$ and sum from $n = 0, \dots, N-1$:
\begin{align*}
 \sum_{n=1}^N \frac{1}{\theta^{n-1}} &\Big( \Lcal(x^n,y) - \Lcal(x,y^n) \Big) \\
 \leq &\Delta_0(x,y) - \frac{1}{\theta^N} \Big( \Delta_N(x,y) - \theta \langle K(x^N - x^{N-1}),y^N - y\rangle \Big) \\
 &\quad- \sum_{n=1}^N \frac{1-\tau L_f - \tau \sigma \theta^2 L^2}{2\tau\theta^{n-1}} \|x^{n-1} - x^n \|^2 + \frac{1-\tau L_f}{2\tau \theta^{N-1}} \| x^N - x^{N-1} \|^2 \\
 &\qquad \quad+ \sum_{n=1}^N \frac{1}{\theta^{n-1}} \left[ \left( \|e^n\| + \sqrt{\frac{2 \varepsilon_n}{\tau}} \right) \|x - x^n \| + \varepsilon_n + \delta_n \right], 
\end{align*}
which again by Young's inequality implies 
\begin{align}\label{eq:main_ineq_linear}
 \sum_{n=1}^N \frac{1}{\theta^{n-1}} &\Big( \Lcal(x^n,y) - \Lcal(x,y^n) \Big) + \frac{1}{2\tau \theta^N} \|x-x^N\|^2 \nonumber \\ 
 &\leq \Delta_0(x,y) + \sum_{n=1}^N \frac{1}{\theta^{n-1}} \left[ \left( \|e^{n}\| + \sqrt{\frac{2 \varepsilon_{n}}{\tau}} \right) \|x - x^n \| + \varepsilon_{n} + \delta_{n} \right].
\end{align}
For a saddle point $(\xs,\ys)$, the sum on the left hand side is positive, hence we obtain
\begin{align*}
 \|x-x^N\|^2 &\leq \theta^N \|\xs-x^0\|^2 + \theta^N \frac{\tau}{\sigma} \|\ys-y^0\|^2 \\ 
 &+ 2 \theta^N \sum_{n=1}^N \frac{1}{\theta^{n-1}}(\tau \|e^n\| + \sqrt{2\tau \e_n}) \|\xs - x^n \| + 2 \theta^N \sum_{n=1}^N \frac{\tau}{\theta^{n-1}} (\e_n + \delta_n).
\end{align*}
Evoking Lemma \ref{lem:technical_result} and denoting 
\begin{align*}
 A_N := \sum_{n=1}^N \frac{1}{\theta^{n-1}} (\tau \|e^n\| + \sqrt{2\tau \e_n}), \qquad B_N := \sum_{n=1}^N \frac{\tau}{\theta^{n-1}} (\e_n + \delta_n), 
\end{align*}
we obtain 
\begin{align*}
 \|\xs - x^N \| \leq \theta^N A_N + \left[ \theta^N \|\xs - x^0 \|^2 + \theta^N \frac{\tau}{\sigma} \| \ys - y^0 \|^2 + 2 \theta^N B_N + \theta^{2N} A_N^2 \right]^{\frac{1}{2}}.
\end{align*}
By monotonicity we have the same bound for all $n \leq N$: 
\begin{align*}
 \|\xs - x^n \| &\leq \theta^N A_N + \left[ \theta^N \|\xs - x^0 \|^2 + \theta^N \frac{\tau}{\sigma} \| \ys - y^0 \|^2 + 2 \theta^N B_N + \theta^{2N} A_N^2 \right]^{\frac{1}{2}} \\
 &\leq 2 \theta^N A_N + \theta^{\frac{N}{2}} \|\xs-x^0\| + \theta^{\frac{N}{2}} \sqrt{\frac{\tau}{\sigma}} \|\ys - y^0 \| + \theta^{\frac{N}{2}} \sqrt{B_N}.
\end{align*}
We now again use inequality \eqref{eq:main_ineq_linear} to obtain a bound for the sum: 
\begin{align*}
 \sum_{n=1}^N \frac{1}{\theta^{n-1}} &( \Lcal(x^n,\ys) - \Lcal(\xs,y^n) ) 
 \leq \frac{1}{2\tau} \left[ \|\xs-x^0\|^2 + \frac{\tau}{\sigma} \|\ys -y^0\|^2 + 2 B_N \right. \\
 &\left. + 2 A_N \left( 2 \theta^N A_N + \theta^{\frac{N}{2}} \|\xs-x^0\| + \theta^{\frac{N}{2}} \sqrt{\frac{\tau}{\sigma}} \|\ys-y^0\| + \theta^{\frac{N}{2}}\sqrt{2 B_N} \right) \right] \\  
 & \leq \frac{1}{2\tau} \left( \|\xs -x^0\| + \sqrt{\frac{\tau}{\sigma}} \|\ys-y^0\| + 2 \theta^{\frac{N}{2}} A_N + \sqrt{2 B_N}  \right)^2.
\end{align*}
Eventually we let 
\begin{align*}
 T_N := \sum_{n=1}^N \frac{1}{\theta^{n-1}} = \frac{\theta^N-1}{\theta-1} \frac{1}{\theta^{N-1}}, \quad X^N := \frac{1}{T_N} \sum_{n=1}^N \frac{1}{\theta^{n-1}} x^n, \quad Y^N := \frac{1}{T_N} \sum_{n=1}^N \frac{1}{\theta^{n-1}} y^n
\end{align*}
to deduce the assertion by convexity and Jensen's inequality.
By the same argumentation as above we can also use inequality \eqref{eq:main_ineq_linear} to obtain the convergence of the iterates: 
\begin{align*}
 \frac{\|\xs-x^0\|^2}{2\tau} \leq  \frac{\theta^N}{2\tau} \left( \|\xs -x^0\| + \sqrt{\frac{\tau}{\sigma}} \|\ys-y^0\| + 2 \theta^{\frac{N}{2}} A_N + \sqrt{2 B_N}  \right)^2
\end{align*}

\end{proof}

%
%
\bibliographystyle{plain}
\bibliography{ref_inex_pd}{}

\begin{thebibliography}{10}

\bibitem{Alberti2003}
Giovanni Alberti, Guy Bouchitt{\'e}, and Gianni Dal~Maso.
\newblock {The calibration method for the Mumford-Shah functional and
  free-discontinuity problems}.
\newblock {\em Calculus of Variations and Partial Differential Equations},
  16(3):299--333, 2003.

\bibitem{Alliney1992}
S.~Alliney.
\newblock Digital filters as absolute norm regularizers.
\newblock {\em IEEE Transactions on Signal Processing}, 40(6):1548--1562, 1992.

\bibitem{Alliney1996}
S.~Alliney.
\newblock Recursive median filters of increasing order: a variational approach.
\newblock {\em IEEE Transactions on Signal Processing}, 44(6):1346--1354, 1996.

\bibitem{Alliney1997}
S.~Alliney.
\newblock A property of the minimum vectors of a regularizing functional
  defined by means of the absolute norm.
\newblock {\em IEEE Transactions on Signal Processing}, 45(4):913--917, 1997.

\bibitem{Aujol2015}
J.-F. Aujol and Ch. Dossal.
\newblock {Stability of over-relaxations for the forward-backward algorithm,
  application to FISTA}.
\newblock {\em SIAM Journal on Optimization}, 25(4):2408--2433, 2015.

\bibitem{Barbero2011}
\'{A}lvaro Barbero and Suvrit Sra.
\newblock Fast newton-type methods for total variation regularization.
\newblock In {\em Proceedings of the 28th International Conference on
  International Conference on Machine Learning}, ICML'11, pages 313--320, USA,
  2011. Omnipress.

\bibitem{Beck2009b}
A.~Beck and M.~Teboulle.
\newblock Fast gradient-based algorithms for constrained total variation image
  denoising and deblurring problems.
\newblock {\em IEEE Transactions on Image Processing}, 18(11):2419--2434, Nov
  2009.

\bibitem{Beck2009}
Amir Beck and Marc Teboulle.
\newblock A fast iterative shrinkage-thresholding algorithm for linear inverse
  problems.
\newblock {\em SIAM Journal on Imaging Sciences}, 2(1):183--202, 2009.

\bibitem{Bonettini2016}
S.~Bonettini, I.~Loris, F.~Porta, and M.~Prato.
\newblock Variable metric inexact line-search-based methods for nonsmooth
  optimization.
\newblock {\em SIAM Journal on Optimization}, 26(2):891--921, 2016.

\bibitem{Boyd2011}
Stephen Boyd, Neal Parikh, Eric Chu, Borja Peleato, and Jonathan Eckstein.
\newblock Distributed optimization and statistical learning via the alternating
  direction method of multipliers.
\newblock {\em Foundations and Trends in Machine Learning}, 3(1):1--122, 2011.

\bibitem{Brancolini2018}
Alessio Brancolini, Carolin Rossmanith, and Benedikt Wirth.
\newblock {Optimal micropatterns in 2D transport networks and their relation to
  image inpainting}.
\newblock {\em Archive for Rational Mechanics and Analysis}, 228(1):279--308,
  Apr 2018.

\bibitem{Bregman1967}
L.M. Bregman.
\newblock The relaxation method of finding the common point of convex sets and
  its application to the solution of problems in convex programming.
\newblock {\em USSR Computational Mathematics and Mathematical Physics},
  7(3):200 -- 217, 1967.

\bibitem{Brinkmann2017}
Eva-Maria Brinkmann, Martin Burger, Julian Rasch, and Camille Sutour.
\newblock Bias reduction in variational regularization.
\newblock {\em Journal of Mathematical Imaging and Vision}, 59(3):534--566, Nov
  2017.

\bibitem{Cai2010}
Jian-Feng Cai, Emmanuel~J. Candès, and Zuowei Shen.
\newblock A singular value thresholding algorithm for matrix completion.
\newblock {\em SIAM Journal on Optimization}, 20(4):1956--1982, 2010.

\bibitem{Chambolle2011}
A.~Chambolle and T.~Pock.
\newblock A first-order primal-dual algorithm for convex problems with
  applications to imaging.
\newblock {\em Journal of Mathematical Imaging and Vision}, 40(1):120--145,
  2011.

\bibitem{Chambolle2017}
Antonin Chambolle, Matthias~J Ehrhardt, Peter Richtarik, and Carola-Bibiane
  Sch{\"o}nlieb.
\newblock Stochastic primal-dual hybrid gradient algorithm with arbitrary
  sampling and imaging applications.
\newblock working paper or preprint, June 2017.

\bibitem{Chambolle2016a}
Antonin Chambolle and Thomas Pock.
\newblock An introduction to continuous optimization for imaging.
\newblock {\em Acta Numerica}, 25:161–319, 2016.

\bibitem{Chambolle2016}
Antonin Chambolle and Thomas Pock.
\newblock On the ergodic convergence rates of a first-order primal--dual
  algorithm.
\newblock {\em Mathematical Programming}, 159(1-2):253--287, 2016.

\bibitem{Chan2005}
Tony~F. Chan and Selim Esedoglu.
\newblock Aspects of total variation regularized l1 function approximation.
\newblock {\em SIAM Journal on Applied Mathematics}, 65(5):1817--1837, 2005.

\bibitem{Chaux2009}
Caroline Chaux, Jean-Christophe Pesquet, and Nelly Pustelnik.
\newblock Nested iterative algorithms for convex constrained image recovery
  problems.
\newblock {\em SIAM Journal on Imaging Sciences}, 2(2):730--762, 2009.

\bibitem{Combettes2004}
Patrick~L. Combettes.
\newblock Solving monotone inclusions via compositions of nonexpansive averaged
  operators.
\newblock {\em Optimization}, 53(5-6):475--504, 2004.

\bibitem{Combettes2011}
Patrick~L. Combettes and Jean-Christophe Pesquet.
\newblock {\em Proximal Splitting Methods in Signal Processing}, pages
  185--212.
\newblock Springer New York, New York, NY, 2011.

\bibitem{Combettes2005}
Patrick~L. Combettes and Valérie~R. Wajs.
\newblock Signal recovery by proximal forward-backward splitting.
\newblock {\em Multiscale Modeling \& Simulation}, 4(4):1168--1200, 2005.

\bibitem{Cominetti1997}
R.~Cominetti.
\newblock Coupling the proximal point algorithm with approximation methods.
\newblock {\em Journal of Optimization Theory and Applications},
  95(3):581--600, 1997.

\bibitem{Condat2013}
Laurent Condat.
\newblock A primal--dual splitting method for convex optimization involving
  lipschitzian, proximable and linear composite terms.
\newblock {\em Journal of Optimization Theory and Applications},
  158(2):460--479, Aug 2013.

\bibitem{Aspremont2008}
Alexandre d'Aspremont.
\newblock Smooth optimization with approximate gradient.
\newblock {\em SIAM Journal on Optimization}, 19(3):1171--1183, 2008.

\bibitem{Devolder2014}
Olivier Devolder, Fran{\c{c}}ois Glineur, and Yurii Nesterov.
\newblock First-order methods of smooth convex optimization with inexact
  oracle.
\newblock {\em Mathematical Programming}, 146(1):37--75, Aug 2014.

\bibitem{Dong2009}
Yiqiu Dong, Michael Hintermüller, and Marrick Neri.
\newblock An efficient primal-dual method for l1-tv image restoration.
\newblock {\em SIAM Journal on Imaging Sciences}, 2(4):1168--1189, 2009.

\bibitem{Dykstra1983}
Richard~L. Dykstra.
\newblock An algorithm for restricted least squares regression.
\newblock {\em Journal of the American Statistical Association},
  78(384):837--842, 1983.

\bibitem{Ehrhardt2016}
Matthias Ehrhardt and Marta Betcke.
\newblock {Multicontrast MRI Reconstruction with structure-guided total
  variation}.
\newblock {\em SIAM Journal on Imaging Sciences}, 9(3):1084--1106, 2016.

\bibitem{Esser2010a}
Ernie Esser.
\newblock Applications of lagrangian-based alternating direction methods and
  connections to split bregman.
\newblock 2010.

\bibitem{Esser2010}
Ernie Esser, Xiaoqun Zhang, and Tony~F. Chan.
\newblock A general framework for a class of first order primal-dual algorithms
  for convex optimization in imaging science.
\newblock {\em SIAM Journal on Imaging Sciences}, 3(4):1015--1046, 2010.

\bibitem{Fadili2011}
J.~M. Fadili and G.~Peyre.
\newblock Total variation projection with first order schemes.
\newblock {\em IEEE Transactions on Image Processing}, 20(3):657--669, March
  2011.

\bibitem{Friedlander2012}
Michael~P. Friedlander and Mark Schmidt.
\newblock Hybrid deterministic-stochastic methods for data fitting.
\newblock {\em SIAM Journal on Scientific Computing}, 34(3):A1380--A1405, 2012.

\bibitem{Fu2006}
Haoying Fu, Michael~K. Ng, Mila Nikolova, and Jesse~L. Barlow.
\newblock Efficient minimization methods of mixed l2-l1 and l1-l1 norms for
  image restoration.
\newblock {\em SIAM Journal on Scientific Computing}, 27(6):1881--1902, 2006.

\bibitem{Gueler1992}
Osman G\"uler.
\newblock New proximal point algorithms for convex minimization.
\newblock {\em SIAM Journal on Optimization}, 2(4):649--664, 1992.

\bibitem{He2012}
Bingsheng He and Xiaoming Yuan.
\newblock An accelerated inexact proximal point algorithm for convex
  minimization.
\newblock {\em Journal of Optimization Theory and Applications},
  154(2):536--548, Aug 2012.

\bibitem{Karkkainen2005}
Tommi K{\"a}rkk{\"a}inen, Karl Kunisch, and Kirsi Majava.
\newblock Denoising of smooth images using l 1-fitting.
\newblock {\em Computing}, 74(4):353--376, 2005.

\bibitem{Lemaire1992}
Bernard Lemaire.
\newblock About the convergence of the proximal method.
\newblock In {\em Advances in Optimization}, pages 39--51, Berlin, Heidelberg,
  1992. Springer Berlin Heidelberg.

\bibitem{Lin2017}
Hongzhou Lin, Julien Mairal, and Zaid Harchaoui.
\newblock Catalyst acceleration for first-order convex optimization: from
  theory to practice.
\newblock arXiv preprint.

\bibitem{Lin2015}
Hongzhou Lin, Julien Mairal, and Zaid Harchaoui.
\newblock A universal catalyst for first-order optimization.
\newblock In {\em Advances in Neural Information Processing Systems}, pages
  3384--3392, 2015.

\bibitem{Lions1979}
P.~L. Lions and B.~Mercier.
\newblock Splitting algorithms for the sum of two nonlinear operators.
\newblock {\em SIAM Journal on Numerical Analysis}, 16(6):964--979, 1979.

\bibitem{Luo1993}
Zhi-Quan Luo and Paul Tseng.
\newblock Error bounds and convergence analysis of feasible descent methods: a
  general approach.
\newblock {\em Annals of Operations Research}, 46(1):157--178, Mar 1993.

\bibitem{Ma2011}
Shiqian Ma, Donald Goldfarb, and Lifeng Chen.
\newblock Fixed point and bregman iterative methods for matrix rank
  minimization.
\newblock {\em Mathematical Programming}, 128(1):321--353, Jun 2011.

\bibitem{Martinet1970}
B.~Martinet.
\newblock Brève communication. régularisation d'inéquations variationnelles
  par approximations successives.
\newblock {\em ESAIM: Mathematical Modelling and Numerical Analysis -
  Modélisation Mathématique et Analyse Numérique}, 4(R3):154--158, 1970.

\bibitem{Moreau1965}
J.J. Moreau.
\newblock Proximité et dualité dans un espace hilbertien.
\newblock {\em Bulletin de la Société Mathématique de France}, 93:273--299,
  1965.

\bibitem{Nedic2001}
Angelia Nedi{\'{c}} and Dimitri Bertsekas.
\newblock {\em Convergence Rate of Incremental Subgradient Algorithms}, pages
  223--264.
\newblock Springer US, Boston, MA, 2001.

\bibitem{Nemirovski2004}
Arkadi~S. Nemirovski.
\newblock Prox-method with rate of convergence {$O(1/t)$} for variational
  inequalities with {L}ipschitz continuous monotone operators and smooth
  convex-concave saddle point problems.
\newblock {\em SIAM J. Optim.}, 15(1):229--251 (electronic), 2004.

\bibitem{Nesterov2014}
Yurii Nesterov.
\newblock {\em Introductory Lectures on Convex Optimization: A Basic Course}.
\newblock Springer Publishing Company, Incorporated, 1 edition, 2014.

\bibitem{Nikolova2002}
Mila Nikolova.
\newblock Minimizers of cost-functions involving nonsmooth data-fidelity terms.
  application to the processing of outliers.
\newblock {\em SIAM Journal on Numerical Analysis}, 40(3):965--994, 2002.

\bibitem{Nikolova2004}
Mila Nikolova.
\newblock A variational approach to remove outliers and impulse noise.
\newblock {\em Journal of Mathematical Imaging and Vision}, 20(1):99--120,
  2004.

\bibitem{Osher2005}
Stanley Osher, Martin Burger, Donald Goldfarb, Jinjun Xu, and Wotao Yin.
\newblock An iterative regularization method for total variation-based image
  restoration.
\newblock {\em Multiscale Modeling \& Simulation}, 4(2):460--489, 2005.

\bibitem{Patriksson1997}
Michael Patriksson.
\newblock Cost approximation: A unified framework of descent algorithms for
  nonlinear programs.
\newblock 1997.

\bibitem{Pock2009}
T.~Pock, D.~Cremers, H.~Bischof, and A.~Chambolle.
\newblock An algorithm for minimizing the mumford-shah functional.
\newblock In {\em 2009 IEEE 12th International Conference on Computer Vision},
  pages 1133--1140, Sept 2009.

\bibitem{Rockafellar1976a}
R.~T. Rockafellar.
\newblock Augmented lagrangians and applications of the proximal point
  algorithm in convex programming.
\newblock {\em Mathematics of Operations Research}, 1(2):97--116, 1976.

\bibitem{Rockafellar1976}
R.~Tyrrell Rockafellar.
\newblock Monotone operators and the proximal point algorithm.
\newblock {\em SIAM Journal on Control and Optimization}, 14(5):877--898, 1976.

\bibitem{Salzo2012}
Saverio Salzo and Silvia Villa.
\newblock Inexact and accelerated proximal point algorithms.
\newblock {\em Journal of Convex Analysis}, 19(4):1167--1192, 2012.

\bibitem{Sawatzky2013}
A.~Sawatzky, C.~Brune, T.~Koesters, F.~Wübbeling, and M.~Burger.
\newblock {EM-TV methods for inverse problems with poisson noise}.
\newblock In {\em Level Set and PDE Based Reconstruction Methods in Imaging},
  volume 2090 of {\em Lecture Notes in Mathematics}, pages 1--70. Springer
  International Publishing, 2013.

\bibitem{Schmidt2011}
Mark Schmidt, Nicolas~L. Roux, and Francis~R. Bach.
\newblock Convergence rates of inexact proximal-gradient methods for convex
  optimization.
\newblock In J.~Shawe-Taylor, R.~S. Zemel, P.~L. Bartlett, F.~Pereira, and
  K.~Q. Weinberger, editors, {\em Advances in Neural Information Processing
  Systems 24}, pages 1458--1466. Curran Associates, Inc., 2011.

\bibitem{Villa2013}
Silvia Villa, Saverio Salzo, Luca Baldassarre, and Alessandro Verri.
\newblock Accelerated and inexact forward-backward algorithms.
\newblock {\em SIAM Journal on Optimization}, 23(3):1607--1633, 2013.

\bibitem{Wu2012}
Chunlin Wu, Juyong Zhang, Yuping Duan, and Xue-Cheng Tai.
\newblock Augmented lagrangian method for total variation based image
  restoration and segmentation over triangulated surfaces.
\newblock {\em Journal of Scientific Computing}, 50(1):145--166, Jan 2012.

\bibitem{Zalinescu2002}
C.~Zalinescu.
\newblock {\em Convex Analysis in General Vector Spaces}.
\newblock World Scientific, 2002.

\bibitem{Zaslavski2010}
Alexander~J. Zaslavski.
\newblock Convergence of a proximal point method in the presence of
  computational errors in hilbert spaces.
\newblock {\em SIAM Journal on Optimization}, 20(5):2413--2421, 2010.

\end{thebibliography}

\end{document}